\documentclass{amsart}
\usepackage{hyperref}
\usepackage{amsthm}
\usepackage{amsmath}
\usepackage{amssymb}
\usepackage{xfrac}
\usepackage{mathtools}
\usepackage{tikz-cd}
\usepackage[autostyle]{csquotes}

\newtheorem{prop}{Proposition}

\newtheorem*{definition*}{Definition}
\newtheorem{theorem}[prop]{Theorem}
\newtheorem{thm}[prop]{Theorem}
\newtheorem{cor}[prop]{Corollary}
\newtheorem{lemma}[prop]{Lemma}
\newtheorem{remark}[prop]{Remark}

\newtheorem{example}[prop]{Example}
\numberwithin{prop}{section}

\newcommand{\R}{\ensuremath{\mathcal{R}}}
\newcommand{\RS}{\ensuremath{\mathcal{S}}}
\newcommand{\I}{\ensuremath{\mathcal{I}}}
\newcommand{\K}{\ensuremath{\mathcal{K}}}
\newcommand{\J}{\ensuremath{\mathcal{J}}}
\newcommand{\LL}{\ensuremath{\mathcal{L}}}

\DeclareMathOperator*{\lcm}{lcm}

\newcommand{\NN}{\ensuremath{\mathbb{N}}}
\newcommand{\ZZ}{\ensuremath{\mathbb{Z}}}

\newcommand{\CC}{\ensuremath{\mathbb{C}}}

\newcommand{\ann}[1]{\ensuremath{#1^{0}}}

\newcommand{\sr}[1]{\textrm{sr}\left(\ensuremath{#1}\right)}

\newcommand{\units}[1]{\mathcal{O}_{#1}}

\newcommand{\unitsord}[2]{\mathcal{O}_{#1}^{\left[#2\right]}}

\newcommand{\mfi}[1]{\ensuremath{#1^*}}

\newcommand{\Oh}[1]{\ensuremath{ \mathcal{O} \left( #1 \right) }}

\newcommand{\SL}[2]{\mathrm{SL}_{#1}\left(#2\right)}
\newcommand{\Fn}[2]{\mathrm{F}_{#1}\left(#2\right)}
\newcommand{\Mn}[2]{\mathrm{M}_{#1}\left(#2\right)}
\newcommand{\Mnzt}[2]{\mathrm{M}^{\mathrm{zt}}_{#1}\left(#2\right)}

\newcommand{\Off}[2]{\mathrm{M}^{\mathrm{zd}}_{#1}\left(#2\right)}
\newcommand{\EL}[2]{\mathrm{EL}_{#1}\left(#2\right)}

\newcommand{\PSL}[2]{\mathrm{PSL}_{#1}\left(#2\right)}

\newcommand{\SLtil}[2]{\widetilde{\mathrm{SL}}_{#1}\left(#2\right)}

\newcommand{\Id}{\mathrm{Id}}

\newcommand{\elm}[2]{\mathrm{E}_{#1,#2}}
\newcommand{\ngp}[1]{\mathrm{N}_{#1}}
\newcommand{\ver}[1]{\mathrm{V}_{#1}}
\newcommand{\hor}[1]{\mathrm{H}_{#1}}
\newcommand{\verd}[1]{\widehat{\mathrm{V}_{#1}(\mathcal{R})}}
\newcommand{\hord}[1]{\widehat{\mathrm{H}_{#1}(\mathcal{R})}}
\newcommand{\elmd}[2]{\widehat{\mathrm{E}_{#1,#2}(\mathcal{R})}}

\newcommand{\verc}[1]{\varphi^{\ver{#1}}}
\newcommand{\horc}[1]{\varphi^{\hor{#1}}}

\newcommand{\verp}[1]{\mu^{\ver{#1}}}
\newcommand{\horp}[1]{\mu^{\hor{#1}}}

\newcommand{\verpm}[1]{\nu^{\ver{#1}}}
\newcommand{\horpm}[1]{\nu^{\hor{#1}}}

\newcommand{\verpv}[1]{\mathrm{P}^{\ver{#1}}}
\newcommand{\horpv}[1]{\mathrm{P}^{\hor{#1}}}
\newcommand{\elmpv}[2]{\mathrm{P}^{\elm{#1}{#2}}}

\newcommand{\levI}[1]{\I_{#1}}	
\newcommand{\kerI}[1]{\K_{#1}}	
\newcommand{\finI}[1]{\mathcal{N}_{#1}}	
\newcommand{\degI}[1]{\mathcal{D}_{#1}}	

\newcommand{\rF}[1]{\mathcal{R}_{#1}}	
\newcommand{\zF}[1]{\mathrm{Z}_{#1}}	
\newcommand{\nF}[1]{\mathrm{N}_{#1}}	
\newcommand{\dF}[1]{\mathrm{D}_{#1}}	
\newcommand{\tF}[1]{\theta_{#1}}		
\newcommand{\tFd}[1]{\widehat{\theta}_{#1}}   
\newcommand{\aF}[1]{\mathcal{A}_{#1}}	
\newcommand{\aFd}[1]{\widehat{\mathcal{A}}_{#1}}	
\newcommand{\aFf}[1]{{\widehat{\mathcal{A}}_{#1}^\mathrm{Fin}}}	
\newcommand{\Orb}[1]{O_{#1}}	
\newcommand{\Ind}[3]{\mathrm{Ind}_{#1}\left(#2;  #3\right)}

\newcommand{\chars}[1]{\mathrm{Ch}\left(#1\right)}
\newcommand{\traces}[1]{\mathrm{Tr}\left(#1\right)}
\newcommand{\charrel}[2]{\mathrm{Ch}_{#1}\left(#2\right)}
\newcommand{\tracerel}[2]{\mathrm{Tr}_{#1}\left(#2\right)}

\DeclarePairedDelimiter\abs{\lvert}{\rvert}%
\DeclarePairedDelimiter\norm{\lVert}{\rVert}%

\makeatletter
\let\oldabs\abs
\def\abs{\@ifstar{\oldabs}{\oldabs*}}
\let\oldnorm\norm
\def\norm{\@ifstar{\oldnorm}{\oldnorm*}}
\makeatother

\newcommand{\nrm}{\ensuremath{\vartriangleleft }}

\title[Characters of $\EL{d}{\R}$]{Characters of the group  $\EL{d}{\R}$ for a commutative Noetherian ring $\R$}

\author{Omer Lavi and Arie Levit}

\begin{document}

\begin{abstract}

Let $\R$ be a commutative Noetherian ring with unit. We classify the characters of the group $\EL{d}{\R}$ provided that $d$ is greater than the stable range of the ring $\R$.  It follows that every character of $\EL{d}{\R}$ is induced from a finite dimensional representation. Towards our main result we classify   $\EL{d}{\R}$-invariant probability measures on the Pontryagin dual group of $\R^d$.
\end{abstract}

\maketitle

\section{Introduction}
\label{sec:intro}






Let $\R$ be a commutative Noetherian ring with unit\footnote{Equivalently  $\R$ is any quotient of the commutative ring $\ZZ\left[x_1,\ldots,x_k\right]$ for some $k \in \NN$.}. Let $\EL{d}{\R}$ for some   $d \in \NN$ be the subgroup of the special linear group $\SL{d}{\R}$ generated by   elementary matrices. 
A \emph{trace}  on the group $\EL{d}{\R}$ is a positive definite conjugation invariant function $\varphi : \EL{d}{\R} \to \CC $ with $\varphi(e) =1$. A \emph{character} of $\EL{d}{\R} $ is a   trace $\varphi$ that cannot be written as a non-trivial convex combination. 
Our main goal is to understand  the set $\chars{\EL{d}{\R}}$ of  characters of the group $\EL{d}{\R}$.

\subsection*{Normal subgroups} For each normal subgroup $N \nrm \EL{d}{\R}$ the characteristic function $1_N$ is a trace on the group $\EL{d}{\R}$. The function $1_N$ is a character if and only if every non-trivial conjugacy class of the quotient group $\EL{d}{\R}/N$ is infinite.  So our goal depends   on understanding   the normal subgroup structure of the group $\EL{d}{\R}$.

Assume that $d \ge 3$. Normal subgroups of the group $\EL{d}{\R}$ are controlled by ideals in the ring $\R$. Let us give a few examples. Consider an ideal $\I \nrm \R$. The   congruence subgroup $\mathrm{C}_d(\I)$  is the kernel of the reduction map modulo the ideal $\I$ $$\rho_\I : \EL{d}{\R} \to \SL{d}{\R/\I}.$$ 
A closely related normal subgroup $\widetilde{\mathrm{C}}_d(\I) \nrm \EL{d}{\R}$ is the preimage of the center $Z(\SL{d}{\R/\I})$ via the reduction map $\rho_\I$. Another normal subgroup   $\EL{d}{\I} \nrm \EL{d}{\R}$ associated to the ideal $\I$ is  the normal closure  of the  elementary matrices belonging to $\mathrm{C}_d(\I)$. By definition $\EL{d}{\I} \le \mathrm{C}_d(\I) \le \widetilde{\mathrm{C}}_d(\I)$.

While these examples are not exhaustive,    it is known that
$$ \widetilde{\mathrm{C}}_d(\I) / \EL{d}{\I} = Z(\EL{d}{\R} / \EL{d}{\I}).$$
In particular  every intermediate subgroup $\EL{d}{\I} \le H \le \widetilde{\mathrm{C}}_d(\I)$ is normal in $\EL{d}{\R}$. The converse is also true, namely  every normal subgroup $N \nrm \EL{d}{\R}$ admits a uniquely determined ideal $\J  \nrm \R$ such that $\EL{d}{\J} \le N \le \widetilde{\mathrm{C}}_d(\J)$. 

\subsection*{Depth ideals} 
The characters of the group $\EL{d}{\R}$ are    controlled by ideals in the ring $\R$. To state our classification result  we  introduce the notion of depth ideals.

 An ideal $\I$ in the ring $\R$ is a  \emph{depth ideal} if $|\J/\I| = \infty$ for every   ideal $ \I \subsetneq \J \nrm \R$.  
Every ideal $\J \nrm \R$ admits a uniquely determined depth ideal $\mfi{\J} \nrm \R $ with $\J \subset \mfi{\J}$ and $|\mfi{\J}/\J| < \infty$.
The ring $\R$ is called \emph{just infinite} if $\mfi{\left(0\right)} = \left(0\right)$ and $\mfi{\I} = \R$ for every non-zero  ideal $ \I \nrm \R$.  For example every Dedekind domain is just infinite. On the other hand the ring $\ZZ\left[x_1,\ldots,x_k\right]$ is  not just infinite for any $k \in \NN$.

\subsection*{Main results}
Let $\varphi \in \chars{\EL{d}{\R}}$ be a character. The  \emph{kernel} of $\varphi$ is 
 $$\ker \varphi = \{g \in \EL{d}{\R} \: : \: \varphi(g) = 1\}.$$
This is a normal subgroup of $\EL{d}{\R}$ and therefore there is a uniquely determined \emph{kernel ideal} $\kerI{\varphi} \nrm \R$ satisfying $\EL{d}{\kerI{\varphi}} \le \ker \varphi \le \widetilde{\mathrm{C}}_d({\kerI{\varphi}})$.  The \emph{level ideal} associated to the character $\varphi$ is   the uniquely determined depth ideal $\levI{\varphi} = \mfi{\kerI{\varphi}}$. 

\begin{theorem}
\label{thm:main theorem}
Assume that $d >  \max\{\sr{\R},2\}$. Let $\varphi \in \EL{d}{\R}$ be a character with level ideal $\levI{\varphi}$ and kernel ideal $\kerI{\varphi}$. Then
\begin{itemize}
\item the character $\varphi$ is induced\footnote{A trace on the group $G$ is \emph{induced} from the subgroup $H $ if $\varphi(g) = 0$ for all elements $g \in G \setminus H$.} from the normal subgroup $\widetilde{\mathrm{C}}_d(\levI{\varphi})$,
\item the subquotient 
$ \mathcal{A}_d(\levI{\varphi}, \kerI{\varphi})   = \widetilde{\mathrm{C}}_d(\levI{\varphi})/\EL{d}{\kerI{\varphi}}$ 
 is  virtually abelian\footnote{A group is called virtually abelian if it admits a finite index abelian subgroup. A discrete group is virtually abelian if and only if it is type I.}, and
\item there is a finite essential\footnote{The notion of \emph{essential} orbits is defined in  \S\ref{sec:induced from abelian}.    This is a technical non-degeneracy assumption needed to ensure our parametrization is bijective.} $\EL{d}{\R}$-orbit $\Orb{\varphi}$ in $\chars{\mathcal{A}_d(\levI{\varphi}, \kerI{\varphi})}$ such that the restriction of the character $\varphi$ to the subgroup $\widetilde{\mathrm{C}}_d(\levI{\varphi})$ is given by
$$ \varphi  = \frac{1}{|\Orb{\varphi} |} \sum_{\psi\in \Orb{\varphi}} \psi.$$
\end{itemize}
\end{theorem}

The notation $\sr{\R}$   stands for the \emph{stable range}\footnote{See \S\ref{sec: structure of SLd} below for the definition and a discussion of stable range as well as for more information on the normal subgroup structure.} of the ring $\R$. For example $\sr{\R} \le 2$ if $\R$ is a Dedekind domain and $\sr{\ZZ\left[x_1,\ldots,x_k\right]} \le k + 2$. 

It is possible to go in the other direction as well.

\begin{theorem}
\label{thm:main converse}
Assume that $ d >  \max\{\sr{\R},2\}$. Let $\I, \K \nrm \R$ be a pair of ideals with $\mfi{\K} = \I$. Then
\begin{itemize}
\item the subquotient 
$\mathcal{A}_d(\I, \K) =\widetilde{\mathrm{C}}_d(\I)/\EL{d}{\K}$
is virtually abelian, and
\item for every finite essential $\EL{d}{\R}$-orbit    $O \subset \chars{\mathcal{A}_d(\I, \K )}$ there  is a uniquely  determined character $\varphi \in \chars{\EL{d}{\R}}$ with $\levI{\varphi} =\I, \kerI{\varphi} = \K$ and $\Orb{\varphi} = O$.
\end{itemize}
\end{theorem}

  Theorems \ref{thm:main theorem} and \ref{thm:main converse}    establish a bijective correspondence between  $\chars{\EL{d}{\R}}$ and the set of all triplets $(\I, \K, O)$ as above. This amounts to a classification of the characters of $\EL{d}{\R}$ provided that $d > \sr{\R}$. Arguably the most important parameter of a given character $\varphi$ is  its level ideal $\levI{\varphi} \nrm \R$.

In fact, the subquotients $\mathcal{A}_d(\levI{\varphi}, \kerI{\varphi})$ are   virtually central in the quotient  group $\EL{d}{\R} / \EL{d}{\kerI{\varphi}}$, i.e. they admit a finite index central subgroup. This   property is   stronger than being virtually abelian per se, allowing us   to deduce the following.

\begin{cor}
\label{cor:finitely induced}
Assume that $d >  \max\{\sr{\R},2\}$. Let $\varphi \in \chars{\EL{d}{\R}}$ be a character. Denote $N = \widetilde{\mathrm{C}}_d(\levI{\varphi}) \nrm \EL{d}{\R}$. Then there is a  unitary representation $\pi$ of the group $N$   on a finite dimensional Hilbert space $\mathcal{H}$  such that
$$ \varphi(g) = \begin{cases} 
\frac{\mathrm{tr} \pi(g)}{\dim_\CC \mathcal{H}} & g \in N,\\
0 & g \notin N.
\end{cases}$$
\end{cor}

In other words, assuming  $d > \sr{\R}$ every character of the group $\EL{d}{\R}$ is induced from a finite dimensional representation of a particular normal subgroup.

The above results   generalize and were inspired    by Bekka's  paper  \cite{bekka},  see the \enquote{Scientific Acknowledgements}   paragraph below.

\subsection*{Examples}

Let us illustrate our Theorems \ref{thm:main theorem} and \ref{thm:main converse} with a few examples.

\begin{enumerate}

\item 
\label{ex:finite field}
If $\R$ is a finite field or a product of finite fields then $\mfi{\I} = \R$ for every ideal $\I \nrm \R$. In particular $\levI{\varphi} = \R$ and the group $\mathcal{A}_d(\levI{\varphi}, \kerI{\varphi})$ is a quotient of the finite group $\SL{d}{\R}$ for all characters $\varphi \in \chars{\SL{d}{\R}}$.  Our main results are trivially true.

\item
\label{ex:on K1}
Assume that $d > \sr{\R}  $. It follows that  $\EL{d}{\R}$ is a normal subgroup of $\SL{d}{\R}$. The quotient $\SL{d}{\R}/\EL{d}{\R}$ is   isomorphic to the $K$-theoretic abelian  group $\mathrm{SK}_1(\R)$. Whenever $\mathrm{SK}_1(\R)$ is trivial our results  apply equally well to $\SL{d}{\R}$. The group  $\mathrm{SK}_1(\R)$ is known to be trivial  if $\R$ is a field, a ring of integers $\mathcal{O}_k$ of an algebraic field \cite{bass1967solution}  or  $ \ZZ\left[x_1,\ldots,x_k\right]$ for some $k \in \NN \cup \{0\}$  \cite{bass1964whitehead}.

\item
\label{ex:bekka}
 Let $\R$ be the ring  $\ZZ$. The characters of the group $\SL{d}{\ZZ}$ for $d > 2$ were classified  in \cite{bekka}.  Bekka proved that every character $\varphi \in \chars{\SL{d}{\ZZ}}$  factors through a finite dimensional representation of $\SL{d}{\ZZ}$ or is   induced from $Z(\SL{d}{\ZZ})$. In our terminology,  Bekka's first and second cases  correspond to characters with level ideal $\ZZ$ and $\left(0\right)$ respectively.

\item
\label{ex:dedekind}
 The same conclusion as in Example (\ref{ex:bekka}) applies more generally   if $\R$ is  the ring of  integers $\mathcal{O}_k$ in the algebraic number field $k$ or the localization of $\mathcal{O}_k$ at a finite set of primes. This is a Dedekind domain with $\sr{\R} = 2$. Therefore our work gives a new proof of   special cases of \cite{peterson2014character} and \cite{boutonnet2019stationary}.


\item 
\label{ex:just infinite}
The previous Examples (\ref{ex:bekka}) and (\ref{ex:dedekind})  are generalized by the following  special case of our main results.

\begin{cor}
\label{cor:just infinite case}
If $\R$ is   just infinite   and $d > \sr{\R}$ then every character $\varphi \in \chars{\EL{d}{\R}}$ factors through a finite dimensional representation of $\EL{d}{\R}$ or is induced from $Z(\EL{d}{\R})$.
\end{cor}

Consider   the polynomial ring  $ \mathbb{F}_p\left[x\right]$ for some prime  $p \in \NN$. It is easy to see that $\mathbb{F}_p\left[x\right]$ is just infinite using the fact it is a principal ideal domain. Moreover  $\sr{\mathbb{F}_p\left[x\right]} = 2$ and $\mathrm{SK}_1(\mathbb{F}_p\left[x\right])$ is trivial. Therefore Corollary \ref{cor:just infinite case} applies  to the groups $\SL{d}{\mathbb{F}_p\left[x\right]}$ for all primes $p$ and all $d \ge 3$.

\item 
\label{ex:universal lattice}
 The group $\SL{d}{\ZZ\left[x_1,\ldots,x_k\right]}$ is a called \emph{universal lattice}. This terminology was introduced by Shalom \cite{shalom1999bounded}. The following is a special case  of our Corollary \ref{cor:finitely induced}.
 
\begin{cor}
\label{cor:universal lattices}
If $d > k + 2$ then every character of the universal lattice $\SL{d}{\ZZ[x_1,\ldots,x_k]}$ is induced from a finite dimensional representation.
\end{cor}

\item 
\label{ex:units}
Let $ f :\R \to \RS$ be a   ring epimorphism   with   $ \ker f = \I \nrm \R$. Let $\units{\RS}$ be the   group of units of the ring $\RS$. 
The central subquotient $ \widetilde{\mathrm{C}}_d(\I)/\mathrm{C}_d(\I)$  is isomorphic\footnote{See \S\ref{sec: structure of SLd}
and in particular   Proposition \ref{prop:center in quotients} for details.}    to the subgroup 
$$\unitsord{\RS}{d} = \{ u \in \units{\RS} \: : \: u^d = 1_{\RS} \}.$$
Therefore any multiplicative character of the abelian group $\unitsord{\RS}{d}$ can be induced  via this isomorphism to a character of the group $\EL{d}{\R}$.
  We remark that the   group $\unitsord{\RS}{d} $ may in general be infinitely generated.

\item 
\label{ex:K1}
Let $\I \nrm \R$ be an ideal. Assuming that $d > \sr{\R}$ the central subquotient $\mathrm{C}_d(\I) / \EL{d}{\I}$ is isomorphic to a subgroup of the abelian relative $K$-theoretic group $\mathrm{SK}_1(\R;\I)$.   Similarly to Example (\ref{ex:units}) any multiplicative character of this abelian group can be induced to a character of $\EL{d}{\R}$.

\item 
\label{ex:finite quot}
Let $\I \nrm \R$ be a depth ideal and $\K \nrm \R$ be an ideal with $\mfi{\K} = \I$. The subquotient $\mathrm{C}_d(\I)/\mathrm{C}_d(\K)$ is a finite group. The group $\EL{d}{\R}$ is acting on the set of irreducible representations of this finite group. Any finite $\EL{d}{\R}$-orbit for this action gives rise to a character of the group $\EL{d}{\R}$.

\item In the most general case a character $\varphi \in \chars{\EL{d}{\R}}$ is   a "mixture" of the three  constructions presented in   Examples (\ref{ex:units}), (\ref{ex:K1}) and (\ref{ex:finite quot}).

\end{enumerate}

\begin{remark} 
A general Noetherian ring $\R$ may  admit infinitely many depth ideals even up to an automorphism. For example the principal ideal $(x^n) \nrm \ZZ\left[x\right]$ is a depth ideal  for all $n \in \NN$.
\end{remark}

\subsection*{Measure classification}

Consider the "vertical" abelian subgroup $\ver{1}(\R) $ of the group $ \EL{d}{\R}$ generated by the elementary subgroups $\elm{2}{1}(\R),\ldots,\elm{d}{1}(\R)$. The subgroup $\ver{1}(\R) $  is   isomorphic to the ring $\R^{d-1}$ with its additive structure. The   normalizer   $N_{\EL{d}{\R}}( \ver{1}(\R))$  contains a copy of the smaller group $\EL{d-1}{\R}$ acting on the subgroup $\ver{1}(\R) $ by conjugation via its natural action  on  $\R^{d-1}$.

The restriction of any character   $\varphi \in \chars{\EL{d}{\R}}$ to the abelian subgroup $\ver{1}(\R)$  is a positive definite function. This restriction can be expressed as the Fourier transform of a  Borel probability measure $\mu_1$ on the Pontryagin dual   $\widehat{\R}^{d-1}$. Consider the dual action  of $\EL{d-1}{\R}$ on the dual group  $\widehat{\R}^{d-1}$   by group automorphisms. As $\varphi$ is conjugation invariant this action preserves the measure  $\mu_1$.   
Our point of departure towards   Theorem \ref{thm:main theorem} is a classification of  all such $\EL{d-1}{\R}$-invariant probability measures\footnote{To ease our notations we replace $d-1$ with $d$ in the statement of Theorem \ref{thm:measure classification intro}}. Such a classification may also be of independent interest.

\begin{theorem}
\label{thm:measure classification intro}
Assume that $d \ge 2$. Let $\mu$ be an ergodic $\EL{d}{\R}$-invariant Borel probability measure on $\widehat{\R}^d$. Then there exists a uniquely determined depth ideal $\I_\mu \nrm \R$ and a finite $\EL{d}{\R}$-orbit 
$$\omega \subset \widehat{\I}_\mu^d \cong (\widehat{\R}/\ann{\I_\mu})^d$$
such that
$$ \mu =  \frac{1}{|\omega|} \sum_{x   \in \omega} x_* \nu$$
where $\nu$ is the Haar measure on the compact abelian group $(\ann{\I}_\mu)^d$. 

Conversely, any such  probability measure $\mu_{\I, \omega}$ on $\widehat{\R}^d$ determined  by a depth ideal $\I \nrm \R$ and a finite $\EL{d}{\R}$-orbit $\omega \subset \widehat{\I}^d$ is $\EL{d}{\R}$-invariant and ergodic.
\end{theorem} 

The Haar measure $\nu_K$ of any closed $\EL{d}{\R}$-invariant subgroup $K \le \widehat{\R}^d$ is clearly $\EL{d}{\R}$-invariant. Theorem \ref{thm:measure classification intro}  says that, up to finite index, this essentially accounts for all the $\EL{d}{\R}$-invariant and ergodic probability measures.  

In  the special case of the ring $\ZZ$ of integers,     Theorem \ref{thm:measure classification intro} specializes to the well-known classification \cite{dani1979ergodic, burger} of ergodic $\SL{d}{\ZZ}$-invariant probability measures on the $d$-dimensional torus $\left( \mathbb{R} /\ZZ \right)^d$. The ring $\ZZ$ is just infinite. The depth ideal $\left(0\right)$ corresponds to the Haar measure on the torus and the depth ideal $  \ZZ$ corresponds to   atomic measures supported on rational $\SL{d}{\ZZ}$-orbits.

The proof of Theorem \ref{thm:measure classification intro} is inspired by  \cite[Proposition 9]{burger}  relying on a  harmonic analysis principle due to Wiener  \cite[Theorem A.2.2]{graham}.  Our situation is made more complicated by the presence of depth ideals and by the lack of a good description of $\EL{d}{\R}$-orbits in $\R^d$. 

\subsection*{Related work}
 The study of group characters has a long and fruitful history. Characters of finite groups are  intimately related to irreducible representations and were studied  extensively   by mathematicians such as Frobenius, Schur and others. 
 
   Thoma extended the theory of group characters  to infinite groups   \cite{thoma1964unitare} and observed  the connection to von Neumann algebras and factors.   Thoma described the characters of the infinite symmetric group \cite{thoma1964positiv}. Vershik \cite{vershik2012totally}  revisited the infinite symmetric group  and related  characters to the study of non-free actions. 
     See e.g.  \cite{peterson2016character} for a more detailed account of the historical development.

Motivated by the Margulis superrigidity   \cite{margulis1977discrete} and the Zimmer cocycle superrigidity theorems \cite{zimmer1980strong} Connes conjectured that   irreducible lattices in  higher rank semisimple Lie groups are \emph{character rigid}. This means that every character of such a lattice $\Gamma$    factors through a finite dimensional representation of $\Gamma$ or is centrally induced. This conjecture was outlined by Jones in \cite{jones2000ten}  as well as e.g. in  \cite{creutz2013character} and \cite{peterson2014character}.

Taking a first step in the  direction of Connes' question,   Bekka showed that the lattices $\SL{d}{\ZZ}$ are character rigid  for all $d \ge 3$  \cite{bekka}.     Peterson and Thom \cite{peterson2016character} proved that the lattices $\SL{2}{A}$ are character rigid where  $A$ is either an infinite field or a localization of a ring of algebraic integers admitting infinitely many units.  Connes' question was  finally settled in the positive by  Peterson \cite{peterson2014character}. See  also the related work of Creutz and Peterson \cite{creutz2013character} as well as the recent work of Boutonnet and Houdayer \cite{boutonnet2019stationary} on stationary characters.

In as far as classifying characters is a generalization of understanding  normal subgroup structure, the above mentioned results as well as the present work follow the line of algebraic investigation of \cite{mennicke1965finite,wilson1972normal, vaserstein1981normal, borevich1984arrangement}, as well as that of the Margulis normal subgroup theorem \cite{margulis1978quotient,margulis1979finiteness} for lattices and the Stuck--Zimmer theorem \cite{stuck1994stabilizers,bader2006factor} for non-free actions.

An important and in a sense the most general class of groups to which our results apply are  the so called \emph{universal lattices}, namely the groups $\SL{d}{\ZZ[x_1,\ldots,x_k]}$ for some $d,k \in \NN$ (our proof requires   $ d > k + 2$). Universal lattices are significant in that they surject onto many arithmetic as well as $S$-arithmetic lattices. These groups were studied by several authors with relation to representation theory and Kazhdan's property (T), see e.g. \cite{shalom1999bounded, shalom2006algebraization, kassabov2006universal, shenfeld2009semisimple, ershov2010property}.



The notion of "centrally induced" characters was studied in the context of nilpotent groups, see e.g. \cite{howe1977representations, carey1984characters, kaniuth2006induced} and the references therein.

Other related works dealing with characters of infinite groups in a geometric context include \cite{dudko2013characters, dudko2014finite,dudko2018diagonal}.

\begin{remark}
Some authors use the terminology \enquote{character} for our \enquote{trace}, and \enquote{indecomposable character} for our \enquote{character}.
\end{remark}


\subsection*{Further questions} Let us point out a few questions that were left open.
\begin{enumerate}
\item We find it intriguing to look for a purely algebraic ring-theoretic characterization of depth ideals (or their classes up to automorphism) in a given commutative Noetherian ring $\R$. For example, this seems like a challenging problem for the ring $\ZZ\left[x_1,\ldots,x_k\right]$.  

 \item In Bekka's proof of character rigidity for the group   $\SL{d}{\ZZ}$ the assumption  $d >  2$ is necessary, for the group $\SL{2}{\ZZ}$ is virtually free and its character theory is "wild". It would be interesting to investigate to what extent our assumption that $d > \sr{\R}$ is necessary in the case of the group $\EL{d}{\R}$.
\item The majority of this work continues to apply with respect to the (perhaps more natural) group  $\SL{d}{\R}$ instead of the group $\EL{d}{\R}$.

One key advantage of working with the smaller group $\EL{d}{\R}$ is that    the subquotient $(\SLtil{d}{\I} \cap \EL{d}{\R})/\EL{d}{\I}$ is central $\EL{d}{\R}/\EL{d}{\I}$  for an arbitrary ideal $\I \nrm \R$   by the theorem of Borevich and Vavilov. 

On the other hand, while the subquotient $\SL{d}{\I}/\EL{d}{\I}$ is indeed central in $\SL{d}{\R}/\EL{d}{\I}$ by the relative Whitehead lemma, the larger subquotient $\SLtil{d}{\I}/\EL{d}{\I}$ is   to the best of our understanding in general two-step nilpotent.

This leads to the $K$-theoretic question of studying the two-step nilpotent subquotients $\SLtil{d}{\I}/\EL{d}{\I}$. When are those central in the quotient group $\SL{d}{\R}/\EL{d}{\I}$? A definitive answer would allow this work to be extended to deal 
with the group $\SL{d}{\R}$.

\item It seems potentially   possible but technically challenging to study characters of Chevalley groups  over root systems other than $A_n$.

\item The questions studied in this work  make sense over the  non-commutative  free ring $\ZZ\left<t_1,\ldots,t_k\right>$ or alternatively over an arbitrary Noetherian non-commutative ring.
\end{enumerate}

\subsection*{Synopsis}

Depth ideals in Noetherian rings and their basic properties are studied in \S\ref{sec:preliminaries}. The whole of \S\ref{sec:classification of prob measures} is dedicated to the question of classifying $\EL{d}{\R}$-invariant probability measures on the Pontryagin dual of the additive group of the ring $\R^d$. This includes a careful study of closed $\EL{d}{\R}$-invariant subgroups and the corresponding Haar measures. It is interesting to note that some ideas that come to the fore in later sections already appear in \S\ref{sec:classification of prob measures} in rudimentary form.

The   theory of characters is the subject of \S\ref{sec: character theory}. This section has an expository nature. We cover the GNS construction, the connection to von Neumann algebras and factors, relative and induced characters and Howe's lemma on the vanishing of characters on the second center.  Another important notion discussed in \S\ref{sec: character theory} is relative characters  of virtually central subgroups.

The following two \S\ref{sec: structure of SLd} and \S\ref{sec:abelian subgroups} are also for the most part  expository and deal with the matrix group $\EL{d}{\R}$ over a general commutative Noetherian ring and its various subgroups. In \S\ref{sec: structure of SLd} we look at normal subgroups, such as the center and congruence subgroups associated to ideals in the ring $\R$. In \S\ref{sec:abelian subgroups}   we consider the "vertical" and "horizontal" abelian subgroups $\ver{i}(\R)$ and $\hor{i}(\R)$ respectively, and their normalizers $\ngp{i}(\R)$ in the group $\EL{d}{\R}$.  The assumption that $d > \sr{\R}$ plays a role in  \S\ref{sec:abelian subgroups} towards a   particular normal form decomposition for conjugates of arbitrary elements, see Proposition \ref{prop:normal form for conjugates}.

The proof of Theorems \ref{thm:main theorem} and \ref{thm:main converse} gets going in earnest in \S\ref{sec: level of the character}. This is where we apply our measure classification result Theorem \ref{thm:measure classification intro} to the Fourier transforms of the restriction of a given character $\varphi \in \chars{\EL{d}{\R}}$ to the "vertical" and "horizontal" abelian subgroups. We deduce that all of these probability measures   correspond to a   uniquely determined depth ideal $\levI{\varphi} \nrm \R$, the so called \emph{level ideal} associated to the character $\varphi$.

It is \S\ref{sec:kernel ideal} where our proof most   diverges from Bekka's \cite{bekka}  due to complications introduced by the fact that a commutative Noetherian ring will in general admit multiple depth ideals. We consider a uniquely determined \emph{kernel ideal} $\kerI{\varphi} \nrm \R$ associated to the kernel of the character $\varphi \in \chars{\EL{d}{\R}}$. A very careful analysis in needed   to prove that $\mfi{\kerI{\varphi}} =\levI{\varphi}$, or in other words that $|\levI{\varphi}/\kerI{\varphi}| < \infty$.

We show in \S\ref{sec:induced from level}
  that any character $\varphi \in \chars{\EL{d}{\R}}$ is induced from the normal subgroup $\SLtil{d}{\levI{\varphi}} \cap \EL{d}{\R}$ where $\levI{\varphi}$ is the level ideal corresponding to $\varphi$. In the final \S\ref{sec:induced from abelian} we conclude the formal proof  of Theorems \ref{thm:main theorem} and \ref{thm:main converse} relying on results from the previous sections.

\subsection*{Scientific acknowledgements}

Our work  was tremendously inspired by Bachir Bekka's    paper  \cite{bekka}. We owe Professor  Bekka   an intellectual debt for his approach to   character classification   for the group $\SL{d}{\ZZ}$ with $d \ge 3$. This work grew out of an attempt to extend and adapt   \cite{bekka} to    general commutative Noetherian rings.  We have attempted to explicitly reference  \cite{bekka} when appropriate, however an attentive reader will be able to detect Bekka's strategy in the scaffolding of Theorems \ref{thm:main theorem} and \ref{thm:main converse} and their proofs.

An important source of inspiration towards the   classification of   $\EL{d}{\R}$-invariant probability measures on $\widehat{\R}^d$  in Theorem \ref{thm:measure classification intro} was Marc Burger's  work \cite{burger}.  

The authors would like to thank Avraham (Rami) Aizenbud for several   insightful discussions and Nikolai Vavilov for  his  illuminating  comments and suggestions.


\subsection*{Acknowledgements} This project turned out to be an  unexpectedly  prolonged endeavor, spanning multiple stages of the authors' professional careers and personal lives. 
The authors would like to thank their respective advisors Professors Uri Bader and Tsachik Gelander for their extraordinary and  continuous encouragement and  advice. It is unlikely this project would     come to  fruition without Uri and Tsachik's support. 


\vspace{10pt}

\setcounter{tocdepth}{2} 
\tableofcontents

\begin{center}
\emph{This work is dedicated to our wives Yael and Neta and to our children. We would like to thank them   for their support and patience.}
\end{center}

\pagebreak

\section{Depth in Noetherian rings}
\label{sec:preliminaries}

Let $\R$ be a Noetherian commutative  unital ring. Let $x_1,\ldots,x_k \in \R$ with $k \in \NN$ be a fixed set of generators  for the ring $\R$.

\subsection*{Depth ideals}
\label{sub:Noetherian rings}

Let $\I \nrm \R$ be any ideal. The Noetherianity of $\R$ implies that there exists a unique ideal $\mfi{\I} \nrm \R$   maximal with respect to the two conditions $\I \subset \mfi{\I}$ and $|\mfi{\I}/\I| < \infty$. We will say that the ideal $\mfi{\I}$ is the \emph{depth} of the ideal $\I$ in the ring $\R$.

A \emph{depth ideal} is any ideal $\J\nrm\R$  satisfying $\mfi{\J} = \J$. Clearly $\J$ is a depth ideal if and only if the quotient ring $\R/ \J$ admits no non-zero ideals of finite cardinality.
Note that the depth of any ideal is a depth ideal, and any depth ideal is equal to its own depth.

It is clear that the ring $\R$ satisfies $|\R| < \infty$ if and only if $\mfi{\left(0\right)} = \R$. The ring $\R$ is called \emph{just infinite} if $\mfi{\left(0\right)} = \left(0\right)$ and every non-zero ideal $\I \nrm \R$ has $\mfi{\I} = \R$.

\begin{example}
The only depth ideals in the ring $\ZZ$ are  $\left(0\right)$ and $\ZZ$. If $\left(n\right) \nrm \ZZ$ is any non-zero ideal then $\mfi{\left(n\right)} = \ZZ$. The same situation holds true in any Dedekind domain. In particular every Dedekind domain is just infinite.
\end{example}

The following results are intended to provide a workable criterion for depth ideals.

\begin{lemma}
	\label{lem:division of polynomials}
	Let $n,m,N,M \in \NN$ be such that $n < m$ and $N < M$. Let $x \in \R$ be any element. If $N \ge n$ and $(m-n)|(M-N)$ then $(x^m-x^n)|(x^M-x^N)$.
\end{lemma}
\begin{proof}
	Assume that $M-N = l(m-n)$ for some $l \in \NN$. The fact that $x^m-x^n$ divides $x^M-x^N$      in the ring $\R$ can be seen by considering the following equation
	$$ x^M - x^N = x^{N}(x^{M-N}-1)= x^{N-n}(x^m-x^n)(x^{(l-1)(m-n)}+\cdots+x^{m-n}+1). $$
\end{proof}

\begin{prop}
	\label{prop:condition for a finitely generated module to be finite}
	Let $D$ be a finitely generated $\R$-module.   Then $|D| < \infty$ if and only if
	\begin{itemize}
		\item there exists $N \in \NN$ such that $N \in \mathrm{Ann}_{\R}(D)$,  and
		\item for each $ i \in \{1,\ldots,k\}$ there are indices $n_i,m_i \in \NN$ with $n_i < m_i$ such that $x_i^{m_i} - x_i^{n_i} \in \mathrm{Ann}_{\R}(D)$.
	\end{itemize}
\end{prop}
\begin{proof}
	Fix a finite generating set $a_1,\ldots,a_l  $  for the $\R$-module $D$ for some $l \in \NN$. 
	
Working in one direction  assume that $|D| < \infty$. In particular $D$ is a finite abelian group with respect to its additive structure. Therefore $D$ has a finite exponent $N \in \NN$. In other words $N \in \mathrm{Ann}_{\R}(D)$.
	
Given a  generator $x_i$ of the ring $\R$  and a generator $a_j$ of the module $D$ for $i \in \{1,\ldots,k\}$ and $j \in \{1,\ldots,l\}$  the elements $x_i^n a_j \in D$ cannot be pairwise distinct as $n$ ranges over $\NN$. Therefore there are some indices $m_{i,j}, n_{i,j} \in \NN$ with $m_{i,j} > n_{i,j}$ satisfying
	$$x_i^{n_{i,j}}-x_i^{m_{i,j}} \in \mathrm{Ann}_{\R}(a_j).$$
For all $ i \in \{1,\ldots,k\}$ denote
	$$ N_i = \max_{j \in \{1,\ldots,l\}} \{ n_{i,j}   \} \quad \text{and} \quad M_i = N_i + \lcm_{ j \in \{1,\ldots,l\}} \{ m_{i,j} - n_{i,j}  \}. $$
We conclude by   Lemma \ref{lem:division of polynomials} that $x_i^{M_i} - x_i^{N_i} \in \mathrm{Ann}_\R(a_j)$ for all $j \in \{1,\ldots,l\}$. This implies that $x_i^{M_i} - x_i^{N_i} \in \mathrm{Ann}_\R(D)$ for all $ i \in \{1,\ldots,k\}$ as required.

	In the other direction, assume that the annihilator $\mathrm{Ann}_{\R}(D)$ contains some integer $N \in \NN$ as well as the elements $x_i^{m_i} - x_i^{n_i} \in \R$ for each $ i\in \{1,\ldots,k\}$ as above. Consider the following countable subset of the module $D$
	$$ B = \left\{x_1^{e_1}  \cdots x_k^{e_k} a_j \: : \: e_i \in \NN\cup\{0\}, \; j \in \{1,\ldots, l\}\right\}.$$
	It is clear that $B$ generates $D$ as a $\ZZ$-module.
	We may use the elements $ x_i^{m_i} - x_i^{n_i} $ of $\mathrm{Ann}_\R(D)$  to reduce every exponent $e_i$ below $m_i$.    Therefore $D$ admits a finite generating set as a $\ZZ$-module. Since $N \in \mathrm{Ann}_{\R}(D)$ it follows that $|D| < \infty$.
\end{proof}

This implies that if $D_1$ and $D_2$ are two finitely generated $\R$-modules with
$\mathrm{Ann}_{\R}(D_1) \subset \mathrm{Ann}_{\R}(D_2)$
and $|D_1| < \infty$ then $|D_2| < \infty$.

\begin{prop}
	\label{prop:a condition for an ideal to be maximal of finite index}
	An ideal $\I \nrm \R$ is a depth ideal if and only if every element $r \in \R \setminus \I$ satisfies at least one of the following two conditions
	\begin{enumerate}
		\item $nr \notin \I$ for all $n \in \NN$, or
		\item there exists an index $i \in \{1,\ldots, k\}$   such that $(x_i^n  - x_i^m)r \notin \I$ for all distinct $n,m \in \NN$.
	\end{enumerate}
\end{prop}
\begin{proof}
	Recall that  the ideal $\I$ is   a  depth ideal if and only if the quotient ring $\R/\I$ admits no non-zero ideals of finite cardinality.  The last statement is equivalent to the quotient $\R/\I$ not admitting  any    non-zero  \emph{principal} ideals of finite cardinality.
	
	Consider any non-trivial element $r + \I   $ of the quotient ring $\R/\I$. Let
	$$\LL = \R r + \I \nrm \R/\I$$ be the principal ideal  generated by $r + \I$. The ideal $\LL \nrm \R/\I$ has $|\LL | = \infty$   if and only if at least one of above two conditions is satisfied, see   Proposition \ref{prop:condition for a finitely generated module to be finite} above.
\end{proof}

\begin{example}
	Consider  the ring $  \ZZ\left[x_1,\ldots,x_k\right]$. Any ideal generated by a collection of monomials of the form $x_1^{e_1}\cdots x_k^{e_k}$ with $e_i \in \NN \cup \{0\}$ is a depth ideal.
\end{example}

\begin{example} The sum of two depth ideals need not be a depth ideal. This is witnessed by the   two depths ideals $ (2)$ and $ (x)$  in the ring $\ZZ[x]$.  
\end{example}

Let us summarize some further properties of the notion of depth.

\begin{prop}
	\label{prop:properties of depth}
	Let $\I, \J \nrm \R$ be a pair of ideals.
	\begin{enumerate}
		\item 
		\label{it:intersection depth}
		$\mfi{\left(\I \cap \J\right)} = \mfi{\I} \cap \mfi{\J}$.
		
		\item 
		\label{it:subset depth}
		If $\J \subset \I$ then $\mfi{\J} \subset \mfi{\I}$.

		\item 
		\label{it:commensurable}		
		$\mfi{\I} = \mfi{\J}$ if and only if $\I$ and $\J$ are \emph{commensurable}, i.e. $\left|\I/(\I \cap \J)\right| < \infty$ and $\left|\J/(\I \cap \J)\right| < \infty$.
	\end{enumerate}
\end{prop}
\begin{proof}
We first show that $\mfi{\I} \cap \mfi{\J}$ is a depth ideal. Indeed, let $\LL \nrm \R$ be any ideal such that $\mfi{\I} \cap \mfi{\J} \subset \LL$ and $\left|\LL/(\mfi{\I} \cap \mfi{\J})\right| < \infty$. Therefore $\left|(\LL+ \mfi{\I}) / \mfi{\I} \right|< \infty$ as well as $\left|(\LL+ \mfi{\J}) / \mfi{\J} \right|< \infty$. Since both $\mfi{\I}$ and $\mfi{\J}$ are depth ideals it follows that $\LL\subset \mfi{\I} \cap \mfi{\J} $, as required. In other words the quotient ring $\R/(\mfi{\I} \cap \mfi{\J})$ admits no non-zero ideals of finite cardinality.

Assume that $ \J \subset \I$. It follows that $\J   \subset \mfi{\I}$ and so $\J \subset \mfi{\I} \cap \mfi{\J} \subset \mfi{\J}$. The intersection $\mfi{\I} \cap \mfi{\J}$ is a depth ideal by the above argument. As $\left|(\mfi{\I} \cap \mfi{\J}) / \J   \right| < \infty$ we conclude that $\mfi{\J} = \mfi{\I} \cap \mfi{\J} \subset \mfi{\I}$. Item (\ref{it:subset depth}) follows.

To prove Item (\ref{it:intersection depth}) it remains to establish that $|(\mfi{\I} \cap \mfi{\J})/(\I \cap \J)|<\infty$. We may apply Proposition 	\ref{prop:condition for a finitely generated module to be finite} to find some integers $N_\I,N_\J \in \NN$ satisfying
$$N_\I \in \mathrm{Ann}_\R (\mfi{\I}/\I) \quad \text{and} \quad N_\J \in \mathrm{Ann}_\R (\mfi{\J}/\J).$$
It is easy to verify that
$$  N = \lcm\{N_\I, N_\J\} \in \mathrm{Ann}_\R ((\mfi{\I} \cap \mfi{\J})/(\I \cap \J)).$$
Likewise for each $ i\in \{1,\ldots,k\}$ there are some indices $m_{\I,i},n_{\I,i},m_{\J,i},n_{\J,i} \in \NN$ with $m_{\I,i} > n_{\I,i}$ and $m_{\J,i} > n_{\J,i}$ satisfying
$$ x_i^{m_{\I,i}} - x_i^{n_{\I,i}} \in \mathrm{Ann}_\R (\mfi{\I}/\I) \quad \text{and} \quad x_i^{m_{\J,i}} - x_i^{n_{\J,i}} \in \mathrm{Ann}_\R (\mfi{\J}/\J).$$
Therefore for each $ i \in \{1,\ldots,x_k\}$ the integers 
$$ n_i = \max \{ n_{\I,i}, n_{\J,i} \} \in \NN  \quad \text{and} \quad m_i = n_i + \lcm \{ m_{\I,i} - n_{\I,i}, m_{\J,i} - n_{\J,i} \} \in \NN $$
 satisfy
$$ x_i^{m_i} - x_i^{n_i} \in \mathrm{Ann}_\R ((\mfi{\I} \cap \mfi{\J})/(\I \cap \J)).$$
    according to Lemma 	\ref{lem:division of polynomials}.
    We conclude the proof of Item (\ref{it:intersection depth}) relying on the converse direction of Proposition 	\ref{prop:condition for a finitely generated module to be finite}.

It remains to prove Item (\ref{it:commensurable}). In one direction assume that $\mfi{\I} = \mfi{\J}$. It follows from Item (\ref{it:intersection depth}) that   $\mfi{\I} = \mfi{\J}$ is the depth of the ideal $\I \cap \J$. Therefore $\I$ and $\J$ are commensurable. The other direction of Item (\ref{it:commensurable}) follows from the definition of depth.
\end{proof}

Item (\ref{it:intersection depth}) of  Proposition \ref{prop:properties of depth} implies in particular that the intersection of two depth ideals is again a depth ideal. Moreover, it follows from Item (\ref{it:commensurable}) that having the same depth is an equivalence relation on the set of   ideals of the ring $\R$. The depth ideals  parameterize the equivalence classes.


\begin{remark}
The question of classifying depth ideals in commutative Noetherian  rings seems to be of interest. We were unable to arrive at any definite statement. See the \enquote{Further questions} paragraph of \S\ref{sec:intro}.
\end{remark}

\subsection*{Ideal quotient}
Let $\I$ and $\J$  be a pair of ideals in the ring $\R$.  The \emph{ideal quotient} $ \J : \I$ is discussed  e.g. in  \cite[p. 8]{atiyah2018introduction}. It is given by
$$  \J : \I = \mathrm{Ann}_\R((\I + \J)/\J) = \{x \in \R \: : \: x\I \subset \J \}. $$
Note the the ideal quotient $ \J : \I$ is an ideal in $\R$. Let $\LL$ be another ideal in the ring $\R$. The following facts are immediate from the above definition. 
\begin{itemize}
	\item $\I(\J : \I) \subset \J \subset \J : \I$,
	\item $(\I \cap \J) : \I = \J : \I$, and
	\item if $\J \subset \I$  then $\J : \LL \subset \I : \LL$.
\end{itemize}
The following is another elementary consequence of the definition.

\begin{prop}
\label{prop:law of cancellation for ideal quotient}
If $\LL \nrm \R$ then
$  (\LL: \I):(\J:\I) \subset \LL : (\J:\I)\I$.
\end{prop}
\begin{proof}
Consider some element $z \in (\LL: \I):(\J:\I)$. To prove that $z \in \LL : (\J:\I)\I$ we need to verify that  all
$y \in (\J:\I)$ and $x \in \I$ satisfy $zyx \in \LL$. But $zy \in (\LL: \I)$   so that $zyx \in (\LL:\I) \I \subset \LL$ as required.
\end{proof}

\begin{prop}
\label{prop:finiteness of ideal quotient}
Assume that   $\J \subset \I$. Then  $|\I/\J| < \infty$    if and only if the ideal quotient $\J : \I$ has finite index in the ring $\R$.
\end{prop}
\begin{proof}
Observe that
$$\mathrm{Ann}_{\R}(\I/\J) = \J : \I  = \mathrm{Ann}_{\R}(\R/(\J : \I)).$$
As both $\R$-modules $\I/\J$ and $\R/(\J : \I)$ are finitely generated, the result follows from  Proposition \ref{prop:condition for a finitely generated module to be finite}.
\end{proof}

\begin{prop}
\label{prop:finiteness of something with ideal quotients}
Assume that   $\J \subset \I$ and $|\I/\J| < \infty$. Then
$ | \I / (\J:\I)\I | < \infty$.
\end{prop}
\begin{proof}
Since $|\I/\J| < \infty$ it follows from Proposition \ref{prop:finiteness of ideal quotient} that the ideal quotient $\J:\I$ has finite index in $\R$. Regard the ideal $\I$ as a $\R$-module. Then $(\J:\I)\I$ is an $\R$-submodule   of $\I$. The quotient $\R$-module $\I / (\J:\I)\I$ is finitely generated and satisfies
$$  \mathrm{Ann}_{\R}(\R/(\J : \I))=  \J:\I \subset \mathrm{Ann}_{\R}(\I / (\J:\I)\I). $$
The   conclusion follows from Proposition \ref{prop:condition for a finitely generated module to be finite} and the remark following it.
\end{proof}

\begin{prop}
\label{prop:infinite index implies infinite index ideal quotient}
Let $\LL \nrm  \R$ be an ideal. Assume that
$ \LL \subset \J \subset \I$. If  $|\I/\J| < \infty$ and $ |\J/\LL| = \infty$ then
 $ |(\J: \I)/ (\LL:\I)| = \infty$.
\end{prop}
\begin{proof}
Note that $ |(\J:\I)/(\LL: \I)| = \infty $ if and only if the ideal
$ (\LL: \I):(\J:\I) $
has infinite index in the ring $\R$ by Proposition \ref{prop:finiteness of ideal quotient}.  
This statement follows provided that
the ideal
$$\LL : (\J:\I)\I = (\LL \cap  (\J:\I)\I) : (\J:\I)\I $$ has infinite index in the ring $\R$ according to Proposition \ref{prop:law of cancellation for ideal quotient}. On the one hand, the ideal $(\J:\I)\I$ has finite index in $\I$ by Proposition \ref{prop:finiteness of something with ideal quotients}. One the other hand,  the ideal $\LL \cap  (\J:\I)\I$ is contained in $\LL$ and in particular has infinite index in $\I$. The result follows from yet another application of Proposition \ref{prop:finiteness of ideal quotient}.
\end{proof}



\section{Classification of $\EL{d}{\R}$-invariant probability measures on  $\widehat{\R}^d$}
\label{sec:classification of prob measures}

Let $\R$ be a Noetherian  commutative unital   ring.
Fix some $d \ge 2$. Let $\EL{d}{\R}$ denote the subgroup of the special linear group $\SL{d}{\R}$ generated by    elementary matrices. These are   the matrices  $\elm{i}{j}(r)$  given by
$$ (\elm{i}{j}(r))_{k,l} = \begin{cases} 1 & k=l, \\ r & \text{$k=i$ and $l = j$,}\\ 0 & \text{otherwise} \end{cases}$$
for some pair of distinct indices $i,j \in \{1,\ldots,d\}$  and some ring element $r \in \R$.

Consider the discrete abelian group $\Gamma = \R^d$ regarded with its additive ring operation. 
The matrix group $\EL{d}{\R}$ is acting on the abelian group $\Gamma$ by group automorphisms via matrix multiplication.

Let $K$ denote the Pontryagin dual  of the group $\Gamma$.   The abelian group  $K$  is   compact.
The matrix group  $\EL{d}{\R}$ is acting on $K$ by group automorphisms via the dual action, namely
$$ g(\chi)(\gamma) = \chi(g^{-1} \gamma)  \quad \forall g\in \EL{d}{\R}, \chi \in K, \gamma \in \Gamma. $$
We will use the same notation for both actions of $\EL{d}{\R}$ on $\Gamma$ and on $K$.

The goal of this section is to classify all $\EL{d}{\R}$-invariant Borel probability measures on the compact group $K$ in terms of depth ideals in the ring $\R$ and complete the proof of Theorem \ref{thm:measure classification intro}. First we would like to restate  this classification  more precisely using some additional notation.

For a closed subgroup $H\le K$ we   let $\ann{H} \le \Gamma $   denote the \emph{annihilator}  of   $H$,  namely
$\ann{H} = \{ \gamma \in \Gamma \, : \, \chi(\gamma) = 1 \; \forall \chi \in H\}$.  Similarly, for a subgroup $\Delta \le \Gamma$ the \emph{annihilator}  $\ann{\Delta}$ is a closed subgroup   of $K$. 

Given an ideal $\I \nrm \R$ denote $\Gamma_{\I} = \I^d$ and $K_{\I} = \ann{\Gamma_{\I}} = \left(\ann{\I}\right)^d$. Every such closed subgroup  $K_\I$ is  clearly $\EL{d}{\R}$-invariant.



\begin{theorem}
\label{thm:classification of invariant measures}
The  set of  ergodic  $\EL{d}{\R}$-invariant probability measures on the compact group $K$ bijectively corresponds to the set of all pairs   $(\I,\omega)$ where  $\I \nrm \R$ is a depth ideal and  $\omega$ is a finite orbit for the $\EL{d}{\R}$-action on the quotient $K/K_\I$.
The ergodic probability measure corresponding to the pair $(\I,\omega)$ is given by
	$$ \mu_{\I,\omega} = \frac{1}{\left|\omega \right|} \sum_{gK_\I \in \omega} g_* \nu_\I $$
	where $\nu_\I$ is the Haar measure on the compact group $K_\I$.
\end{theorem}

The above formula for the probability measure $\mu_{(\I,\omega)}$ makes sense, as the Haar measure $\mu_\I$ is left-invariant for the action of the group $K_\I$.

\begin{example}
Let $\R = \ZZ$ so that $K = (\mathbb{R}/\mathbb{Z})^d$ is the $d$-dimensional torus. The ring $\ZZ$ is just infinite and admits only two depth ideals, namely $\left(0\right)$ and $  \ZZ$. The case $K_{\left(0\right)} = K$  corresponds to the Haar measure on the torus. The remaining case $K_\ZZ = \{e\}$ corresponds to atomic probability measures supported on finite rational orbits. This  recovers a well-known measure classification result   \cite{dani1979ergodic, burger}.
\end{example}



\begin{remark}
We  use the same notation   $d $ for the size  of the matrix group $\EL{d}{\R}$ considered  in   this section as well as elsewhere  throughout this paper.
However, towards classifying the characters of the group   $\EL{d}{\R}$ we will be applying Theorem \ref{thm:classification of invariant measures}
 to the action of the \emph{smaller} matrix group $\EL{d-1}{\R}$ on the abelian group $\R^{d-1}$.
\end{remark}

\subsection*{$\EL{d}{\R}$-invariant subgroups}

The  first step towards proving  Theorem \ref{thm:classification of invariant measures}  is to  classify all   $\EL{d}{\R}$-invariant subgroups of $K$.

\begin{prop}
	\label{prop:G-invariant subgroups}
	A subgroup $\Delta$ of $  \Gamma$ is $\EL{d}{\R}$-invariant if and only if
	$\Delta = \Gamma_\I$   for some ideal $\I \nrm \R$.
\end{prop}
\begin{proof}
As the group $\EL{d}{\R}$ is generated by elementary matrices, it is clear that  the subgroup  $\Gamma_\I = \I^d$ of $\Gamma$ is $\EL{d}{\R}$-invariant for every ideal $\I \nrm \R$. This proves one direction of the statement.

To see the converse direction  consider an arbitrary $\EL{d}{\R}$-invariant subgroup $\Delta$ of $ \Gamma$. Let  $\I\nrm \R$   denote  the ideal  generated by the  coordinates of all  group elements of $\Delta$, namely
$$ \I = \sum_{\gamma =  (\gamma_1,\ldots,\gamma_d) \in \Gamma} \, \sum_{i=1}^d \R \gamma_i. $$
Clearly $\Delta \le \Gamma_\I$. To prove the opposite containment $ \Gamma_\I \le \Delta $ it suffices to show that the subgroup $\Delta$ contains all elements  of the form $(0,\ldots,0,rs,0,\ldots,0)$ where $r \in \R$ is any ring element and $s \in \R$ is   any coordinate in some element $\gamma \in \Delta$.
Indeed let $\gamma \in \Delta$ be a group element whose $i$-th coordinate is equal to $s \in \R$. Fix some index $j \in \{1,\ldots,d\} \setminus \{i\}$.  Then for any ring element $  r\in \R $
	$$ \elm{j}{i}(r) x - x = (0,\ldots,0,rs, 0, \ldots, 0) \in \Delta$$
with the $rs$ entry being at the $j$-th coordinate.
Recall that, up to sign, every two coordinates can be swapped by the   $\EL{d}{\R}$-action relying on the formula
	$$ \elm{i}{j}(1) \elm{j}{i}(-1)\elm{i}{j}(1)  = \begin{pmatrix} * & & & & * \\
	& 0 &   & +1&  \\
	& & * & & \\
	&  -1 &   & 0 & \\
	* & & & & *
	\end{pmatrix}.$$
	Therefore, replacing the ring element $r$ by $-r$ in the above argument if necessary, the element $(0,\ldots,0,rs,0,\ldots,0)$  with the non-zero   entry being at an arbitrary coordinate belongs to $\Delta$. We conclude that  $\Gamma_\I \le \Delta$ and so $\Delta = \Gamma_I$ as required.
\end{proof}

A subgroup $\Delta$ of $\Gamma$ is $\EL{d}{\R}$-invariant if and only if its annihilator $\ann{\Delta}$ is. The fact that $\ann{\Gamma_\I} = K_\I$ gives the following.

\begin{cor}
A closed subgroup  $H$ of $K$ is $\EL{d}{\R}$-invariant if and only $H = K_\I$ for some ideal $\I \nrm \R$.
\end{cor}

Let $\I$ and $\J$ be a pair of ideals in the ring $\R$. Clearly $\Gamma_{\J} \le \Gamma_{\I}$ if and only if $\J \le \I$. In that case the index
$\left[\Gamma_{\I}:\Gamma_{\J}\right]$ is finite if and only if $|\sfrac{\I}{\J}| < \infty$.

These remarks  immediately imply that every   $\EL{d}{\R}$-invariant subgroup $\Delta \le \Gamma$  is contained in a unique $\EL{d}{\R}$-invariant subgroup $\mfi{\Delta} \le \Gamma$ maximal with respect to the condition $[\mfi{\Delta} : \Delta] < \infty$. If $\Delta = \Gamma_{\I}$ then $\mfi{\Delta} = \Gamma_{\mfi{\I}}$.
There is a dual notion $\mfi{H}$ defined for every   closed $\EL{d}{\R}$-invariant subgroup $H$ of $  K$.


%
%
%


\subsection*{Essential subgroups}
The next step towards    Theorem \ref{thm:classification of invariant measures} is to use the ergodicity assumption to understand the support of the Borel probability measure $\mu$.

	For a closed $\EL{d}{\R}$-invariant subgroup  $H$ of $K$ denote
		$$ \Oh{H} = H \setminus \bigcup_{H' \lneq H} H' $$
	where the union is taken over all $\EL{d}{\R}$-invariant proper closed subgroups $H'$ contained in $H$.
There are only countably many $\EL{d}{\R}$-invariant closed subgroups of $K$ by Proposition \ref{prop:G-invariant subgroups}. Therefore $\Oh{H}$ is a Borel subset of $K$ for every such subgroup $H$. 

The sets $\Oh{H}$ form a Borel partition of the compact group $K$. To see this, note that if $H_1$ and $  H_2  $ are two distinct closed $\EL{d}{\R}$-invariant subgroups of $K$ then $H_1 \cap H_2 $ is also a closed   $\EL{d}{\R}$-invariant subgroup of $K$. Therefore $\Oh{H_1} \cap \Oh{H_2} = \emptyset$. Moreover every element $g \in K$ belongs to  some Borel set $\Oh{H}$ where $H$ is the smallest closed $\EL{d}{\R}$-invariant subgroup containing $g$.

Let $\mu$ be an ergodic $\EL{d}{\R}$-invariant  probability measure on the compact group $K$. The above discussion shows that  there exists a unique ideal $\I_\mu \nrm \R$ such that
 $ \mathrm{supp}(\mu) \subset \Oh{K_{\I_\mu}}$.

 We will say that $\I_\mu$ is the \emph{essential ideal} and $K_{\I_\mu}$ is the \emph{essential subgroup} associated to  the ergodic probability measure $\mu$. Note that $\mu(K_{\I_\mu}) = 1$ and $\mu(H') = 0$ for every proper  $\EL{d}{\R}$-invariant closed subgroup $H'$ of $K_{\I_\mu}$.

\subsection*{Orbits for the $\EL{d}{\R}$-action on $\R^d$ }

The Euclidean algorithm can be used to prove that the orbits of the $\SL{d}{\ZZ}$-action on the abelian group $\ZZ^d$ are   parametrized by the greatest common divisor of   the coordinates of a given point $x \in \ZZ^d$.
We were unable to give such a precise characterization of orbits in the  general case.
Still, we show that in  certain situations orbits can be assumed to be fairly large, in the following precise sense.

\begin{prop}
	\label{prop:orbits are sufficiently large}
	Let $\mu$ be an ergodic $\EL{d}{\R}$-invariant  probability measure on the compact group $K$ with essential ideal $\I_\mu$.   Let $\gamma \in \Gamma \setminus \mfi{\Gamma}_{\I_\mu}$. Then there exists a subgroup $\Lambda$ of $  \Gamma$ such that 
 $\gamma + \Lambda \subset \EL{d}{\R} \gamma $ and
		 $\mu(\ann{\Lambda}) = 0$.
\end{prop}

The notation $\mfi{\Gamma}_{\I_\mu}$ in the statement of Proposition \ref{prop:orbits are sufficiently large} stands for the maximal subgroup of $\Gamma$ containing $\Gamma_{\I_\mu}$ with respect to the condition that $\left[\mfi{\Gamma}_{\I_\mu}:\Gamma_{\I_\mu}\right] < \infty$. In other words $\mfi{\Gamma}_{\I_\mu} = \Gamma_{\mfi{\I}_\mu}$ where $\mfi{\I}_\mu$ is the depth of the essential ideal $\I_\mu$.
Note that the subgroup   $\Lambda$ will not be  $\EL{d}{\R}$-invariant in general.

We first prove a straightforward but technical Lemma, to be used in the proof of Proposition \ref{prop:orbits are sufficiently large}  below.
\begin{lemma}
	\label{lem:subgroups generated by pairs}
 Let $\I \nrm \R$ be any ideal. For every element $s \in \R$ consider the subgroup $\Sigma_s \le \Gamma$ given by
		$$ \Sigma_s =   \left\{  (\gamma_1,\gamma_2,\ldots,\gamma_d) \in \Gamma_{\I} \: : \: \gamma_1 = s \gamma_2 \} \right\}. $$
		Then every pair of elements $s, s' \in \R$ satisfy
		$\Gamma_{(s-s' )\I} \le  \Sigma_{s} + \Sigma_{s'}.$
\end{lemma}

\begin{proof}
Let $s,s' \in \R$ be a fixed pair of elements. Then
	$$ (sr, r,0,\ldots,0) - (s'r, r,0,\ldots,0) = ((s-s')r, 0,0,\ldots,0) \in \Sigma_s + \Sigma_{s'} $$
and
	$$ - (ss'r, s'r,0,\ldots,0) + ( s'sr, sr,0,\ldots,0 ) = (0,(s-s')r,0,\ldots,0) \in \Sigma_s + \Sigma_{s'} $$
 for every element $r \in \I$. 
It is clear that   $(0,\ldots,0,\I,0,\ldots,0) \le \Sigma_s + \Sigma_{s'}$ holds true  starting from the third coordinate and onwards. This concludes the proof.
\end{proof}

\begin{proof}[Proof of Proposition \ref{prop:orbits are sufficiently large}]
Let $\I_\mu$ be the essential ideal and $ K_{\I_\mu}$ be the essential subgroup associated to the ergodic $\EL{d}{\R}$-invariant probability measure $\mu$. 

Write $\gamma = (\gamma_1,\ldots, \gamma_d) \in \Gamma \setminus \Gamma_{\I^*_\mu}$  and assume without loss of generality that $\gamma_1 \notin \mfi{\I}_\mu$.
  Consider the subgroup $\Lambda \le \Gamma$ given by
	$$ \Lambda = \{(\delta_1,\ldots,\delta_d) \in \Gamma \: : \: \delta_1 = 0, \; \delta_2, \ldots, \delta_d \in \R\gamma_1 \}.$$
Applying the product of the commuting elementary matrices $\elm{2}{1}(r_2),\ldots,\elm{d}{1}(r_d)$   for some ring elements $r_2,\ldots,r_d \in \R$  to the element $\gamma$ gives
$$ \elm{2}{1}(r_2)\cdots \elm{d}{1}(r_d)\gamma = \gamma + (0,r_2\gamma_1,\ldots,r_d\gamma_1). $$
It follows that $\gamma + \Lambda \subset  \EL{d}{\R}\gamma$.
	
We claim that the subgroup $\ann{\Lambda}$ of $K$ admits countably many $ \EL{d}{\R}$-translates whose  pairwise intersections are   $\mu$-null. Since the probability measure $\mu$ is $ \EL{d}{\R}$-invariant, this claim  implies that  $\mu(\ann{\Lambda}) = 0$ as required.

Since the essential ideal  $\mfi{\I}_\mu$ is a depth ideal   and $\gamma_1 \notin \mfi{\I}$ it follows that at least one of the two conditions in Proposition \ref{prop:a condition for an ideal to be maximal of finite index} is satisfied. In either case it is possible to find a sequence of   elements $s_n \in \R$ such that $(s_n-s_m) \gamma_1 \notin \mfi{\I}_\mu$ for every pair of distinct indices $n,m \in \NN$.

Consider the elementary matrices $ g_n = \elm{1}{2} (s_n) \in \EL{d}{\R}$ for all $n \in \NN$. Fix a  pair of distinct indices $n,m \in \NN$.   It follows from  Lemma \ref{lem:subgroups generated by pairs} that
$$ \Gamma_{((s_n-s_m)\gamma_1)} \le g_n \Lambda + g_m \Lambda. $$
 Consider the ideal $\LL = \I_\mu + ((s_n-s_m)\gamma_1) \nrm \R$. We get
	$$ \Gamma_{\LL}   \le g_n\Lambda +g_m \Lambda +  \Gamma_{\I_\mu}. $$
	Passing to annihilators reverses the direction of the inclusion, namely
	$$ g_n \ann{\Lambda} \cap g_m \ann{\Lambda} \cap K_{\I_\mu}  \le K_{\LL}. $$
However $(s_n-s_m)\gamma_1 \notin \mfi{\I}_\mu$ and therefore  $\mfi{\I}_\mu \neq \mfi{\LL} $. In particular $\I_\mu \lneq \LL$ and so  $K_{\LL} \lneq K_{\I_\mu}$. Since $K_{\I_\mu}$ is the essential subgroup   with respect to the probability measure $\mu$   we have that
 $$ \mu(g_n \ann{\Lambda} \cap g_m \ann{\Lambda})  = \mu(g_n \ann{\Lambda} \cap g_m \ann{\Lambda} \cap K_{\I_\mu}) \le \mu(K_\LL) = 0.$$ This concludes the proof of the above claim.
\end{proof}

The next   result concerning $\EL{d}{\R}$-orbits is used to show that the probability measures that appear in Theorem \ref{thm:classification of invariant measures} are indeed ergodic. It is a generalization of a classical proof for the ergodicity of the Haar measure on the $d$-dimensional torus $(\mathbb{R}/\ZZ)^d$ with respect to the natural $\SL{d}{\ZZ}$-action, see e.g.   \cite[Example 2.1.4]{zimmer}.

\begin{prop}
	\label{prop:orbits are infinite in quotient by maximal finite index}
Let $\I \nrm \R$ be a depth ideal. 
	Then the $\EL{d}{\R}$-orbit of every non-trivial element of the quotient $\Gamma / \Gamma_\I$ is infinite.
\end{prop}
\begin{proof}
Let $\gamma + \Gamma_\I \in \Gamma / \Gamma_\I$ be any non-trivial element, namely $\gamma \notin \Gamma_\I$. In particular there is an index $ i \in \{1,\ldots,d\}$ such that the $i$-th coordinate  $\gamma_i$ of the element $\gamma$  satisfies $\gamma_i \notin \I$.  As $\I$ is a depth ideal, one of the two conditions in Proposition \ref{prop:a condition for an ideal to be maximal of finite index} is satisfied with respect to the element $\gamma_i$. In either case we may find a  sequence $s_n \in \R$ such that the elements $s_n\gamma_i \in \R$ are pairwise distinct modulo the ideal $\I$.
Fix an arbitrary index $j \in \{1,\ldots,d\} \setminus \{i\}$. The elements $\elm{j}{i}(s_n )   \gamma$ all belong to the $\EL{d}{\R}$-orbit of the element $\gamma$ and are pairwise distinct when considered in the quotient $\Gamma / \Gamma_\I$.
\end{proof}

\subsection*{Haar measure and the Fourier transform}
Let $T$ be an arbitrary locally compact abelian group. Let $A$ be the Pontryagin dual of $T$. We briefly recall a few generalities concerning the group $T$ and its dual $A$ to be used below.

The \emph{Fourier transform} of a complex Radon measure $\nu$ on the group $T$ is the continuous function $\widehat{\nu} : A \to \CC$ given by
$$ \widehat{\nu}(a) = \int_T a(t) \, \mathrm{d}\nu(t) \quad \forall a \in A.$$

\begin{prop}
	\label{prop:integral of Haar is delta}
 Let  $H $ be a closed subgroup of $T$. Let $\lambda$ be the Haar measure of $H$ regarded as a probability measure on the group $T$.   For every element $t \in T$ the Fourier transform  of the translate $t_* \lambda$ is given by
 $$ \widehat{t_* \lambda}(a) = \begin{cases}
 t(a) & a \in H^0\\
 0 & a \notin H^0.
 \end{cases}
 $$
\end{prop}
 \begin{proof}
Let $a \in A = \widehat{T}$ be any element with $a \notin \ann{H}$. We will show that $\widehat{ \lambda}(a) = 0$.
	There is some element $h_0 \in H$   with $a(h_0) \neq 1$. Therefore
$$	\widehat{\lambda}(a)   =      \int_H a(h) \mathrm{d} \lambda(h)  
	 =  \int_H a(h_0 h) \mathrm{d} \lambda(h) =   a(h_0) \widehat{\lambda}(a).
$$
	It follows that $\widehat{  \lambda}(a) = 0$ for the particular element $a \notin \ann{H}$ as above. 
		Fix any element $t \in T$. It is clear that $\widehat{t_* \lambda}  = t  \widehat{\lambda}$. The conclusion follows.
\end{proof}

In particular, the Fourier transform of the Haar measure of the subgroup $H$ regarded as a probability measure on the group $T$ is equal to the characteristic function of the annihilator $\ann{H}$.

The following elementary result concerns atomic probability measures on the group $T$.

\begin{prop}
	\label{prop:convolution has atoms if and only if both have atoms}
Let $\nu_1$ and $\nu_2$ be a pair of probability measures on the group $T$. The convolution $\nu_2 * \nu_2$ admits atoms if and only if both   $\nu_1$ and $\nu_2$ admit atoms.
\end{prop}
\begin{proof}
It follows from the definition of the convolution $\nu_1 * \nu_2$ that
	$$ (\nu_1 * \nu_2)(\{t\}) = (\nu_1\times \nu_2)(\{(t_1,t_2) \in T \times T \: : \: t_1 t_2 = t \}) $$
	for every element $t \in T$. Therefore if $\nu_1(\{t_1\}) > 0$ and $\nu_2(\{t_2\}) > 0$ for some pair of elements $t_1,t_2 \in T$ then $(\nu_1 * \nu_2)(\{t_1 t_2\}) > 0$ as well. Conversely, assume that $(\nu_1 * \nu_2)(\{t\}) > 0$ for some element $ t \in T$. According to Fubini's theorem
	$$   (\nu_1 * \nu_2)(\{t\}) = \int_T \nu_2(\{t_2 \in T \: : \: t_1t_2 = t\})\,  \mathrm{d}\nu_1(t_1)= \int_T \nu_2(\{t_1^{-1}t \}) \, \mathrm{d}\nu_1(t_1). $$
	This means that $\nu_2(\{t_1^{-1}t\}) > 0$ for $t_1 \in T$ belonging to  a set of positive $\nu_1$-measure. In particular both $\nu_1$ and $\nu_2$ admit atoms.
\end{proof}

\subsection*{Proof of the classification theorem}
We conclude the classification of $\EL{d}{\R}$-invariant probability measures on the compact group $K$.

\begin{prop}
\label{sub:proof of classification}	\label{prop:property W implies dual vanishes outside annihilator}
	Let $\mu$ be an ergodic $\EL{d}{\R}$-invariant  probability measure on the compact group $K$. Let $K_{\I_\mu}$ be the essential subgroup of $\mu$ and   $\Gamma_{\I_\mu}  $ be its annihilator. Then $\hat{\mu}(\gamma) =0 $ for all $\gamma \in \Gamma \setminus  \mfi{\Gamma}_{\I_\mu}$.
\end{prop}
\begin{proof}
	Assume towards contradiction that $\hat{\mu}(\gamma_0) = z \neq 0 $ for some $\gamma_0 \in \Gamma \setminus \mfi{\Gamma}_{\I_\mu}$.
Making use of  Proposition \ref{prop:orbits are sufficiently large} we find a subgroup
$\Lambda$ of the discrete group $\Gamma$ satisfying both $\gamma_0 + \Lambda \subset \EL{d}{\R} \gamma$ and $\mu(\ann{\Lambda}) = 0$.

Let $\nu$ be the complex Radon measure on the compact dual group $K$ given by the formula $\mathrm{d} \nu = \gamma_0 \mathrm{d}\mu$. The Fourier transform $\hat{\nu}$ of the complex measure $\nu$ satisfies
$ \hat{\nu}(\gamma) = \hat{\mu}(\gamma + \gamma_0)$
for all elements $\gamma \in \Gamma$. In particular $\hat{\nu}_{|\Lambda}   = z$. The two measures $\mu$ and $\nu$ are equivalent so that $\nu(\ann{\Lambda}) = 0$.
	

	
	Consider the quotient group $\overline{K} =  K/\ann{\Lambda}$ and the quotient map $\pi : K \to \overline{K}$. Let $\overline{\nu}$ denote the push-forward complex Radon measure $ \pi_* \nu$ on $\overline{K}$. The Pontryagin dual of  $\overline{K}$ is isomorphic to the subgroup $\Lambda$. In particular, the Fourier transform of $\overline{\nu}$ coincides with the restriction of   $\hat{\nu}$ to the subgroup $\Lambda$ and is therefore a constant function. The map taking a complex Radon measure to its Fourier transform is injective \cite[Proposition 4.17]{folland}. It follows that $\overline{\nu} = z \delta_{\ann{\Lambda}}$ where $\delta_{\ann{\Lambda}}$ is the atomic probability measure supported at the trivial coset $\ann{\Lambda}$ of the quotient $\overline{K}$. As $z \neq 0$ this contradicts the assumption that $\nu(\ann{\Lambda}) = 0$.
%
%
%
\end{proof}
%
%
%




The following   is inspired by   the classical proof of the corresponding fact for the $\SL{n}{\ZZ}$-action on the $n$-dimensional torus $(\mathbb{R}/\ZZ)^d$, see e.g. \cite[Example 2.1.4]{zimmer}.
\begin{prop}
	\label{prop:the invariant measures are ergodic}
Let $\I \nrm \R$ be a depth ideal and    $\omega$ be a finite orbit for the $\EL{d}{\R}$-action on the quotient $K/K_\I$.  Let $\nu_\I$ denote the Haar measure of the subgroup group $K_\I$ regarded as a probability measure on compact group $K$. Then the   probability measure $\mu_{\I,\omega} $ given by
	$$ \mu_{\I,\omega} = \frac{1}{\left|\omega \right|} \sum_{gK_\I \in \omega} g_* \nu_\I $$
is $\EL{d}{\R}$-invariant and ergodic.
\end{prop}
\begin{proof}
Assume to begin with that the orbit $\omega$ consists of the single coset $K_\I$ so that $|\omega| = 1$. In this case $\mu_{\I,\omega}$ is simply the Haar measure $\nu_\I$ of the compact group $K_\I$. The     $\EL{d}{\R}$-action on the group   $K_\I$ is by automorphisms and hence preserves the   measure $\nu_\I$. It remains to show that the measure $\nu_\I$ is ergodic for this action. Consider any $\EL{d}{\R}$-invariant function $f \in L^2(K_\I,\nu_\I)$. The  Fourier transform $\hat{f} \in L^2(\Gamma/\Gamma_\I)$ of the function $f$ is therefore $\EL{d}{\R}$-invariant as well. It follows from Proposition \ref{prop:orbits are infinite in quotient by maximal finite index} that $\hat{f}(\gamma) = 0$ for every non-trivial element $\gamma \in \Gamma/\Gamma_\I$. Therefore  the function $f$ is $\nu_\I$-almost surely constant, as required.
	
Consider the general case so that $\omega \subset K/K_\I$ is an arbitrary finite $\EL{d}{\R}$-orbit. The  invariance of the   probability measure $\nu_{\I,\omega}$ follows from the  invariance of the Haar measure $\nu_\I$ established in the above paragraph. Let $A$ denote the kernel of the $\EL{d}{\R}$-action on the finite orbit $\omega$ so that   $A$ is a finite index subgroup.
 To see that the probability measure $\nu_{\I,\omega}$ is $\EL{d}{\R}$-ergodic it will suffice to show that $A$ is ergodic on $g_*\nu_\I$ for any coset $gK_\I \in \omega$.

Fix a coset 	$gK_\I \in \omega$ and an element $k \in gK_\I$.  Let $T  : K \to K$ denote the translation $T  : g \mapsto g + k$. We see that
	$$ ag = a(g-k) + ak = T  a T ^{-1} + b (a)$$
	for all elements $a \in A$ and $g \in K$, where the
  map $b  : A \to  K_\I$ is given by
	$$ b (a) = ak  - k. $$
The translation $T$  can be used to identify the  trivial coset  $K_\I$ with the coset $gK_\I$. This will make the $A$-action on the coset  $gK_\I$   correspond to the $A$-action on the coset $K_\I$   followed by a translation. Note that Fourier transform takes   translation    to multiplication (by a complex number of modulus one). Since the non-trivial $A$-orbits on the group $\Gamma/\Gamma_\I$ are infinite by Proposition \ref{prop:orbits are infinite in quotient by maximal finite index}, the above argument goes through to show that the $A$-action on the coset  $gK_\I$ is ergodic.
\end{proof}

The following gives   the forward direction of the classification theorem.

\begin{prop}
\label{prop:one direction of classification}
	Let $\mu$ be an ergodic $\EL{d}{\R}$-invariant  probability measure on the compact group $K$ with essential ideal $\I_\mu$.  Then there is a   finite orbit $\omega$ for the $\EL{d}{\R}$-action on the quotient $K/K_{\mfi{\I}_\mu}$ such that
$$ \mu = \frac{1}{\left|\omega \right|} \sum_{g  \in \omega} g_* \nu  $$
where $\nu $ is the Haar measure on the compact group $K_{\mfi{\I_\mu}}$.
\end{prop}

\begin{proof}
Consider the   essential subgroup $K_{\I_\mu}$   of the compact group $K $ associated to the probability measure  $\mu$.  In particular $ \mu(K_{\I_\mu}) = 1$. 

The annihilator   of the essential subgroup $K_{\I_\mu}$ is given by the subgroup  $\Gamma_{\I_\mu}$. 
Therefore every character $\chi\in\Gamma_{\I_\mu}$ is $\mu$-almost surely trivial on the group  $K$. In particular
$$\hat{\mu}_{|\Gamma_{\I_\mu}} = 1$$
where $\hat{\mu} : \Gamma \to \CC$ is the Fourier transform  of the probability measure $\mu$.
In fact   $\gamma + \chi = \gamma$ holds $\mu$-almost surely  on $K$ for every other character $\gamma \in \Gamma$. Therefore the Fourier transform  $\hat{\mu}$ is constant on cosets of the subgroup $\Gamma_{\I_\mu}$.
On the other hand 
	$$ \hat{\mu}_{|(\mfi{\Gamma_{\I_\mu}})^c}   = 0.$$
as follows from Proposition \ref{prop:property W implies dual vanishes outside annihilator}.

	Consider the finite abelian quotient group $F = \mfi{\Gamma_{\I_\mu}}/\Gamma_{\I_\mu}$. The above observations show that the restriction of the Fourier transform $\hat{\mu}$ to the subgroup $ \mfi{\Gamma_{\I_\mu}}$ descends   to a well-defined positive definite function $\psi$ on the finite group $F$.     Bochner's theorem \cite[Theorem 4.18]{folland} implies that  the function $\psi$ is the Fourier transform of a uniquely determined probability measure $\theta$   on the Pontryagin dual $\hat{F} \cong K_{\I_\mu}/K_{\mfi{\I_\mu}}$. It follows that $\theta$ must be a uniform probability measure supported on some finite orbit $\omega \subset K_{\I_\mu}/K_{\mfi{\I_\mu}} \subset K/K_{\mfi{\I_\mu}}$.
	
Let $\nu_{\I,\omega}$ be the ergodic $\EL{d}{\R}$-invariant probability measure corresponding to the depth ideal $\I$ and to the finite orbit $\omega$. The previous paragraph combined with Proposition 	\ref{prop:integral of Haar is delta} and the injectivity of the Fourier transform for probability measures show  that $\mu = \nu_{\I,\omega}$.
\end{proof}

The two Propositions 	\ref{prop:the invariant measures are ergodic} and \ref{prop:one direction of classification} establish the two directions of the correspondence between ergodic $\EL{d}{\R}$-invariant probability measures on the compact group $K$ and pairs $(\I, \omega)$ as stated in the classification Theorem \ref{thm:classification of invariant measures}. It is easy to verify that this correspondence is bijective. The proof of  Theorem  \ref{thm:measure classification intro} of the introduction is complete.


\section{Theory of characters}
\label{sec: character theory}

Let $G$ be a countable   group.
A function $\varphi : G \to \CC$ is \emph{positive definite} if for every $n \in \NN$ and every choice of elements $g_1,\ldots,g_n \in G$ the $n\times n$ matrix with entries given by $\varphi(g_j^{-1} g_i) $ is positive definite.

A \emph{trace} on the group $G$ is a positive definite conjugation invariant function $\varphi \colon G \to \CC$ satisfying  $\varphi(e) = 1$. A trace $\varphi$ is said to be \emph{irreducible} if $\varphi$ cannot be written as a non-trivial convex combination of traces.
	A \emph{character} on the group $G$ is an irreducible trace.

Let $\traces{G}$ and $\chars{G}$ respectively  denote the set of all traces and characters on the group $G$. 
The set   $\traces{G}$ is a weak-$*$ closed convex subset of $L^\infty(G)$  \cite[Lemma C.5.4]{bekka2008kazhdan}. The subset $\chars{G}$ can be identified with the set of extreme points of $\traces{G}$. Choquet's theorem \cite[\S10]{phelps1911lectures} says that every trace $\varphi \in \traces{G}$ can be written as an integral
$$ \varphi = \int_{\chars{G}} \psi \, \mathrm{d} \mu_\varphi(\psi) $$
for some \emph{uniquely determined} Borel probability measure $\mu_\varphi$ on $\chars{G}$. Conversely, any Borel probability measure on the set $\chars{G}$ of extreme points  determines a trace of $G$ via the above formula.

\subsection*{The GNS construction}

Positive definite functions on the group $G$ are closely related to cyclic unitary representations of $G$.

\begin{theorem}[Gelfand--Naimark--Segal] \label {thm: GNS}
 Let $\varphi$ be a positive definite function on the group $G$ with $\varphi(e) = 1$. Then there is a unitary representation $\pi_\varphi : G \to \mathcal{U}(\mathcal{H}_\varphi)$ on a Hilbert space $\mathcal{H}_{\varphi}$ admitting a cyclic unit vector $v_\varphi \in \mathcal{H}_\varphi$ such that
 $$\varphi (g) = \langle \pi_{\varphi}(g) v_{\varphi},v_{\varphi} \rangle $$
 for every element $g \in G$. The triplet $(\mathcal{H}_{\varphi}, \pi_\varphi, v_\varphi)$ is uniquely determined up to a unitary equivalence.
 \end{theorem}

The triplet $(\mathcal{H}_{\varphi}, \pi_\varphi, v_\varphi)$ appearing in the above theorem is called the Gelfand--Naimark--Segal (GNS) construction associated the positive definite function $\varphi$ on the group $G$. 

The representation $\pi_\varphi$ on the Hilbert space $\mathcal{H}_\varphi$ is associated to a certain "left action" in the GNS construction. In fact, since $\varphi$ is conjugation invariant, there exists another unitary representation $\rho_\varphi$ on the same Hilbert space $\mathcal{H}_\varphi$ arising from a "right action" in the GNS construction. The representation $\rho_\varphi$ commutes with $\pi_\varphi$ and likewise satisfies
 $\varphi (g) = \langle \rho_{\varphi}(g) v_{\varphi},v_{\varphi} \rangle $ for all elements $g \in G$.


\begin{prop}
\label{prop:finite dimensional implies trace} 
Let $\varphi \in \traces{G}$ be a trace with GNS triplet $(\mathcal{H}_\varphi, \pi_\varphi, v_\varphi)$. If the Hilbert space $ \mathcal{H}_\varphi $ is finite dimensional then 
$$ \varphi  = \frac{ \mathrm{tr} \pi_\varphi  }{\dim_\CC \mathcal{H}_\varphi}.$$
\end{prop}
\begin{proof}
 The conjugation invariance of the character $\varphi$ implies that
 $$ \varphi(g) = \left< \pi_\varphi(g) \pi_\varphi(h) v_\varphi, \pi_\varphi(h) v_\varphi \right>$$
 holds true  for all elements $g,h \in G$. Denote $d= \dim_\CC \mathcal{H}_\varphi$. Let $h_{1},\ldots,h_{d} \in G$ be group elements such that the  vectors 
 $$\pi_\varphi(h_1) v_\varphi, \ldots, \pi_\varphi(h_d) v_\varphi \in \mathcal{H}_\varphi$$ form a basis for the Hilbert space $\mathcal{H}_\varphi$. Then
 $$ \varphi(g) = \frac{1}{d} \sum_{i=1}^d \left< \pi_\varphi(g) \pi_\varphi(h_i) v_\varphi, \pi_\varphi(h_i) v_\varphi \right> = \frac{\mathrm{tr} \pi_\varphi(g) }{\dim_\CC \mathcal{H}_\varphi} $$
 for all elements $g \in G$  as required.
\end{proof}

The converse to Proposition \ref{prop:finite dimensional implies trace} is also true, namely if $\pi : G \to \mathcal{U}(\mathcal{H})$ is any unitary representation   in a finite dimensional Hilbert space then the function  
$$ \varphi  = \frac{\mathrm{tr}\pi }{\dim_\CC \mathcal{H}}$$
is a trace on the group $G$.

%
%
\begin{prop}
\label{prop:square power is a trace}
If $\varphi \in \traces{G}$ then $|\varphi|^2 \in \traces{G}$.  
\end{prop}
This is a well-known fact, see e.g. \cite[Proposition C.1.6]{bekka2008kazhdan}.
\begin{proof}[Proof of Proposition \ref{prop:square power is a trace}]
Let $\varphi \in \traces{G}$ be a trace on the group $G$. It is clear that the function $|\varphi|^2 : G \to \CC$ is conjugation invariant and satisfies $|\varphi|^2(e) = 1$. To see that $|\varphi|^2$ is positive definite consider the GNS construction $(\pi_\varphi, \mathcal{H}_\varphi, v_\varphi)$ associated to the trace $\varphi$. Let $\pi^*_\varphi$ be the contragredient representation associated to the unitary representation $\pi_\varphi$. Then
$$ |\varphi|^2(g) = \left< \pi_\varphi(g) v_\varphi \otimes \pi^*_\varphi(g)  v_\varphi, v_\varphi \otimes v_\varphi \right> \quad \forall g \in G.$$
Therefore $|\varphi|^2$ is positive definite. We conclude that $|\varphi|^2 \in \traces{G}$.
\end{proof}

\subsection*{Faithful traces}

 The \emph{kernel} of a given trace  $\varphi \in \traces{G}$ is
$$ \ker \varphi = \{ g \in G \: : \: \varphi(g) = 1 \}. $$
A trace $\varphi \in \traces{G}$ is called \emph{faithful} if $\ker \varphi = \{e\}$.

\begin{prop}
\label{prop:trace factors through the kernel}
If $\varphi  \in \traces{G}$   then    $\ker \varphi \nrm G$ and $\overline{\varphi} : g \ker\varphi \mapsto  \varphi(g )$ defines a    faithful trace on the quotient group   $\bar{G} = G/\ker \varphi$. In fact $\ker \varphi = \ker \pi_\varphi$.
\end{prop}
\begin{proof}
Consider the GNS construction $(\mathcal{H}_{\varphi}, \pi_\varphi, v_\varphi)$ associated to the trace $\varphi$. An element $g \in G$ satisfies $g \in \ker \varphi$ if and only if $\pi_\varphi(g) v_\varphi = v_\varphi$. It follows that $\ker \varphi$ is a subgroup of $G$. The conjugation invariance of the trace $\varphi$  implies that the subgroup $\ker \varphi$ is normal. The formula $\varphi (g) = \langle \pi_{\varphi}(g) v_{\varphi},v_{\varphi} \rangle$ shows that $\varphi(g) = \varphi(gn)$ holds true for every pair of elements $g \in G$ and $n \in \ker \varphi$. So the function $\overline{\varphi}$ is well-defined. It is clear that $\overline{\varphi}$ is conjugation invariant, positive definite and satisfies $\overline{\varphi}(\ker \varphi) = 1$. Lastly, as the vector $v_\varphi$ is   cyclic   and the character $\varphi$ is conjugation invariant it follows in fact that  $\ker \varphi = \ker \pi_\varphi$.
\end{proof}


 \subsection*{von Neumann algebras and characters}

 Let $\mathcal{H}$ be a Hilbert space and $B(\mathcal{H})$ be the algebra of all bounded operators on $\mathcal{H}$.
  A \emph{von Neumann algebra} $\mathcal{L}$ on the Hilbert space $\mathcal{H}$  is  subalgebra of $B(\mathcal{H})$  that  contains the identity, is closed in the weak operator topology and  under taking adjoint.

The   von Nuemann bicommutant theorem says that a subalgebra $\mathcal{L}$ of $B(\mathcal{H})$ containing the identity and closed under taking adjoints is a von Neumann algebra if and only if $\mathcal{L}$ is equal to its double commutant $\mathcal{L}''$.

  A von Neumann algebra $\mathcal{L}$ on the Hilbert space $\mathcal{H}$ is called a \emph{factor} if its center $Z(\mathcal{L})$ is equal to  $\CC$.

Let  $\varphi \in \traces{G}$ be a trace   on a group $G$ with    GNS construction   $(\mathcal{H}_\varphi, \pi_\varphi, v_\varphi)$. The trace $\varphi$ defines a von Neumann algebra $\mathcal{L}_{\varphi}$ on the Hilbert space $\mathcal{H}_\varphi$. The von Neumann algebra $\mathcal{L}_{\varphi}$  is defined to be  the closure of the algebra generated by the unitary operators $\{\pi_\varphi(g) \: : \: g \in G\}$  in the weak operator topology. The trace  $\varphi$ is a character if and only if the von Neumann algebra $\mathcal{L}_{\varphi}$ is a factor \cite{thoma1964unitare}. 

We remark that the commutant $\mathcal{L}'_\varphi$ coincides with the von Neumann algebra generated by $\{\rho_\varphi(g) \: : \: g \in G\}$. In particular an operator $x \in B(\mathcal{H}_\varphi)$ belongs to $Z(\mathcal{L}_\varphi)$  if and only if $x$ commutes with $\pi_\varphi(g)$ and $\rho_\varphi(g)$ for all elements $g \in G$.

Consider the kernel   $\ker \varphi$ of the trace $\varphi \in \traces{G}$.  Let $\overline{\varphi}$ be the faithful trace on the quotient  group $G/\ker \varphi$   as in Proposition \ref{prop:trace factors through the kernel}.
Since $\ker \varphi = \ker \pi_\varphi$  the two von Neumann algebras $\mathcal{L}_{\varphi}$ and $\mathcal{L}_{\bar{\varphi}}$ and isomorphic. In particular $\varphi$ is a character if and only if $\overline{\varphi}$ is.

%
%
%

\begin{lemma}[Schur's lemma]
\label{lem:Schur's lemma for characters}
  If $\varphi \in \chars{G}$ is a character then the restriction $\varphi_{|Z(G)}  $   is a multiplicative character of the center $Z(G)$ and
  $$ \varphi(gz) = \varphi(g) \varphi(z)$$
for every pair of elements $g \in G$ and $z \in Z(G)$.
\end{lemma}
\begin{proof}
Let $\varphi \in \chars{G}$ be a character with a corresponding GNS construction $(\mathcal{H}_\varphi, \pi_\varphi, v_\varphi)$. The von Neumann algebra $\mathcal{L}_{\varphi}$ on the Hilbert space $\mathcal{H}_\varphi$ is a factor. For every central element $z \in Z(G)$ the operator $\pi_\varphi(z)$ lies in the center $Z(\mathcal{L}_\varphi)$ of the von Neumann algebra $\mathcal{L}_{\varphi}$. It follows that $\pi_\varphi(z) v_\varphi = \chi(z) v_\varphi$ for some multiplicative character $\chi : Z(G) \to S^1$ and all elements $z \in Z(G)$, as required. Finally let  $g \in G$ and  $z \in Z(G)$ be a pair of elements. Then
$$ \varphi(gz) = \left< \pi_\varphi(gz) v_\varphi, v_\varphi \right> = \left< \pi_\varphi(g)  \pi_\varphi(z) v_\varphi, v_\varphi \right>= \chi(z)\left< \pi_\varphi(g) v_\varphi, v_\varphi \right>= \varphi(z) \varphi(g)$$
as required.
\end{proof}

\begin{cor}
\label{cor:center in kernel of square}
If $\varphi \in \chars{G}$ then $Z(G) \le \ker |\varphi|^2$.
\end{cor}
\begin{proof}
Let $\varphi \in \chars{G}$ be a character. Then $|\varphi|^2 \in \traces{G}$ by Proposition \ref{prop:square power is a trace}    and $|\varphi|^2_{|Z(G)} = 1$ from Lemma \ref{lem:Schur's lemma for characters}.

\end{proof}





\subsection*{Relative characters}
Let $N$ be a normal subgroup of the  group $G$.  The set of \emph{relative traces} on the subgroup $N$ is given by
$$\tracerel{G}{N} = \{ \varphi \in \traces{N} \: : \: \varphi(n^g) = \varphi(n) \quad \forall n \in N, g \in G \}. $$
The set of  \emph{relative characters}  $\charrel{G}{N}$ consists of those relative traces that cannot be written as a non-trivial convex combination of   relative traces.

The restriction of any character $\varphi \in \chars{G}$ to the normal subgroup $N$ results in a relative character $\varphi_{|N} \in \charrel{G}{N}$  \cite[Lemma 14]{thoma1964unitare}. Conversely, given any relative character $\psi \in \charrel{G}{N}$  there exists a character $\varphi \in \chars{G}$ such that $\varphi_{|N} = \psi$  \cite[Lemma 16]{thoma1964unitare}.

%
%
%

Recall that Choquet's theorem sets up a bijective correspondence between the set $\traces{N}$ of all traces on the group $N$ and the set $\mathcal{P}(\chars{N})$ of all Borel probability measures on the set of extreme points $\chars{N}$. 
This correspondence is equivariant with respect to the natural $G$-actions on both sets. It follows that the subset 
   $\tracerel{G}{N}$ of relative traces corresponds to $G$-invariant probability measures, and furthermore the subset $\charrel{G}{N}$ of relative characters corresponds to   ergodic ones.
   
\begin{prop}
\label{prop:extension by 0 is a trace}
If $\psi \in \tracerel{G}{N}$ is a relative trace then the function $\hat{\psi} : G \to \CC $ given by
$$\hat{\psi}(g) = \begin{cases} \psi(g) & g\in N\\ 0 & g \notin N \end{cases}$$
belongs to $\traces{G}$.
\end{prop}
\begin{proof}
Consider the GNS construction $(\mathcal{H}_{\psi}, \pi_\psi, v_\psi)$ corresponding to the relative character $\psi$ so that $\pi_\psi$ is a unitary representation of the normal subgroup $N$. Let $\widehat{\pi}_\psi = \mathrm{Ind}_N^G(\pi_\psi)$ be the induced unitary representation of the group $G$. The Hilbert space $\hat{\mathcal{H}}_\psi$ on which the representation $\widehat{\pi}_\psi$ acts is given by the orthogonal direct sum
$$ \hat{\mathcal{H}}_\psi = \bigoplus_{  G/N} \mathcal{H}_\psi.$$
It is easy to see that  the function $\hat{\psi}$ coincides with the matrix coefficient of the vector $v_\psi$ in the unitary representation $ \widehat{\pi}_\psi $ and as such is positive definite. The function $\hat{\psi}$ is clearly conjugation invariant and satisfies $\hat{\psi}(e) = 1$ as required.
\end{proof}

%
%
%
%

\subsection*{Virtually central subgroups}

A subgroup $H$ of $G$ is   \emph{virtually central in $G$} if $\left[H : H \cap Z(G) \right] < \infty$. 
Let $N$ be a virtually central normal subgroup of   $G$. Denote 
$$Z = N \cap Z(G) \quad \text{and} \quad M = N / Z$$
so that $|M| < \infty$ by definition.

\begin{prop}
\label{prop:finitely many characters with a given multiplicative character}
Let $\chi$  be a multiplicative character of the central group $Z$. Then
\begin{enumerate}
\item there are only finitely many characters $\psi \in \chars{N}$ with $\psi_{|Z} = \chi$, and
\item the GNS construction associated to every such character $\psi \in \chars{N}$ is finite dimensional.
\end{enumerate}
\end{prop}

The proof   is based on well-known facts concerning projective representations. We outline the proof for the reader's convenience.

\begin{proof}[Proof of Proposition \ref{prop:finitely many characters with a given multiplicative character}]
Fix an arbitrary section $\lambda : M \to N$  with $\lambda(e_M) = e_N$. The function $c_\chi : M \times M \to \CC$ given by
$$ c_{\chi}(m_1,m_2) = \chi( \lambda(m_1) \lambda(m_2) \lambda ( m_1 m_2 )^{-1} ) \quad \forall m_1,m_2 \in M$$
  is a $2$-cocycle \cite[\S2]{karpilovsky1985projective}. Recall that a projective $c_\chi$-representation of the quotient group $M$ on a given vector space $V$ is a map $\bar{\pi} : M \to \mathrm{GL}(V)$ satisfying
  $$ \bar{\pi}(m_1) \bar{\pi}(m_2) = c_\chi(m_1,m_2) \bar{\pi}(m_1m_2)$$
for all elements $m_1, m_2 \in M$. 
  The set of all representations     of the group $N$ restricting to the multiplicative character $\chi$ on the subgroup $Z$ bijectively corresponds to the set of all  projective $c_\chi$-representations   of the quotient group $M$. The latter set stands in a bijective correspondence with  the set of all modules over the twisted group algebra $A_\chi = \CC^{c_\chi} M$  \cite[Theorem 3.2.5]{karpilovsky1985projective}. As the   $\CC$-algebra $A_\chi$ is semisimple \cite[Theorem 3.2.10]{karpilovsky1985projective}  every irreducible projective $c_\chi$-representation is finite dimensional and there are only finitely many isomorphism classes of such representations.

Consider a character $\psi \in \chars{N}$ with $\psi_{|Z} = \chi$ and with GNS triplet $(\mathcal{H}_\psi, \pi_\psi,  v_\psi)$. Let $\bar{\pi}_\psi$ be the projective $c_\chi$-representation  of the finite quotient group $M$ corresponding to the representation $\pi_\psi$ of the group $N$. Therefore $\mathcal{H}_\psi$ can be regarded as a cyclic module over the twisted group algebra $A_\chi$. Since $A_\chi$ is semisimple it follows  that $\mathcal{H}_\psi$ decomposes as a direct sum of non-isomorphic simple $A_\chi$-modules. However $\psi$ is a character so that the von Neumann algebra $\mathcal{L}_\psi$ is a factor. This algebra is finite dimensional and coincides with $\bar{\pi}_\psi(A_\chi)$. We conclude that $\mathcal{H}_\psi$ is a simple $A_\chi$-module. In other words $\bar{\pi}_\psi$  is irreducible (and so is $\pi_\psi$). The desired conclusion  follows from this.
\end{proof}

%
%
%

We remark that, as with finite groups, characters of the group $N$ restricting to a given multiplicative character $\chi$ of the central subgroup $Z$ bijectively correspond to irreducible projective $c_\chi$-representations of the finite group $M$ over the $2$-cocycle $c_\chi$ corresponding to $\chi$   \cite[Theorem 7.1.11]{karpilovsky1985projective}.

\begin{prop}
\label{prop:relative character of virtually central}
If  $\psi \in \charrel{G}{N}$ is a relative character then there is a finite $G$-orbit $O_\psi \subset \chars{N}$ such that
 $$ \psi = \frac{1}{|O_\psi|} \sum_{\zeta \in O_\psi} \zeta.$$
\end{prop}
\begin{proof}
The restriction of the character $\psi$ to the central subgroup $Z = N \cap Z(G)$ is given by some multiplicative character $\chi : Z  \to \CC$, see Lemma   \ref{lem:Schur's lemma for characters}. There is an ergodic $G$-invariant probability measure $\mu_\psi$   on the space $\chars{N}$ such that $\psi = \int \zeta \, \mathrm{d} \mu_\psi(\zeta)$.
It is clear that the probability measure $\mu_\psi$ is supported  on the Borel subset
$$ \mathrm{Ch}_\chi(N)  = \{\psi \in \chars{N} \: : \: \psi_{|Z} = \chi\}.$$ 
However the subset $ \mathrm{Ch}_\chi(N)$  is finite by Proposition \ref{prop:finitely many characters with a given multiplicative character}. Therefore $\mu_\psi$ must be a uniform probability measure supported on some finite $G$-orbit $O_\chi \subset \mathrm{Ch}_\chi(N)$. The conclusion follows.
\end{proof}

\begin{cor}
\label{cor:every irreducible rep of virtually abelian is finite dimensional}
Let $\psi \in \charrel{G}{N}$ be a relative character. Then there exists a finite dimensional space $V_\psi$ and a representation  $\pi_\psi : N \to \mathrm{GL}(V)$ satisfying 
$$\psi(n) = \frac{\mathrm{tr}\pi_\psi(n)}{\dim_\CC V } $$ for all elements $n \in N$.
\end{cor}

\begin{proof}
Let $O_\psi \subset \chars{N}$ be the finite $G$-orbit   given by Proposition \ref{prop:relative character of virtually central}.  Let $(\mathcal{H}_\zeta, \pi_\zeta, v_\zeta)$ be the GNS construction associated to each character $\zeta \in O_\psi$. As $G$ is  transitive on the orbit $O_\psi$ the dimension $d = \dim_\CC \mathcal{H}_\zeta$ is independent of   the choice of the character $\zeta \in O_\psi$.  Moreover   $ d < \infty$  according to Proposition \ref{prop:finitely many characters with a given multiplicative character}.  Lastly 
 $$ \zeta(n) = \frac{\mathrm{tr} \pi_\zeta (n)}{d}  $$
holds true for all elements $n \in N$ and all characters $\zeta \in O_\psi$, see Proposition \ref{prop:finite dimensional implies trace}. To conclude the proof consider the representation
 $\pi_\psi = \bigoplus_{\zeta \in O_\psi} \pi_\zeta$ on the finite dimensional Hilbert space $V = \bigoplus_{\zeta \in O_\psi} \mathcal{H}_\zeta$.   
\end{proof}

\begin{remark}
A virtually central group is clearly virtually abelian. However, the center of a virtually abelian group need not be of finite index.  For example, the infinite dihedral group is virtually abelian and center free.
\end{remark}

A discrete group is \emph{type I} if and only if it is virtually abelian \cite{thoma1964unitare}. In particular a virtually central subgroup is type I.

\subsection*{Induced traces}

Let $H$ and $N$ be subgroups of $G$ with $N \le H \le G$.  A trace $\varphi \in \traces{H}$ is \emph{induced  from the subgroup $N$} if $\varphi(g) = 0$ for all elements $g \in H \setminus N$, or what is equivalent\footnote{See Proposition \ref{prop:extension by 0 is a trace} and its proof for more details on this equivalence.
}, the  corresponding GNS construction $\pi_\varphi : H \to \mathcal{U}(\mathcal{H}_\varphi)$ is induced in the representation-theoretic sense from a unitary representation of $N$. 
Denote
\begin{equation*}
\label{def Ind}
 \Ind{G}{H}{L} = \{ \varphi \in \traces{G} \: : \: \varphi(g) = 0 \quad \forall g \in H \setminus N\}.
 \end{equation*}
In other words $ \Ind{G}{H}{L}$ consists of those traces $\varphi \in \traces{G}$ whose  restriction $\varphi_{|H}$ is induced from the subgroup $L$. 

The following results are  based to a large extent on    \cite{bekka}.


%

\begin{lemma}
	\label{lem:sequence of elements different mod N goes weak-star to zero}
Let $\varphi \in \Ind{G}{H}{N}$ be a trace with GNS construction $(\pi_\varphi, \mathcal{H}_\varphi, v_\varphi)$. 
If  $h_n \in H$ is a sequence of elements belonging to pairwise distinct  left cosets of the subgroup $N$   then the     unit vectors $\pi_\varphi(h_n)v_\varphi$ weak-$*$ converge to $0$  in  $\mathcal{H}_\varphi$.
\end{lemma}


\begin{proof}
We claim that the unit vectors $\pi_\varphi(h_n) v_\varphi$ are pairwise orthogonal, as can be seen by   observing that
$$ \left< \pi_\varphi(h_n) v_\varphi, \pi_\varphi(h_m) v_\varphi \right> = \left< \pi_\varphi(h_m ^{-1} h_n) v_\varphi,  v_\varphi \right> = \varphi(h_m^{-1}h_n) = 0$$
for all distinct indices $n,m \in \NN$. The result follows from Bessel's inequality.
\end{proof}

Another way to see why Lemma \ref{lem:sequence of elements different mod N goes weak-star to zero} is true would be using the fact that the unitary representation $\pi_\varphi$ is induced   from the subgroup $N$.


\begin{lemma}
	\label{lem:application of lemma on sequences of elements for vanishing}
Let $\varphi \in \Ind{G}{H}{N}$ be a trace and $g \in G$ be an element.
		If   there is a sequence of elements  $x_n \in G$  such that the commutators  $\left[g,x_n\right] $ belong to pairwise distinct cosets of $N$  and  such that
	$\left[g,x_n\right] \in H$ for all $n \in \NN$
	then $\varphi(g) = 0 $.
\end{lemma}

\begin{proof}
The conjugation invariance of the trace $\varphi$ implies that
$$ \varphi(g) = \varphi(x_n^{-1} g x_n ) = \varphi(g\left[g,x_n\right]) = \left< \pi_\varphi(\left[g,x_n\right])v_\varphi, \pi_\varphi(g)^{-1}v_\varphi \right> $$
for all indices $n \in \NN$. The sequence of vectors $\pi_\varphi(\left[g,x_n\right])v_\varphi  $   weak-$*$    converges    to the zero vector of the Hilbert space $\mathcal{H}_\varphi$   by Lemma \ref{lem:sequence of elements different mod N goes weak-star to zero}. We conclude that  $\varphi(g) = 0$.
\end{proof}


\begin{lemma}
\label{lem:vanishing with two subgroups}
Let $\varphi \in \Ind{G}{H}{N}$ be a trace. Let $K$ and $L$ be subgroups of $G$ and   $g \in KHL$ be an element. If there is a sequence of elements $x_n \in \mathrm{N}_G(H) $ so that
\begin{itemize}
	\item $\left[g,x_n\right]$ are pairwise distinct modulo $N$ and
	\item $\left[K,x_n\right] \subset H$ and $\left[L,x_n\right] \subset H$ for every $n \in \NN$
\end{itemize}
 then $\varphi(g) = 0$.
\end{lemma}
\begin{proof}
Say that $g = khl$ for some elements $k \in K, h \in H$ and $l \in L$. The conjugates $x_n^{-1} g x_n$ can be written as
$$ x_n^{-1} g x_n =  k \left[k,x_n \right] \cdot x_n^{-1}h x_n \cdot \left[x_n, l^{-1}\right] l.$$
We denote
$$y_n = \left[k,x_n \right] \cdot x_n^{-1}h x_n \cdot \left[x_n, l^{-1}\right] $$
for every $n \in \NN$.   The assumptions   imply  that $y_n \in H$. The computation
$$ l^{-1} y_n l =  (kl)^{-1} ky_n l = (kl)^{-1}g \left[g,x_n\right]$$
shows that the elements $y_n$ are pairwise distinct modulo $N$. Therefore
$$\varphi(g) = \varphi(lx_n^{-1}g x_n l^{-1}) = \varphi(lk y_n) = \left<\pi_\varphi(y_n) v_\varphi, \pi_\varphi((lk)^{-1}) v_\varphi \right> $$
The sequence of vectors $\pi_\varphi(y_n)v_\varphi  $ converges to the zero vector  of the Hilbert space  $\mathcal
{H}_\varphi$ in the weak-$*$ topology  by Lemma \ref{lem:sequence of elements different mod N goes weak-star to zero}.  Therefore $\varphi(g) = 0$.
\end{proof}

\begin{remark}
For our purposes it will be enough to use a simpler variant of Lemma \ref{lem:vanishing with two subgroups} involving the product of just two subgroups rather than three. We have decided to include the above slightly  more  general variant nevertheless, as it might be of independent interest.
\end{remark}

\subsection*{Characters of two-step nilpotent groups}

Let $Z^2(G)$ denote the second center (i.e. the   second term in the ascending central sequence) of the countable   group $G$.   This is a characteristic subgroup containing the center $Z(G)$  and  determined by
 $$Z^2(G)/Z(G) = Z(G/Z(G)) \le G/Z(G). $$


\begin{lemma}[Howe \cite{howe1977representations}]
\label{lem:Howe}
If  $\varphi \in \chars{G}$ is a faithful character then 
$$\varphi \in \Ind{G}{Z^2(G)}{Z(G)}.$$
\end{lemma}
\begin{proof}
Consider any element $g \in Z^2(G) \setminus Z(G)$. There exists an element $h \in G$ such that $e \neq \left[g,h\right] \in Z(G)$. The conjugation invariance of  the character $\varphi$    combined with Schur's lemma (reproduced above as Lemma \ref{lem:Schur's lemma for characters}) shows that
$$ \varphi(g) = \varphi(g^h) = \varphi(g[g,h]) = \varphi(g) \varphi([g,h]).$$
Since $\varphi$ is faithful $\varphi([g,h]) \neq 1$. We conclude that  $\varphi(g) = 0$.
\end{proof}

The following is a variant of Howe's lemma for relative characters of two-step nilpotent subgroups. Its proof is essentially the same as that of Lemma \ref{lem:Howe}.

\begin{lemma}
\label{lem:relative Howe}
Let $N $ be a  normal subgroup of $G$ with $\left[N,N\right] \le Z(G)$. Then every faithful relative character   $\varphi \in \charrel{G}{N}$  is induced from $Z(N)$.
\end{lemma}

\section{Normal subgroups of the group $\EL{d}{\R}$}
\label{sec: structure of SLd}
Let $\R$ be a Noetherian ring. Let    $d \in \NN$ be a fixed integer with $d \ge 3$. 



\subsection*{Stable range}
The stable range of a ring is an important invariant in algebraic $K$-theory \cite[Chapter V.\S3]{bass1968algebraic}. It controls various stability phenomena, such as the results cited in Theorems  \ref{thm:stability in K1} and \ref{thm:relative stability} below. 
We will use the explicit definition of stable range to prove the normal form decomposition in    Proposition \ref{prop:normal form for conjugates} as well as some general stability results. 

Stable range is defined as follows. An $n$-tuple of   elements     $r_1,\ldots,r_n \in \R$ is called \emph{unimodular} if
	$$ \R r_1 + \cdots + \R r_n = \R.$$
	We say that an $(n+1)$-tuple of elements $r_1,\ldots, r_{n+1} \in \R$ \emph{can be shortened} if there are elements $s_1,\ldots,s_n\in \R$ such that the $n$-tuple
		$$ r_1 + s_1 r_{n+1}, \, r_2 + s_2r_{n+1}, \ldots, \, r_n + s_n r_{n+1}$$
is unimodular.
 The \emph{stable range} of the ring $\R$ is  defined to be the smallest integer $n \in \NN$ such that   any  unimodular $(n+1)$-tuple of elements can be shortened. It is denoted $\sr{\R}$. If no such   $n$  exists then $\sr{\R} = \infty$.

For example, if $\R$ is   a local ring then $\sr{\R} = 1$.  If $\R$ is  a Dedekind domain then $\sr{\R} \le 2$. If $\R$ is a Noetherian ring then $\sr{\R} + 1$ is bounded above by the Krull dimension of $\R$. In particular $\sr{\ZZ\left[x_1,\ldots,x_k\right]} \le k+2$. Finally $\sr{\R/\I} \le \sr{\R}$ for any ideal $\I \nrm \R$. 

For references see \cite[Propositions V.3.2 and V.3.4, Theorem V.3.5]{bass1968algebraic}. For the case of Dedekind domains see e.g. \cite[Example 12.1.14]{chen2011rings}.

\subsection*{Elementary matrices}
The   \emph{elementary matrix} $\elm{i}{j}(x)\in \SL{d}{\R}$ corresponding to the ring element  $x \in \R$ and   the pair of distinct indices $i,j \in \{1,\ldots,d\}$ is given by
$$ (\elm{i}{j}(x))_{k,l} = \begin{cases}  x & \text{if $i=k$ and $j=l$,}\\ 1 & \text{if $k = l$,} \\ 0 & \text{otherwise}.  \end{cases} 
$$
Note that 
$$\elm{i}{j}(x_1) \elm{i}{j}(x_2) = \elm{i}{j}(x_1 + x_2) \quad \forall x_1,x_2 \in \R.$$

 The \emph{elementary subgroup} $\elm{i}{j}(\I)$ of the matrix group  $\SL{d}{\R}$ corresponding to the ideal   $\I \nrm \R$  is given by
 $$ \elm{i}{j}(\I) = \{ \elm{i}{j}(x) \: : \: x \in \I \}.$$
The   group $\elm{i}{j}(\I)$ is naturally isomorphic to the additive group of the ideal $\I$. 
 The   elementary groups  satisfy the following commutation relations
$$ \text{$\left[\elm{i}{j}(\R),\elm{i}{j'}(\R)\right] = \mathrm{Id}_d$  \quad and \quad $\left[\elm{i}{j}(\R),\elm{i'}{j}(\R)\right] = \mathrm{Id}_d$} $$
for every pair of indices $i',j'  \in \{1,\ldots,d\} \setminus \{i,j\}$.

Let $\EL{d}{\R}$ be the subgroup of the matrix group $\SL{d}{\R}$ generated by  the elementary subgroups $\elm{i}{j}(\R)$ for all pairs of distinct indices $i,j \in \{1,\ldots,d\}$. 

\begin{theorem}[\cite{bass1964k}]
\label{thm:stability in K1}
If $ d > \sr{\R}$ then  $\EL{d}{\R} \nrm \SL{d}{\R}$ and the quotient group   $\mathrm{SK}_1(\R) = \SL{d}{\R} / \EL{d}{\R}$ is abelian and   independent of $d$.\footnote{As the ring $\R$ is assumed to be commutative,  the subgroup $\EL{d}{\R}$ is in fact normal in $\SL{d}{\R}$ for all $d \ge 3$ \cite{suslin1977structure}.
}
\end{theorem}

Several cases where $\mathrm{SK}_1(\R)$ is trivial were mentioned in Example (\ref{ex:on K1}) of \S\ref{sec:intro}.

\subsection*{The center of $\SL{d}{\R}$}
Let $\units{\R}$ denote the group of units of the   ring $\R$. The group  $\units{\R}$ is clearly abelian. However $\units{\R}$ need not be finitely generated, even if $\R$ is. For example, if $\R = \mathbb{F}_2[x,y]/(y^2)$ then every element of the form $1 + y p(x)$ for some polynomial $p \in \mathbb{F}_2[x]$ is a unit.

Let $\unitsord{\R}{d}$ denote the subgroup of the group of units $\units{\R}$ consisting of elements whose order is a divisor of $d$. In other words
$$ \unitsord{\R}{d} = \left\{u \in \units{\R} \: : \: u^d = 1 \right\}. $$
The center of the matrix group $\SL{d}{\R}$ is  
$$ Z(\SL{d}{\R}) =   \left\{ u \Id_{d} \: : \: u \in \unitsord{\R}{d} \right\}.$$
It is shown below that $Z(\SL{d}{\R})  \le \EL{d}{\R}$. In fact $Z(\SL{d}{\R}) = Z(\EL{d}{\R})$. See Corollary \ref{cor:center of EL and SL}.

Let $\RS$ be a commutative Noetherian ring and $f : \R\to \RS$ be a surjective ring homomorphism. It  induces a reduction map $\widehat{f} : \SL{d}{\R} \to \SL{d}{\RS}$.

  \begin{prop}
  \label{prop:center in quotients}
  $ Z(\SL{d}{\RS}) \le \widehat{f}(\EL{d}{\R})$.
  \end{prop}
  \begin{proof}

Let $g \in Z(\SL{d}{\RS})$ be any element. We may write  $g = u \mathrm{Id}_d$ for some unit $u \in \unitsord{\RS}{d}$. 
Let $x,y \in \R$ be   ring elements so that $f(x) = u$ and $f(y)= u^{-1}$. 
 For every pair of distinct indices $i,j \in \{1,\ldots, d\}$ consider the following elements
$$W_{i,j}(x,y) = \elm{i}{j}(x)\elm{j}{i}(-y)\elm{i}{j}(x)\in \EL{d}{\R}$$
and 
  $$D_{i,j}(x,y) = W_{i,j}(x,y)W_{i,j}(-1,-1).$$
   A straightforward calculation shows that the reduction of the matrix $D_{i,j}(x,y)$ is given by
$$(\widehat{f} D_{i,j}(x,y))_{k,l} = \begin{cases} f(x) = u& k=l =i, \\ f(y) = u^{-1} & k=l=j, \\1 &\text{$k = l$ distinct from $i$ and $j$}, \\ 0 & k\neq l. \end{cases} \quad \forall k,l \in \{1,\ldots,d\} $$ 
In other words $D_{i,j}(u) \coloneqq \widehat{f} D_{i,j}(x,y)$ is a diagonal matrix with $u$ and $u^{-1}$ at the $i$-th and $j$-th respective positions along the diagonal. As $u \in \unitsord{\RS}{d}$  it follows that 
  $$   D_{1,2}(u)D_{2,3}(u^2)\ldots D_{d-1,d}(u^{d-1}) = u \mathrm{Id}_d.$$
We conclude that $g = u \mathrm{Id}_d \in \widehat{f}(\EL{d}{\R})$ as required.
  \end{proof}

The above explicit description of the center of the special linear group shows that  $\widehat{f}(Z(\SL{d}{\R})) \le Z(\SL{d}{\RS})$. It may certainly be the case in general that $\widehat{f}(Z(\SL{d}{\R}))$ is a proper subgroup of $Z(\SL{d}{\RS})$.

  \begin{thm}
  \label{thm:center in of SL has finite index}
 If $|\ker f| < \infty$ then $ \left[Z(\SL{d}{\RS}) : \widehat{f}(Z(\SL{d}{\R})\right] < \infty$.
  \end{thm}
  \begin{proof}
  The   ring epimorphism  $f : \R \to \RS $ induces a group homomorphism $\units{f} : \units{\R} \to \units{\RS}$ of the corresponding groups of units. The assumption   $|\ker f| < \infty$ implies that $\units{f} $ is a group epimorphism as shown in \cite[Theorem 3.8]{bartel_lenstra_2017}. It is  moreover clear  that $| \ker \units{f} | \le | \ker f | < \infty$.   

The desired statement  is equivalent to saying that  $\left[\unitsord{\RS}{d} : \units{f} (\unitsord{\R}{d})\right] < \infty$. 
The fact that $|\ker \units{f}| < \infty$ implies that there exists an   integer $N \in \NN$ such that the restriction of the homomorphism $\units{f}$ to the subgroup  $\unitsord{\R}{dN}$ already surjects onto $\unitsord{\RS}{d}$.
The   group  $\unitsord{\RS}{d}$    is an abelian torsion group of bounded exponent. As such and up to isomorphism it can be written as 
$$ \unitsord{\RS}{d} \cong  \bigoplus_{p^k | d} A_{p^k}$$
where each   $A_{p^k}$ is a direct sum of cyclic groups of prime power order $p^k$ dividing $d$, see
\cite[Theorems 1 and 6]{kaplansky1954infinite}. The abelian group $\unitsord{\R}{dN}$ can likewise  be written up to isomorphism as 
$$ \unitsord{\R}{dN} \cong  \bigoplus_{q^k | dN} B_{q^k}$$
  where each   $B_{q^k}$ is a direct sum of cyclic groups of prime power order $q^k$ dividing $dN$. As $|\ker \units{f}| < \infty$  there can only  be finitely many cyclic   direct summands $B_{q^k}$ satisfying $q^k \nmid d$ overall. 
The desired conclusion follows.
 \end{proof}

%
%


\subsection*{Normal subgroup structure} Let $\I \nrm \R$ be an ideal.
The \emph{congruence subgroup} $\SL{d}{\I}$ corresponding to the ideal $\I $ is the kernel of the reduction modulo $\I$  homomorphism 
$$ 1 \to \SL{d}{\I} \to \SL{d}{\R} \to \SL{d}{\R/\I}.$$
Let $\SLtil{d}{\I}$ be the kernel of the composition 
$$ \SL{d}{\R} \to \SL{d}{\R / \I} \to \PSL{d}{\R / \I}. $$
In other words $\SLtil{d}{\I}$ is the preimage in $\SL{d}{\R}$ of the center of $\SL{d}{\R/\I}$. Clearly $\SL{d}{\I} \le \SLtil{d}{\I}$ and   $\SLtil{d}{\I}/\SL{d}{\I}$ is central in $\SL{d}{\R}/\SL{d}{\I}$.
\begin{remark}
Strictly speaking, the group $\SLtil{d}{\I}$ as well as the group $\EL{d}{\I}$ to be introduced below both depend on the ring $\R$ and not only on the ideal $\I$. Since in this work the ring $\R$ is for the most part kept fixed, or is otherwise clear from the context, we will allow ourselves to drop $\R$ from the notation. The group $\SL{d}{\I}$ on the other hand depends only on the ideal $\I$ considered as a ring without unit.
\end{remark}

\begin{lemma}
\label{lem:behaviour of tilde under quotients}
Let $\J \nrm \R$ be an ideal with $\I \subset \J$. Then $ \SLtil{d}{\I} \le \SLtil{d}{\J}$.
\end{lemma}
\begin{proof}
The   reduction map $\widehat{f} : \SL{d}{\R/\I} \to \SL{d}{\R/\J}$ associated to the surjective ring homomorphism $f : \R/ \I \to \R/\J$  satisfies $\widehat{f}(Z(\SL{d}{\R/\I})) \le Z(\SL{d}{\R/\J})$. Therefore the map $\SL{d}{\R} \to \PSL{d}{\R / \J}$ factors through the group $\PSL{d}{\R/\I}$. The conclusion follows.
\end{proof}

Let $\Fn{d}{\I}$ denote the subgroup of $\EL{d}{\R}$ generated by the elementary subgroups $\elm{i}{j}(\I)$. The subgroup $\Fn{d}{\I}$ need not be normal in general. 
Let  $\EL{d}{\I}$ denote the normal closure of $\Fn{d}{\I}$ in the group $\EL{d}{\R}$. 
As $\SL{d}{\I} \nrm \EL{d}{\R}$ and $\elm{i}{j}(\I) \le \SL{d}{\I}$ for every pair of distinct indices $i,j \in \{1,\ldots,d\}$ we have that $\EL{d}{\I} \le \SL{d}{\I}$.

\begin{lemma}[Tits]
\label{lem:Tits lemma}
Assume that $ d \ge 3$. Then $\EL{d}{\I^2} \le \Fn{d}{\I}$.
\end{lemma}
This lemma is a special case of \cite[Proposition 2]{tits1976systemes}. Tits' proof is given in the context of   Chevalley group schemes. We briefly reproduce the proof in our special case for the reader's convenience.

\begin{proof}[Proof of Lemma \ref{lem:Tits lemma}]
Let $\mathrm{F}_d^{\alpha}(\J)$ be the subgroup  generated by the two elementary subgroups  $\elm{i}{j}(\J)$ and $\elm{j}{i}(\I)$ for a given ideal $\J \nrm \R$ and an unordered pair  $\alpha = \{i,j\}$   of distinct   indices $i,j \in \{1,\ldots,d\}$. Let  $\mathrm{EL}^\alpha_d(\J)$ denote the normal closure of the subgroup $\mathrm{F}^\alpha_d(\J)$ in the subgroup $\mathrm{F}^\alpha_d(\R)$.

We claim that the subgroup $\EL{d}{\J}$ is generated by the subgroups $\mathrm{EL}^\alpha_{d}(\J)$ as $\alpha$ ranges over all subsets of $\{1,\dots,d\}$ with $|\alpha| = 2$. To establish the claim it suffices to show that the subgroup $E$ generated as above satisfies $E \nrm \EL{d}{\R}$.    Observe that 
$ \left[\elm{i}{j}(\R), \mathrm{EL}^\alpha_{d}(\J)\right] \subset E$
for all subsets $\alpha$  and all pairs of distinct indices $i,j$ as above. If    $\alpha = \{i,j\}$ then this observation holds true  as $ \elm{i}{j}(\R)$ normalizes $\mathrm{EL}^\alpha_{d}(\J)$, and otherwise as  $\left[\elm{i}{j}(\R), \mathrm{EL}^\alpha_{d}(\J)\right] \subset \Fn{d}{\J}$. The claim follows.

Showing  that $\mathrm{EL}^\alpha_{d}(\I^2) \le \Fn{d}{\I}$ for any given subset $\alpha = \{i,j\} $ with $|\alpha| = 2$ will conclude the proof (in light of the above claim). Let $\mathrm{G}^\alpha_d(\I)$ denote the subgroup of $\Fn{d}{\I}$ generated by all the elementary subgroups $\elm{k}{l}(\I)$ \emph{other than} $\elm{i}{j}(\I)$ and $\elm{j}{i}(\I)$. As $d \ge 3$ the commutation relations imply that  $\mathrm{F}_d^\alpha(\I^2) \le \mathrm{G}^\alpha_d(\I)$. On the other hand $\mathrm{F}^\alpha_d(\R)$ normalizes $\mathrm{G}^\alpha_d(\I)$.   Therefore $\mathrm{EL}^\alpha_{d}(\I^2) \le \mathrm{G}^\alpha_d(\I) \le \Fn{d}{\I}$.
\end{proof}

\begin{theorem}[\cite{bass1968algebraic}]
\label{thm:relative stability}
Assume that $d > \sr{\R}$. Let $\I \nrm \R$ be an ideal. Then the quotient group $\mathrm{SK}_1(\R; \I) = \SL{d}{\I} / \EL{d}{\I}$ is abelian and independent of $d$.
\end{theorem}


\begin{theorem}[Normal subgroup structure theorem]
\label{thm:normal structure theorem}
Assume that $d \ge 3$. A subgroup $H \le \SL{d}{\R}$ is normalized by $\EL{d}{\R}$ if and only if there is an ideal $ \J \nrm \R$ so that
	$$ \EL{d}{\J} \le H \le \SLtil{d}{\J}. $$
The ideal $\J$ associated to the   subgroup $H$ is uniquely determined.
\end{theorem}
\begin{proof}
This theorem is contained in  \cite{wilson1972normal, vaserstein1981normal, borevich1984arrangement}. See also e.g.  \cite[Theorem 11.15 and Corollary 11.16]{magurn2002algebraic}.
\end{proof}

\begin{theorem}[Borevich--Vavilov]
\label{thm:Valivov}
Assume that $d \ge 3$. Let   $\I \nrm \R$ be an ideal. Then 
$$ Z(\EL{d}{\R}/\EL{d}{\I}) = \SLtil{d}{\I}/\EL{d}{\I}.$$
\end{theorem}
\begin{proof}
It is shown in  \cite[Theorem 4]{borevich1984arrangement} that $\left[ \EL{d}{\R}, \SLtil{d}{\I} \right] = \EL{d}{\I}$. This gives containment in one direction. The containment in the opposite direction follows from the definition of the normal subgroup $\SLtil{d}{\I}$.
\end{proof}

\subsection*{Virtually central subquotients}

Let $\I, \K \nrm \R$ be a pair of ideals with $\mfi{\K} = \I$. Consider the subquotient group
$$ \mathcal{A}_d(\I, \K) = (\SLtil{d}{\I} \cap \EL{d}{\R})/ \EL{d}{\K}.$$
Recall from \S\ref{sec: character theory} that a subgroup $H$ of a given group $G$ is virtually central if \linebreak $\left[H:H \cap Z(G)\right] < \infty$.
The following fact will play  an important role in our work.

\begin{theorem}
\label{thm:virtually abelian}
 The subquotient $\mathcal{A}_d(\I, \K)$ 
 is virtually central in the quotient group $\EL{d}{\R}/\EL{d}{\K}$. 
\end{theorem}
  
In particular the subquotient group $\mathcal{A}_d(\I, \K)$  is  virtually abelian.
Before proving Theorem   \ref{thm:virtually abelian} we need the following observation.

 \begin{prop}
  \label{prop: congurence subgroups of ideal is of finite index in the congurence subgroup of its depth}
If $\mfi{\K} = \I$ then  $ [\SLtil{d}{\I}: \SLtil{d}{\K}] < \infty$.
\end{prop}
\begin{proof}
Consider the ring epimorphism $f : \R/\K \to \R/\I $ and the corresponding reduction map $\widehat{f} : \SL{d}{\R/\K} \rightarrow \SL{d}{\R/\I}$. We know from Theorem   \ref{thm:center in of SL has finite index} that
$$i \coloneqq \left[ Z(\SL{d}{\R/\I}) : \widehat{f}(Z(\SL{d}{\R/\K})) \right] < \infty. $$
It follows from the definition of the groups $\widetilde{\mathrm{SL}}_d$ that
$$ \frac{\SLtil{d}{\I}}{\SLtil{d}{\K}} \cong \frac{(\widehat{f})^{-1}Z(\SL{d}{\R/\I})}{Z(\SL{d}{\R/\K})}. $$
We may conclude that
$$ \left| \frac{\SLtil{d}{\I}}{\SLtil{d}{\K}}\right| =  \left| \frac{(\widehat{f})^{-1}Z(\SL{d}{\R/\I})}{Z(\SL{d}{\R/\K})} \right| \le i  |\ker \widehat{f}| < \infty$$
as required.
\end{proof}

%
%

\begin{proof}[Proof of Theorem \ref{thm:virtually abelian}]
Recall that $\I,\K \nrm \R$ is a pair of ideals with $\mfi{\K} = \I$. Consider the following series of subgroups (see Lemma \ref{lem:behaviour of tilde under quotients})
$$ \EL{d}{\K} \le \SLtil{d}{\K} \le \SLtil{d}{\I}. $$
On the one hand, the theorem of Borevich and Vavilov (see Theorem \ref{thm:Valivov} above) implies that $$\left[ \EL{d}{\R}, \SLtil{d}{\K} \right] = \EL{d}{\K}.$$
On the other hand, it was established in Proposition \ref{prop: congurence subgroups of ideal is of finite index in the congurence subgroup of its depth} that
	$$ [\SLtil{d}{\I}: \SLtil{d}{\K}] < \infty.$$
The conclusion follows.
\end{proof}


%
%

\section{Abelian subgroups of the group $\EL{d}{\R}$}
\label{sec:abelian subgroups}
Let $\R$ be a Noetherian ring. Let    $d \in \NN$ be a fixed integer with $d \ge 3$. 

\subsection*{Horizontal and vertical subgroups}
Fix an index  $i \in \{1,\ldots,d\}$. The \emph{horizontal} subgroup $\hor{i}(\R)$ of the matrix group $\SL{d}{\R}$ is  generated by the elementary subgroups $\elm{i}{j}(\R)$ for all indices $j \in \{1,\ldots,d\} \setminus \{i\}$. Likewise, the \emph{vertical} subgroup $\ver{i}(\R)$   is the subgroup generated by the elementary subgroups $\elm{j}{i}(\R)$ for all indices $j \in \{1,\ldots,d\} \setminus \{i\}$.  The   commutation relations given in \S\ref{sec: structure of SLd}
 imply that the   subgroups $\hor{i}(\R)$ and $\ver{i}(\R)$ are all   isomorphic to the additive group of the ring $\R^{d-1}$. Observe that $$\hor{i}(\R) \cap \ver{j} (\R)= \elm{i}{j}(\R).$$

\begin{prop}
\label{prop:commutator of elementary and horizontal/vertical}
The   elementary groups satisfy the   commutation relations
$$\text{ $[ \elm{i}{j}(\R), \hor{k}(\R)] = \elm{k}{j}(\R)$ \quad and \quad $[\elm{i}{j}(\R), \ver{k}(\R)] = \elm{i}{k}(\R)$}$$
for every three distinct indices   $i,j,k \in \{1,\ldots,d\}$.
\end{prop}
\begin{proof}
Let us prove the commutation relation involving the horizontal group $\hor{k}(\R)$.  We have that
\begin{align*}
 [\elm{i}{j}(\R), \hor{k}(\R)] &= [\elm{i}{j}(\R),  \elm{k}{i}(\R) \prod_{l \not\in \{ i,k\} }  \elm{k}{l}(\R) ] = \\
 &= [\elm{i}{j}(\R),    \elm{k}{i}(\R)] [\elm{i}{j}(\R),  \prod_{l \not\in \{ i,k\} }   \elm{k}{l}(\R)]^ {\elm{k}{i}(\R)} = \elm{k}{j}(\R)
\end{align*}
as required.
The proof of the second commutation relation involving the vertical group   $\ver{k}(\R)$ is analogous.
\end{proof}

\subsection*{Normalizers of horizontal and vertical  subgroups}
Fix an index $i \in \{1,\ldots,d\}$. The subgroup $\ngp{i}(\R)$ of the matrix group $\SL{d}{\R}$   is given by
$$ \ngp{i}(\R) = \{g \in \SL{d}{\R} \: : \: \text{$g_{i,i} = 1$ and $g_{i,j} = g_{j,i} = 0$ for all $ j\in \{1,\ldots,d\}\setminus\{i\}$}  \}.$$
The groups $\ngp{i}(\R)$ are all isomorphic to the matrix group $\SL{d-1}{\R}$. Observe that $\ngp{i}(\R)$ normalizes both subgroups $\hor{i}(\R)$ and $\ver{i}(\R)$ for every index $i \in \{1,\ldots,d\}$.

Identifying the subgroup $\ngp{i}(\R)$ with the matrix group $\SL{d-1}{\R}$ and the two subgroups $\hor{i}(\R)$ and $ \ver{i}(\R)$ with the abelian group $\R^{d-1}$ in the obvious way, conjugation corresponds to matrix multiplication. More precisely, given an element $n \in \ngp{i}(\R)$ as well as a pair of elements $ u \in \hor{i}(\R)$ and $v \in \ver{i}(\R)$ we have that
	$$ \text{$n^{-1} u n  = n^t\cdot u$ \quad and \quad $n^{-1}v n = n^{-1}\cdot v$ } $$
	where "$\cdot$" denotes matrix multiplication with respect to the above identification and $n^t$ denotes the matrix transpose of $n$.

The full normalizer of the horizontal and vertical subgroups $\hor{i}(\R)$ and $\ver{i}(\R)$ is larger than $\ngp{i}(\R)$   in general. More precisely we have the following.

\begin{prop}
\label{prop:normalizer of vertical or horizontal subgroup}
If $i \in \{1,\ldots,d\}$ is an index then
$$ N_{\SL{d}{\R}} (\ver{i}(\R)) = \{ g \in \SL{d}{\R} \: : \: \text{$g_{i,j} =0 $  for all $ j \in \{1,\ldots,d\} \setminus \{i\} $} \} $$
and
$$ N_{\SL{d}{\R}} (\hor{i}(\R)) = \{ g \in \SL{d}{\R} \: : \: \text{$g_{j,i} =0 $  for all $ j \in \{1,\ldots,d\} \setminus \{i\} $} \}. $$
\end{prop}
\begin{proof}
We   compute the normalizer of the subgroup $\ver{i}(\R)$. The proof for   $\hor{j}(\R)$ is analogous. Denote $M_i = \{ g \in \SL{d}{\R} \: : \: \text{$g_{i,j} =0 $  for all $ j\neq i $} \}$. It is easy to verify that $M_i \le N_{\SL{d}{\R}} (\ver{i}(\R))$. For the opposite inclusion,   a direct computation shows that
$$ (g^{-1} \elm{k}{i}(1) g)_{i,j} = (g^{-1})_{i,k} g_{i,j} \quad \forall j,k \in \{1,\ldots,d\} \setminus \{i\}. $$
Every element $g \in N_{\SL{d}{\R}} (\ver{i}(\R))$ must therefore satisfy $(g^{-1})_{i,k} g_{i,j} = 0$. In particular either $g_{i,j} = 0$ for all $j \in \{1,\ldots,d\} \setminus \{i\}$ or $(g^{-1})_{i,k} = 0$ for all $k \in \{1,\ldots,d\} \setminus \{i\}$. 
The first case implies that  $g \in M_i$ by definition. In the second case $g^{-1} \in M_i$ so that  $g \in M_i$ as well since $M_i$ is a subgroup.
\end{proof}

Note that if either $g \in N_{\SL{d}{\R}} (\ver{i}(\R)) $ or $g \in N_{\SL{d}{\R}} (\hor{i}(\R)) $ then $g_{i,i} \in \units{\R}$.

\subsection*{Centralizers of elementary subgroups}
\label{sub:centralizers}

%
%


We leave the direct verification of the following computation to the reader.

\begin{prop}
\label{prop:centraliser of an elementary matrix with unit entry}
Let $x \in \R$ be any   element. Denote $\I_x = \mathrm{Ann}_{\R}(x) \nrm \R$. If $i,j \in \{1,\ldots,d\}$ is a pair of distinct indices then
\begin{equation*}
 C_{\SL{d}{\R}}(\elm{i}{j}(x)) = \left\{ g \in \SL{d}{\R} \: : \:  
 \begin{aligned}[c]
 &g_{k,i} = g_{j,l} = 0 \quad \forall  k,l \in  \{1,\ldots,d\}, k\neq i, l\neq j \\ &\text{and} \quad g_{i,i} + \I_x =g_{j,j} + \I_x  \end{aligned} \right\}. 
 \end{equation*}
\end{prop}

If $g \in C_{\SL{d}{\R}}(\elm{i}{j}(x)) $ then the reduction modulo the ideal $\I_x = \mathrm{Ann}_\R(x)$ of the diagonal entries $g_{i,i}$ and $g_{j,j}$ coincides and must be a unit in the quotient ring $\R/\I_x$. In particular, if the element $g$ centralizes all elementary subgroups $\elm{i}{j}(R)$ then $g$ must be of the form $u \mathrm{Id}_d$ for some unit $u \in \unitsord{\R}{d}$.

\begin{cor}
\label{cor:center of EL and SL}
$Z(\EL{d}{\R}) = Z(\SL{d}{\R})$.
\end{cor}
\begin{proof}
This follows from the above discussion combined with Proposition   \ref{prop:center in quotients}.
\end{proof}

\begin{cor}
	\label{cor:centraliser of a vertical/horizontal}
The centralizers of the vertical and the horizontal subgroups are given by
$$ \mathrm{C}_{\SL{d}{\R}} (\ver{i}(\R)) = \mathrm{C}_{\EL{d}{\R}} (\ver{i}(\R)) = Z(\SL{d}{\R}) \times \ver{i}(\R)$$ and 
$$ \mathrm{C}_{\SL{d}{\R}} (\hor{i}(\R)) = \mathrm{C}_{\EL{d}{\R}} (\hor{i}(\R)) = Z(\SL{d}{\R}) \times \hor{i}(\R)$$ for all indices $i \in \{1,\ldots,d\}$.
\end{cor}
\begin{proof}
For the centeralizer in the group $\SL{d}{\R}$ this   follows by intersecting the statement of Proposition \ref{prop:centraliser of an elementary matrix with unit entry} over all values of the index $ j \in \{1,\dots,d\} \setminus \{i\}$ and all ring elements $x \in \R$. The centeralizer in the group $\EL{d}{\R}$ is the same according to Corollary \ref{cor:center of EL and SL}.
\end{proof}

Let $\Mn{d}{\R}$  denote the abelian additive group  of   $d$-by-$d$ matrices with entries in the ring $\R$.

\begin{prop}
\label{prop:not commuting implies infinitely many pairwise disjoint, in N}
Assume that the zero ideal $\left(0\right) \nrm \R$ is a depth ideal. 
Let
$$g \in \SL{d}{\R} \setminus C_{\SL{d}{\R}}(\elm{i}{j}(1))$$
be any element for some pair of distinct indices  $i,j \in \{1,\ldots, d\}$.
  Then there is a sequence of elements $x_n \in \elm{i}{j}(\R)$ such that
%
\begin{enumerate}
\item 
\label{part 1}
the commutators $\left[g,x_n\right]  $ are pairwise distinct modulo  $Z(\SL{d}{\R})$,  
\item  
\label{part 2}
if  either  $g \in \mathrm{N}_{\SL{d}{\R}}(\hor{i}(\R))$ or  $g \in \mathrm{N}_{\SL{d}{\R}}(\ver{j}(\R))$ then the  commutators
$\left[g,x_n\right]  $   belong  to $\hor{i}(\R)$ or $\ver{j}(\R)$ respectively, and
\item 
\label{part 3}
if either $ g \in \ver{k}(\R) $   or   $g \in \hor{k}(\R) $
 for some index $k \in \{1,\ldots, d\} \setminus \{i,j\} $ then moreover the commutators
$\left[g,x_m^{-1} x_n \right]  $   for all $ n,m \in \NN$ with $ n < m$ are     pairwise distinct.
\end{enumerate}
\end{prop}
\begin{proof}
Let $X$ and $Y$ be the matrices in $\Mn{d}{\R}$  given by
$$ X_{k,l} = \begin{cases}  
g_{j,l} & k = i, \\
0 & \text{otherwise}
\end{cases}
$$
and
$$ Y_{k,l} = \begin{cases}  
g_{k,i} & l = j, \\
0 & \text{otherwise}
\end{cases}
$$
for all indices $k,l \in \{1,\ldots,d\}$. A direct computation shows that 
$$ \elm{i}{j}(r) g \elm{i}{j}^{-1}(r) = g + r(X - Y) $$
for all ring elements $r \in \R$. The assumption that $g \notin  C_{\SL{d}{\R}}(\elm{i}{j}(1))$ implies that   $X \neq Y $. In light of  Proposition \ref{prop:a condition for an ideal to be maximal of finite index} there exists a sequence of elements $r_n \in \R$  such that the matrices $r_n (X-Y)$ are  pairwise distinct for all $ n\in \mathbb{N}$. In fact, as $d \ge 3$ and the matrix $X-Y$ has at most two non-zero rows and columns, the matrices $r_n (X-Y)$ must be pairwise distinct modulo the center $Z(\SL{n}{\R})$ as well. Take $x_n = \elm{i}{j}(r_n)$. This concludes the proof of Item (\ref{part 1}).

To establish Item (\ref{part 2}) assume that    either  $g \in \mathrm{N}_{\SL{d}{\R}}(\hor{i}(\R))$ or  $g \in \mathrm{N}_{\SL{d}{\R}}(\ver{j}(\R))$. The same proof as above goes through. In addition,  the commutators $\left[g,x_n\right]$ all belong to the abelian subgroups $\hor{i}(\R)$ or $\ver{j}(\R)$, respectively.
 
 Finally assume  that $g \in \ver{k}(\R) $ for some index  $k  \in \{1,\ldots,d\} \setminus \{i,j\}$. The proof in the case of $g \in \hor{k}(\R)$ is analogous. A direct computation shows that 
$$ [\elm{i}{j}(r), g] = \elm{i}{k}(rg_{j,k} - g_{i,k})  $$
for all ring elements $r \in \R$. The assumption   $g \notin C_{\SL{d}{\R}}(\elm{i}{j}(1))$ implies that $g_{j,k} \neq 0$. According to Proposition \ref{prop:a condition for an ideal to be maximal of finite index} there is a sequence of elements  $r_n \in \R$ such that  the elements $r_n g_{j,k} \in \R$ are  pairwise distinct for all $n \in \NN$.
Up to passing to a subsequence and reindexing, we may assume   that the elements   $(r_{n } - r_{m}) g_{j,k}$ are  pairwise distinct for all $n < m$.
The elements  $x_n = \elm{i}{j}(r_{n})$ satisfy Item (\ref{part 3}).
\end{proof}

The assumption that $\left(0\right)$ is a depth ideal of the ring $\R$ is equivalent to  every non-zero ideal $\I \nrm \R$ satisfying $|\I | =\infty$. This in turn is equivalent to the statement that $\R = \ZZ\left[x_1,\ldots,x_k\right]/\J$ for some $k \in \NN$ and some \emph{depth} ideal $\J$.
\subsection*{Depth ideals and centralizers}

\begin{prop}
\label{prop:not in finite index implies infinite many infinite index}
 Let $\I,\J $ and $\LL$ be ideals  in the ring $\R$ satisfying
$$ \LL \subset \J \subset \I,  \quad |\I/\J| < \infty \quad \text{and} \quad |\J/\LL| = \infty. $$
Let $ g = \elm{j}{k}(s)$ be any element where $s \in \I \setminus \J$ and  $j,k \in \{1,\ldots,d\}$  are distinct indices.
 Then there are
%
 elements   $x_n \in \EL{d}{\R}$ such that the commutators
$\left[g,x_n\right]  $   belong to $\SL{d}{\J}$ and       are pairwise distinct modulo the subgroup $\SLtil{d}{\LL}$  for all $ n \in \NN$.
\end{prop}
\begin{proof}
Fix an index $i \in \{1,\ldots,d\} \setminus \{j,k\}$. 
Consider the ideal  $\I_0 = \R s + \J$ in the ring $\R$. In particular $\J \subset \I_0 \subset \I$. The two  ideal quotients $\J:\I_0$ and $\LL:\I_0$
satisfy
$$|(\J:\I_0)/(\LL:\I_0)| = \infty $$
according to Proposition \ref{prop:infinite index implies infinite index ideal quotient}. Therefore there exists a sequence $q_n \in \J:\I_0$   of elements   pairwise distinct modulo $\LL:\I_0$. Up to passing to a further subsequence and reindexing, we may assume  that $q_n - q_m$ are pairwise distinct modulo $\LL:\I_0$ for all $n < m$ as well. Denote
$ y_n = \elm{i}{j}(q_n) $ for all $ n\in \NN$.

The  $\SL{d}{\R}$-action by conjugation on the quotient group $\SLtil{d}{\I} / \SL{d}{\J}$  has finite
 orbits. Up to passing to a further subsequence and reindexing, and by the pigeon hole principle,    we may assume that
$$ y_n^{-1} g y_n \SL{d}{\J} = y_m^{-1} g y_m \SL{d}{\J}$$
for all $n,m \in \NN$.  Consider the elements 
$$ r_n = q_n - q_1 \in \J : \I_0 \quad \text{and} \quad x_n = y_n y_1^{-1}   = \elm{i}{j}(r_n ) \in \elm{i}{j}(\R).$$
In particular it follows that $ [g,x_n] \in \SL{d}{\J}$ for all $n \in \NN$.

It remains to show that  the commutators $[g,x_n] $ are   pairwise distinct modulo the subgroup $\SLtil{d}{\LL}$. The ring elements $r_n$ are clearly distinct modulo the ideal quotient $\LL:\I_0$. We claim that the ring elements $r_n s$   are pairwise distinct modulo the ideal $\LL$.   Indeed, if $r_n s = r_m s$ modulo the ideal $\LL$ for some $n,m\in \NN$ then $(r_n - r_m)\I_0 \subset \LL$ and   $r_n - r_m \in \LL : \I_0$, which is a contradiction.
Finally, let $X$ and $Y$ be the two matrices in $\Mn{d}{\R}$  given by
$$ X_{l,m} = \begin{cases}  
g_{j,m} & l=i, \\
0 & \text{otherwise}
\end{cases}
$$
and
$$ Y_{l,m} = \begin{cases}  
g_{l,i} & m = j, \\
0 & \text{otherwise}
\end{cases}
$$
for all indices $l,m \in \{1,\ldots,d\}$. Note that $s = (X-Y)_{k,j}  \in \I \setminus \J$. 
A direct computation shows that
$$[g,x_n] = (g + r_n(X-Y))g^{-1} $$
for all $n \in \NN$, so that  these   commutators are pairwise distinct modulo $\SLtil{d}{\LL}$.
\end{proof}

\subsection*{A normal form decomposition}
\label{sub:normal form decomposition}

The following result  is inspired by \cite[Proposition 13]{bekka}. It shows that every element of $\SL{d}{\R}$ is conjugate to a product of three matrices of a particularly simple form. The proof makes an essential use of the notion of stable range and in particular of the fact that $d > \sr{\R}$.

\begin{prop}
	\label{prop:normal form for conjugates}
Assume that $d > \sr{\R}$. Then any element $g \in \SL{d}{\R}$ is conjugate by an element of $\EL{d}{\R}$ to an element $g'$ where 
$$g' =   hvh'v' n $$
 for some elements
$$    h,h' \in \hor{1}(\R), \quad v,v' \in \ver{1}(\R) \quad \text{and} \quad n \in \ngp{1}(\R). $$
Moreover we may assume that $h' = \elm{1}{2}(-1)$.
\end{prop}

The pair of implicit indices $ i = 1$ and $j = 2$  in the above statement can  be replaced with any fixed pair of distinct indices $i,j \in \{1,\ldots,d\}$.

\begin{proof}[Proof of Proposition \ref{prop:normal form for conjugates}]
Let $g \in \SL{d}{\R}$ be any element. 
If $j \in \{2,\ldots,d\}$ and $i \in \{1,\ldots,d\} \setminus \{j\}$ is a pair of distinct indices  and $s \in \R$ is any ring element then
$$ (\elm{i}{j}(s) g \elm{i}{j}(s)^{-1})_{i,1} = g_{i,1} + s g_{j,1}.$$
Since $d >  \sr{\R}$ the element $g$ can be conjugated by an appropriate element of the vertical group $\ver{2}(\R)$ in so that the first column $(r_1,\ldots,r_d) \in \R^d$ of the resulting matrix is such that  $r_1, r_3, \ldots, r_d$ form a unimodular $(d-1)$-tuple. It is therefore possible to find  elements $s_1, s_3, \ldots, s_n \in \R$ satisfying
	$$ s_1 r_1 + s_3 r_3 + \cdots + s_d r_d = 1 - r_1 - r_2. $$
Up to further conjugating   by the element $\elm{2}{3}(s_3)\cdots\elm{2}{d}(s_d) \in \hor{2}(\R)$   we may assume that the resulting element $g'$ satisfies
	$$ (1 + s_1) g'_{1,1} + g'_{2,1} = 1. $$
	Consider the elements $v_1 = \elm{2}{1}(s_1) \in \ver{1}(\R)$ and $h_1 = \elm{1}{2}(1) \in \hor{1}(\R)$.  Denote
	$$g'' = h_1 v_1 g'$$
so that $g''_{1,1} = 1$. Therefore there are suitable elements  $v_2 \in \ver{1}(\R)$ and $h_2 \in \hor{1}(\R)$ so that
	$$ n = v_2 g'' h_2  \in \ngp{1}(\R). $$
Rearranging the above equations gives
$$ v_1^{-1} h_1^{-1} v_2 ^{-1} n =   g' h_2. $$
The desired statement follows for a suitable choice of the elements $h,v,h',v'$ and up to replacing the element $g'$ by its conjugate $h_2^{-1} g' h_2$.
\end{proof}

\section{The level ideal $\levI{\varphi}$}
\label{sec: level of the character}

Let $\R$ be a Noetherian ring and  $d \ge 3$ be a fixed integer.  We associate to each character $\varphi \in \chars{\EL{d}{\R}}$ a uniquely determined depth ideal $\levI{\varphi}\nrm \R$  called its \emph{level ideal}, see Theorem \ref{thm:existence of a depth ideal}. It is shown in  Theorem \ref{thm:vanishing outside big S_varphi using normal form decomposition}  below  that  the character $\varphi$ is induced from the normal subgroup $\EL{d}{\R} \cap \SLtil{d}{\levI{\varphi}}$ provided $d > \sr{R}$.


\label{sec:level of a char}



\subsection*{Restrictions   to  abelian subgroups}

Let $\varphi \in \traces{\EL{d}{\R}}$ be a trace.
The restrictions $ \varphi_{|\hor{i}(\R)}$ and $\varphi_{|\ver{i}(\R)}$ of the trace $\varphi$ to the horizontal and vertical abelian  subgroups $\hor{i}(\R)$ and $\ver{i}(\R)$   are positive definite functions  for all   $i \in \{1,\ldots,d\}$. 
 Bochner's theorem \cite[Theorem 4.18]{folland} says that there are  uniquely determined  probability measures $\horp{i}$ and $\verp{i}$  on the dual  compact abelian groups $\hord{i}$ and $\verd{i}$ respectively such that 
$$ \horc{i} = \widehat{\horp{i}} \quad \text{and} \quad \verc{i} = \widehat{\verp{i}}$$
for all indices $i\in\{1,\ldots,d\}$.

The conjugation invariance of the trace $\varphi$ implies that $\varphi(g) = \varphi(g^n)$ for all elements $g \in \EL{d}{\R}$ and   $n \in \ngp{i}(\R)$. It follows that the probability measures $\horp{i}$ and $\verp{i}$ are invariant under the dual action corresponding to conjugation by the normalizing group $\ngp{i}(\mathcal{R})$ for all $i\in\{1,\ldots,d\}$. Both dual groups $\hord{i}$ and $\verd{i}$ can all be identified with the dual group $\widehat{\R}^{d-1}$ with its additive structure. With this identification the dual action of   $\ngp{i}(\mathcal{R})$ is the natural one via matrix multiplication.

We are now in a position to apply our  invariant measure classification result, namely  Theorem \ref{thm:classification of invariant measures}. We deduce that the two probability measures $\horp{i}_\I$ and $\verp{i}_\I$
can be uniquely written as convex combinations
$$ \horp{i} = \sum_{\I\nrm \R, \mfi{\I}=\I} \alpha^{\hor{i}}_\I \horp {i}_\I \quad \text{and} \quad \verp{i} = \sum_{\I\nrm \R, \mfi{\I}=\I} \alpha^{\ver{i}}_\I \verp {i}_\I $$
ranging over all depth ideals $\I \nrm \R$ and with coefficients $\alpha_\I^{\hor{i}} \ge 0$ and $\alpha_\I^{\ver{i}} \ge 0$ satisfying $\sum \alpha^{\hor{i}}_\I = \sum \alpha^{\ver{i}}_\I = 1$ for all $i \in \{1,\ldots,d\}$. For a given depth ideal $\I \nrm \R$ each probability measure $ \horp {i}_\I$ and $\verp {i}_\I $ is a convex combination of  a countable family of translates of the Haar measure of the compact dual group $(\ann{\I})^d \le \widehat{\R}^d$.

\subsection*{Projection valued measures}

Let $\varphi \in \traces{\EL{d}{\R}}$ be a trace. Consider the GNS construction $(\pi_\varphi, \mathcal{H}_\varphi, v_\varphi)$ associated to the trace $\varphi$ as in Theorem \ref{thm: GNS}. Namely  $\pi_\varphi$ is a unitary representation of the group $\EL{d}{\R}$ on the Hilbert space $\mathcal{H}_\varphi$ admitting a cyclic vector $v_\varphi \in \mathcal{H}_\varphi$ and satisfying
 $$\varphi(g) = \left< \pi_\varphi(g)v_\varphi, v_\varphi \right>  \quad \forall g \in \EL{d}{\R}.$$

The restriction of the unitary representation $\pi_\varphi$ to each horizontal and vertical subgroup 
 gives rise to    unique projection-valued measures $\horpv{i}$ and $\verpv{i}$ on the dual abelian compact groups $\hord{i}$ and $\verd{i}$ respectively such that
$$ \pi_{\varphi|\hor{i}} = \int_{\hord{i}} \chi  \: \mathrm{d}\horpv{i}(\chi) \quad \text{and}  \quad \pi_{\varphi|\ver{i}} = \int_{\verd{i}} \chi \: \mathrm{d}\verpv{i}(\chi).$$


%

To proceed with the analysis of the projection valued measures $\horpv{i}$ and $\verpv{i}$ we   extend the discussion to the elementary groups $\elm{i}{j}(\R)$ defined for each  pair of distinct indices $i,j \in \{1,\ldots,d\}$. Consider the projection valued measures $\elmpv{i}{j}$ on the dual abelian compact groups $ \elmd{i}{j} \cong \widehat{\R}$ satisfying
$$ \pi_{\varphi|\elm{i}{j}(\R)} = \int_{\elmd{i}{j}} \chi  \: \mathrm{d}\elmpv{i}{j}(\chi).$$
It is an elementary    fact of harmonic analysis  that 
$$     \hord{i} / \ann{\elm{i}{j}(\R)} \cong  \verd{j} / \ann{\elm{i}{j}(\R)} \cong  \elmd{i}{j}.$$ Let  $p^{\hor{i}} \colon \hord{i} \to \elmd{i}{j}$ and $p^{\ver{i}} \colon \verd{j} \to \elmd{i}{j}$   denote the resulting quotient maps.

\begin{prop}
$  p^{\hor{i}}_* \horpv{i} = p^{\ver{j}}_* \verpv{j} = \elmpv{i}{j}  $ as projection-valued measures on the dual abelian compact group $\elmd{i}{j}$ for all pairs of distinct indices $i,j \in \{1,\ldots,d\}$.
\label{prop:projection valued measures are nicely behaved}
\end{prop}

\begin{proof}
Fix a pair of distinct indices $i,j \in \{1,\ldots,d\}$. We will prove the statement with respect to the quotient map $p^{\hor{i}} : \hord{i} \to \elmd{i}{j}$. The proof for vertical subgroups is  essentially the same.
For every pair of vectors $u,w \in \mathcal{H}_\varphi$ consider the  complex-valued measures $\horpv{i}_{u,w}$ and $\elmpv{i}{j}_{u,w}$ respectively on the dual compact groups $\hord{i}$ and $\elmd{i}{j}$  given by
$$ \horpv{i}_{u,w}(E) =  \left< \horpv{i}(E)u, w\right> \quad \text{and} \quad \elmpv{i}{j}_{u,w}(F) =  \left< \elmpv{i}{j}(F)u, w\right> $$
for all Borel subsets $E \subset \hord{i}$ and $F \subset \elmd{i}{j}$. The desired conclusion $p^{\hor{i}}_* \horpv{i} = \elmpv{i}{j}  $ holds if and only if  $ p_* \horpv{i}_{u,w}  = \elmpv{i}{j}_{u,w}  $ for all vectors $u,v \in \mathcal{H}_\varphi$. This last statement follows from the uniqueness part of Bochner's theorem \cite[Theorem 4.18]{folland} combined with the observation that
\begin{align*}     \int_{\elmd{i}{j}} \left< \chi(g) u, w \right> \: \mathrm{d}\elmpv{i}{j}_{u,w}(\chi)  &= \left< \pi(g) u, w \right> =  \\
&=  \int_{\hord{i}} \left< \chi(g) u, w \right> \: \mathrm{d}\horpv{i}_{u,w}(\chi)= \\
&=  \int_{\elmd{i}{j}} \left< \chi(g) u, w \right> \: \mathrm{d}p_*^{\hor{i}}\horpv{i}_{u,w}(\chi)
\end{align*}
for all elements $g \in \elm{i}{j}(\R) \le \hor{i}(\R)$.
\end{proof}


%
%

\subsection*{Determination of the level ideal}

The following result  is inspired by \cite[Lemma 10]{bekka}.

\begin{theorem}
\label{thm:existence of a depth ideal}
Let $\varphi \in \chars{\EL{d}{\R}}$ be a character. Then there is a unique depth ideal $\levI{\varphi} \nrm \R$ such  that  $\horp{i} = \mu^{\hor{i}}_{\levI{\varphi}}$ and $\verp{i} = \mu^{\ver{i}}_{\levI{\varphi}}$ for all indices $i\in\{1,\ldots,d\}$.
\end{theorem}

The ideal $\levI{\varphi}$ is the \emph{level ideal} corresponding to the character $\varphi$. It plays an important role towards our main   classification result Theorem \ref{thm:main theorem}.

Before giving the proof of Theorem \ref{thm:existence of a depth ideal} let us introduce some useful notation.
Let $\I \nrm \R$ be a depth ideal. Consider the Borel subset
$$ \mathcal{O}^{\hor{i}}_\I = \left(\bigcup_{\J \nrm \R, \, \mfi{\J} = \I}  \ann{\hor{i}(\J)}\right) \setminus \left(\bigcup_{\J \nrm \R, \, \I \lneq \mfi{\J} }  \ann{\hor{i}(\J)}\right)$$
of the dual compact group $\widehat{\hor{i}(\R)}$. The Borel subsets $\mathcal{O}^{\ver{i}}_\I \subset \verd{i}$ and $\mathcal{O}^{\elm{i}{j}}_\I \subset \elmd{i}{j}$ are defined analogously for all pairs of distinct indices $i,j \in \{1,\ldots,d\}$. Observe that
$$p ^{\hor{i}}(\mathcal{O}^{\hor{i}}_\I) = p ^{\ver{j}}(\mathcal{O}^{\ver{j}}_\I)  = \mathcal{O}_\I^{\elm{i}{j}}.$$
The same argument as in the discussion on  \enquote{essential subgroups} in \S\ref{sec:classification of prob measures}
  shows   that $\horpm{i}_\I(\mathcal{O}^{\hor{i}}_\I) = \verpm{i}_\I(\mathcal{O}^{\ver{i}}_\I ) = 1$. Moreover  the Borel subsets $\mathcal{O}^{\hor{i}}_\I$ as well as $\mathcal{O}^{\ver{i}}_\I$  are pairwise disjoint for a fixed index $i \in \{1,\ldots,d\}$   as the ideal $\I$ varies over the different depth ideals of the ring $\R$.

\begin{proof}[Proof of Theorem \ref{thm:existence of a depth ideal}]
Let $(\pi_\varphi, \mathcal{H}_\varphi, v_\varphi)$ be the GNS construction associated to the character $\varphi$.  Let $\horpv{i}, \verpv{j}$ and $\elmpv{i}{j} $ be the projection valued measures on the dual  compact abelian groups 
$\hord{i}, \verd{j}$ and $\elmd{i}{j}$ respectively introduced above for each pair of distinct indices $i,j\in\{1,\ldots,d\}$.

Let $\I \nrm \R$ be any fixed depth ideal. It follows from  Proposition \ref{prop:projection valued measures are nicely behaved} and from the paragraph preceding this proof that 
$$  \horpv{i}(\mathcal{O}^{\hor{i}}_\I) = \verpv{j}(\mathcal{O}^{\ver{j}}_\I) = \elmpv{i}{j}(\mathcal{O}_\I^{\elm{i}{j}})$$
for all pairs of distinct indices $i,j \in \{1,\ldots,d\}$. Repeating this argument with  varying   $i$'s and $j$'s shows that  the orthogonal projection $$\mathrm{P}_\I = \elmpv{i}{j}(\mathcal{O}_\I^{\elm{i}{j}})$$  is independent of both indices.    
The representation theory of locally compact abelian groups \cite[Theorem 4.44]{folland} implies that   $\mathrm{P}_\I \in \pi_\varphi(\elm{i}{j}(\R))'$ for all pairs of distinct indices $  i,j \in \{1,\ldots,d\}$.
 As the elementary groups  generate the group $\EL{d}{\R}$ we deduce that $$\mathrm{P}_\I \in \pi_\varphi(\EL{d}{\R})'.$$
The spectral theorem  \cite[Theorem 1.44]{folland} shows that any bounded operator $B \in \pi_\varphi(\EL{d}{\R})'$ necessarily commutes with $\elmpv{i}{j} (E)$ for every Borel subset $E \subset \elmd{i}{j}$ and every pair of distinct  $i$ and $j$. In particular 
$$ \mathrm{P}_\I \in  \pi_\varphi(\EL{d}{\R})''. $$ 

We conclude that  the projection $ \mathrm{P}_\I$ lies in the center $Z(\mathcal{L}_\varphi)$ of the von Neumann algebra $\mathcal{L}_\varphi$ generated by the unitary representation $\pi_\varphi$. Since $\varphi$ is a character     the von Neumann algebra   $\mathcal{L}_\varphi$ is a factor. Therefore  the projection $\mathrm{P}_\I$ must be either the zero or   the identity operator   for each given depth ideal $\I \nrm \R$.

The Borel subsets $\mathcal{O}_\I^{\elm{i}{j}}$ form a Borel partition of the dual compact group $\elmd{i}{j}$ as $\I$ varies over all depth ideals in the ring $\R$ and for each fixed pair of distinct indices $i,j\in\{1,\ldots,d\}$.   Therefore there is a unique depth ideal $\levI{\varphi} \nrm \R$ such that  $\mathrm{P}_{\I_\varphi}$ is the identity operator on the Hilbert space $\mathcal{H}_\varphi$. The conclusion follows.
 \end{proof}


\subsection*{Properties of the level ideal}

Recall   the notation $\Ind{G}{H}{N}$ introduced in \S\ref{sec: character theory} --- a trace $\varphi$ belongs to $\Ind{G}{H}{N}$ if and only if the restriction of $\varphi$ to the subgroup $H$ is induced from the subgroup $N$.

\begin{prop}
\label{prop:vanishing on elementary matrices outside level}
If  $\varphi \in \chars{\EL{d}{\R}}$ is a character with level ideal $\levI{\varphi} \nrm \R$ then
 $$ \varphi \in \Ind{\EL{d}{\R}}{\ver{i}({\R})}{\ver{i}( \levI{\varphi})}   \cap \Ind{\EL{d}{\R}}{\hor{i}({\R})}{\hor{i}( \levI{\varphi})} $$   for all indices $i \in \{1,\ldots,d\}$.
\end{prop}
\begin{proof}
According to  Theorem \ref{thm:existence of a depth ideal} the level ideal $\levI{\varphi}$ is such that
$ \varphi_{|\hor{i}} = \widehat{\nu}^{\hor{i}}_{\levI{\varphi}}$ and  $\varphi_{|\ver{i}} = \widehat{\nu}^{\ver{i}}_{\levI{\varphi}}$.
The desired conclusion   follows immediately from Proposition \ref {prop:integral of Haar is delta}.
\end{proof}

\begin{prop}
\label{prop:level is uniquely determined}
 Let $\J \nrm \R$ be an ideal. Let $\varphi \in \traces{\EL{d}{\R}}$ be  a trace satisfying
 $$\varphi \in \Ind{\EL{d}{\R}}{\elm{i}{j}(\R)}{\elm{i}{j}(\J)}$$ 
for a pair of distinct indices $i,j \in \{1,\ldots,d\}$. Write
 $$\varphi = \int_{\chars{\EL{d}{\R}}}
\psi \, \mathrm{d}\mu_\varphi(\psi)$$
where $\mu_\varphi$ is a Borel probability measure   on $\chars{\EL{d}{\R}}$. If   $\J $ is contained in the level ideal $\levI{\psi}$ of $\mu_\varphi$-almost every character $\psi$  then   $\mfi{\J} = \levI{\psi}$ holds  true $\mu_\varphi$-almost surely.
\end{prop}
 \begin{proof}
Let $\mu^{\elm{i}{j}}$ be the uniquely determined Borel probability measure on the   dual  compact abelian group $\elmd{i}{j}$ such that $\varphi_{|\elm{i}{j}} = \widehat{\mu}^{\elm{i}{j}}$.
 Let $\nu_\I$ denote the Haar measure    of the   annihilator subgroup $\ann{\I}  $ of a given depth ideal $\I \nrm\R$   regarded as a Borel probability measure  on  $\elmd{i}{j} \cong \widehat{\R}$. 

There is a family of atomic positive measures $\eta_\I$ on   the dual  compact abelian group $\elmd{i}{j}  $ defined for every depth ideal  $\I \nrm \R$ with $\mfi{\J} \subset \I$ that satisfy $\eta_\I( \elmd{i}{j} \setminus \mathcal{O}^{\elm{i}{j}}_\I ) = 0$ and
$$ \mu^{\elm{i}{j}} = \sum_{\text{$\I$ is depth}, \mfi{\J} \subset \I } \eta_\I * \nu_{\I}.$$
 In particular $\sum |\eta_\I| = 1$ where (as above) the sum is   taken  over all depth ideals $\I \nrm \R$ with $\mfi{\J} \subset \I$.
  
On the other hand, we may consider the Borel probability measure $\theta$ on the dual compact group $\elmd{i}{j}$ given by
$$ \theta = \sum_{\text{$\I$ is depth}, \mfi{\J} \subset \I }  \eta_\I * \nu_{\mfi{\J}}.$$
 Note that $\widehat{\theta}(g)   = 0$ for all elements $g \in \elm{i}{j}(\R) \setminus \elm{i}{j}(\J)$ by Proposition \ref {prop:integral of Haar is delta}. Additionally  
$$ \widehat{\theta}(g) = \sum_{\text{$\I$ is depth}, \mfi{\J} \subset \I }   \int_{\elmd{i}{j}} \chi(g) \, \mathrm{d} \eta_\I = \widehat{\mu}^{\elm{i}{j}}(g) = \varphi(g) $$
for all elements $g \in \elm{i}{j}(\J)$. We deduce that $ \widehat{\theta}  = \widehat{\mu}^{\elm{i}{j}} $ overall as functions on the group $\elm{i}{j}(\R)$. The uniqueness   statement in Bochner's theorem \cite[Theorem 4.18]{folland} implies that $\theta = {\mu}^{\elm{i}{j}}$. Therefore $|\eta_{\mfi{\J}}| = 1$ and $\eta_\I = 0$ for every other depth ideal $\I$. This is equivalent to the desired conclusion.
 \end{proof}

The last result of the current \S\ref{sec: level of the character}
 will be used to prove the converse direction of our main result, namely Theorem \ref{thm:main converse} of the introduction.

\begin{cor}
\label{cor:decomposition of a trace vanishing outside}
 Let $\varphi \in \traces{\EL{d}{\R}}$ be  a trace with
 $$\varphi = \int_{\chars{\EL{d}{\R}}}
\psi \, \mathrm{d}\mu_\varphi(\psi)$$
for some Borel probability measure $\mu_\varphi$ on $\chars{\EL{d}{\R}}$.
 Let $\K \nrm \R$ be an ideal satisfying
 $ \EL{d}{\K} \le \ker \varphi \le \SLtil{d}{\K}$.
If  $$\varphi \in \Ind{\EL{d}{\R}}{\EL{d}{\R}}{\SLtil{d}{\mfi{\K} }} $$
   then  the level ideal $\levI{\psi}$ of  $\mu_\varphi$-almost every character $\psi $ is equal to   $  \mfi{\K}$.
\end{cor}
\begin{proof}
It is a general property of traces that $\|\varphi\|_\infty \le 1$.  Therefore    $\ker \varphi \le \ker \psi$ holds true for $\mu_\varphi$-almost every character $\psi \in \chars{\EL{d}{\R}}$. It follows that   $\K$ is contained in the kernel ideal $\K_\psi$ and so $\mfi{\K} \subset \mfi{\K}_\psi \subset \levI{\psi}$ holds true $\mu_\varphi$-almost surely.\footnote{We will show in \S\ref{sec:kernel ideal} the much more precise statement $\mfi{\kerI{\psi}} =\levI{\psi}$. This is not   needed yet for the current argument.} The statement now follows from Proposition \ref{prop:level is uniquely determined}.
 %
%
%
\end{proof}



%
%

%


\section{The kernel ideal $\kerI{\varphi}$}
\label{sec:kernel ideal}
Let $\R$ be a Noetherian ring. Let $d > \max \{2, \sr{\R} \}$ be a fixed integer.
Let $\varphi \in \traces{\EL{d}{\R}}$ be a trace. The kernel $\ker \varphi$  of the trace $\varphi$ is given by
$$ \ker \varphi = \{ g \in \EL{d}{\R} \: : \: \varphi(g) = 1\}.$$
The kernel $\ker\varphi$  is a normal subgroup of the group $\EL{d}{\R}$, see Proposition \ref{prop:trace factors through the kernel}.  According to the normal subgroup structure theorem (see Theorem \ref{thm:normal structure theorem}) there is a uniquely determined  ideal $\kerI{\varphi} \nrm \R$ satisfying
$$ \EL{d}{\kerI{\varphi}} \le \ker \varphi \le \SLtil{d}{\kerI{\varphi}}. $$
We will say that $\kerI{\varphi}$   is the \emph{kernel ideal} associated to the trace $\varphi$.

\begin{theorem}
\label{thm:ker has finite index in lev}
Let $\varphi \in \chars{\EL{d}{\R}}$ be a character with level ideal $\levI{\varphi}$ and kernel ideal $\kerI{\varphi}$. Then $\mfi{\kerI{\varphi}} = \levI{\varphi}$.
\end{theorem}

Theorem \ref{thm:ker has finite index in lev} is  main goal of the current \S\ref{sec:kernel ideal}. Clearly $\kerI{\varphi} \subset \levI{\varphi}$ so that $\mfi{\kerI{\varphi}} = \levI{\varphi}$ is equivalent to the statement $|\levI{\varphi} / \kerI{\varphi}|<\infty$. It is proved below in two parts, making use of two additional auxiliary ideals $\degI{\varphi}$ and $\finI{\varphi}$ associated to  the character $\varphi$.

\begin{remark}
The kernel ideal is well-defined for any trace on the group $\EL{d}{\R}$ but the level ideal is  defined only for characters.
\end{remark}

\subsection*{The supporting ideal $\finI{\varphi}$}

	Let $\varphi \in \chars{\SL{d}{\R}}$ be a character with level ideal $\levI{\varphi} \nrm \R$ and kernel ideal $\kerI{\varphi} \nrm \R$.  We define a certain     ideal $\finI{\varphi} \nrm \R $ associated to  the character $\varphi$ and satisfying $\mfi{\finI{\varphi}} = \levI{\varphi}$. This is   an auxiliary notion needed towards our proof of Theorem  \ref{thm:ker has finite index in lev}.

\begin{lemma}[Bekka]
	\label{lem:bekkas lemma on finite index ideal}
There exists an ideal $\finI{\varphi} \nrm \R $ with $\mfi{\finI{\varphi}} = \levI{\varphi}$ and $\finI{\varphi}^2 \subset \kerI{\varphi}$.
\end{lemma}
This   Lemma and its proof  are essentially the same as \cite[Lemmas 10 and 11]{bekka}. 

\begin{proof}[Proof of Lemma \ref{lem:bekkas lemma on finite index ideal}]
Let $(\pi_\varphi, \mathcal{H}_{ \varphi}, v_\varphi)$ be the GNS construction corresponding to the character $\varphi$ as given in Theorem \ref{thm: GNS}. Recall that $\pi_\varphi$ is a unitary representation of the group $\EL{d}{\R}$ acting on the Hilbert space $\mathcal{H}_{\varphi}$ with   cyclic vector $v_\varphi \in \mathcal{H}_{\varphi}$ and such that  $$\varphi(g) = \left< \pi_\varphi(g)v_\varphi, v_\varphi \right> \quad \forall g \in \EL{d}{\R}. $$

The argument of \cite[Lemma 10]{bekka} shows that there is an ideal $\finI{\varphi} \nrm \R $ satisfying $\mfi{\finI{\varphi}} = \levI{\varphi}$ and such that the subgroup $\Fn{d}{\finI{\varphi}}$  admits  non-zero invariant vectors in the Hilbert space $\mathcal{H}_\varphi$.   

Indeed, Bekka's proof is given  for  the ring  $\ZZ$ and the countable collection of non-zero ideals in $\ZZ$. It  extends mutatis mutandis to our situation, by considering the level ideal $\levI{\varphi}$ and the countable collection of ideals $\J \nrm \R$ with $\mfi{\J} = \levI{\varphi}$. One has to use the fact that  a pair of ideals $\J_1,\J_2 \nrm \R$ with $\mfi{\J_1} = \mfi{\J_2} = \levI{\varphi}$ satisfies $\mfi{(\J_1 \cap \J_2)} = \levI{\varphi}$, see Proposition \ref{prop:properties of depth}.


Recall that $\EL{d}{\finI{\varphi}^2} \le \Fn{d}{\finI{\varphi}}$ according to Lemma \ref{lem:Tits lemma}. Therefore the Hilbert subspace $\mathcal{H}_0 \le \mathcal{H}_\varphi$ consisting of the $\pi_\varphi(\EL{d}{\finI{\varphi}^2})$-invariant vectors is non-zero.  As $\EL{d}{\finI{\varphi}^2}$ is a normal subgroup of $\EL{d}{\R}$ it follows that $\mathcal{H}_0$ is   $\pi_\varphi(\EL{d}{\R})$-invariant where $\pi_\varphi$ is the \enquote{left action} associated to the GNS construction. As $\pi_\varphi$ and $\rho_\varphi$ commute the Hilbert subspace $\mathcal{H}_0$ is moreover  $\rho_\varphi(\EL{d}{\R})$-invariant for the  \enquote{right action} $\pi_\varphi$. 

We conclude that   the non-zero orthogonal projection to the subspace $\mathcal{H}_0$   lies in  the center $Z(\mathcal{L}_\varphi)$ of the von Neumann algebra $\mathcal{L}_\varphi$. The fact that  $\varphi$ is a character implies that   the von Neumann algebra $\mathcal{L}_\varphi$ is a factor. It follows that $\mathcal{H}_0 = \mathcal{H}_\varphi$. In other words $\EL{d}{\finI{\varphi}^2} \le \ker \varphi$ so that $\finI{\varphi}^2 \subset \kerI{\varphi}$ as required.
\end{proof}

We will say that $\finI{\varphi}$ is the \emph{supporting ideal} of the character $\varphi$ and  continue using the notation $\finI{\varphi}$ for the remainder of \S\ref{sec:kernel ideal}.

Let us introduce   shorthand notations for   several normal subgroups of the group $\EL{d}{\R}$ associated with the character $\varphi$ to be used below. Denote 
$$
 \zF{\varphi} = \SL{d}{\kerI{\varphi}} \cap \EL{d}{\R} \quad \text{and} \quad \nF{\varphi} = \SL{d}{\finI{\varphi}}  \cap \EL{d}{\R}.$$

\subsection*{Two-step nilpotent subquotients}

	Let $\varphi \in \chars{\EL{d}{\R}}$ be a character with  supporting ideal $\finI{\varphi}$ and kernel ideal $\kerI{\varphi}  $. 	
	
	The normal subgroup  $\SL{d}{\finI{\varphi}} $ and its subquotients are best understood in terms of additive groups of matrices. 	
	To be precise, consider the quotient ring
$$ \rF{\varphi} = \R / \kerI{\varphi}.$$
 Let $\Mn{d}{\rF{\varphi}}$ denote the abelian group of $d$-by-$d$ matrices with entries in the ring $\rF{\varphi}$ and with group operation given by pointwise addition. Let $\Mnzt{d}{\rF{\varphi}}$ and $\Off{d}{\rF{\varphi}}$ respectively denote the subgroups of $\Mn{d}{\rF{\varphi}}$ consisting of   matrices with zero trace and of matrices whose diagonal entries are all equal to zero. In particular 
 $$\Off{d}{\rF{\varphi}} \le \Mnzt{d}{\rF{\varphi}} \le \Mn{d}{\rF{\varphi}}.$$
 The group $\EL{d}{\R_\varphi}$ is acting on the abelian group $\Mn{d}{\rF{\varphi}}$ by automorphisms via matrix conjugation preserving the subgroup $\Mnzt{d}{\rF{\varphi}}$ (but not   the subgroup $\Off{d}{\rF{\varphi}}$).

\begin{lemma}
\label{lem:exact sequence with i}
There is an exact sequence of $ \EL{d}{\R}$-equivariant homomorphisms
  $$ 1 \to \SL{d}{\kerI{\varphi}} \to \SL{d}{\finI{\varphi}} \xrightarrow{\iota}   \Mnzt{d}{\rF{\varphi}} $$
where the homomorphism  $\iota$ is given by
$$ \iota(g) = (g - \Id_{d})  \Mn{d}{\kerI{\varphi}}   \quad \forall   g   \in \SL{d}{\finI{\varphi}}. $$
\end{lemma}
\begin{proof}
We  first show that    $\iota(g) \in \Mnzt{d}{\rF{\varphi}}$   for any given element 
$ g   \in \SL{d}{\finI{\varphi}}$.   Leibniz formula for determinants implies that
$$1 =  \det(g)  = 1 + \mathrm{tr}(g - \mathrm{Id}_d) + r  $$
for some ring element $r \in \finI{\varphi}^2   $. Therefore $\mathrm{tr}( g - \mathrm{Id}_d) \in  \finI{\varphi}^2  \subset \kerI{\varphi}$ which means the map $\iota$ is  well-defined  in the set-theoretic sense.

We now show that $\iota$ is a group homomorphism. Consider a pair of elements $g_1, g_2 \in   \SL{d}{\finI{\varphi}}$ and write 	$ g_i = \Id_{d} + A_i $ for some matrices $A_i \in \Mn{d}{\finI{\varphi}}$ and $i \in \{1,2\}$. Note that $A_1 A_2 \in \Mn{d}{\kerI{\varphi}} $. It follows that
$$ \iota(g_1 g_2) = \iota((\Id_d + A_1)(\Id_d + A_2))  = (A_1 + A_2)\Mn{d}{\kerI{\varphi}} = \iota(g_1) + \iota(g_2)$$
as required. It is clear from the definition of $\iota$ that $\ker \iota =\SL{d}{\kerI{\varphi}}$. 
Finally, the $ \EL{d}{\R}$-equivariance of the homomorphism $\iota$ follows from the computation
$$ \iota(hgh^{-1}) = \iota(h(\Id_d + A)h^{-1}) = \iota(\Id_d + hAh^{-1}) = (hAh^{-1})\Mn{d}{\kerI{\varphi}} = h\iota(g) h^{-1}$$
for any   element $g = \Id_d + A \in \SL{d}{\finI{\varphi}} $ with $A  \in \Mn{d}{\finI{\varphi}}$ and for every element $h \in \EL{d}{\R}$.
\end{proof}

\begin{cor}
\label{cor:induced exact sequence}
The homomorphism $\iota$ restricts to the following exact sequence 
$$ 1 \to \Fn{d}{\kerI{\varphi}} \to \Fn{d}{\finI{\varphi}} \xrightarrow{\iota} \Off{d}{\finI{\varphi}/\kerI{\varphi}} \to 1.$$
\end{cor}
\begin{proof}
The restriction of the map $\iota$ to each elementary subgroup $\elm{i}{j}(\finI{\varphi})$  for a given pair of distinct indices $i,j \in \{1,\ldots,d\}$ induces the obvious isomorphism of the quotient group $\elm{i}{j}(\finI{\varphi})/\elm{i}{j}(\kerI{\varphi})$ with the corresponding off-diagonal coordinate  of the additive matrix group $\Off{d}{\finI{\varphi}/\kerI{\varphi}} $.  The surjectivity claim implicit in the short exact sequence in question follows. Furthermore the image $\iota \Fn{d}{\finI{\varphi}}$ is isomorphic to the direct sum of the images of   the elementary subgroups. The exactness of the sequence is now a consequence of the universal property of direct sums.
\end{proof}

\begin{cor}
\label{cor:iota is isomorphism on N/Z}
The map $\iota$ induces an isomorphism of the subquotient $\nF{\varphi} / \zF{\varphi}$ with a subgroup  of the additive matrix group $\Mnzt{d}{\rF{\varphi}}$.
\end{cor}

Note that the subquotient  $\zF{\varphi}$ is central in the quotient group $\EL{d}{\R}/\ker\varphi$ by  Theorem \ref{thm:Valivov}.

\begin{cor}
	\label{cor:nF is nilpotent}
The subquotient group $\nF{\varphi} / \ker\varphi $  is two-step nilpotent.
\end{cor}
\begin{proof}
This follows from Lemma \ref{lem:exact sequence with i}  as 
$ \left[\nF{\varphi}, \nF{\varphi}  \right] \le \zF{\varphi}   \le Z(\EL{d}{\R} / \ker\varphi)$.
\end{proof}


\subsection*{The kernel and and  degeneracy ideals are commensurable}

Let $\varphi$ be a character of the group $\EL{d}{\R}$.  
Let $\dF{\varphi}$ be the preimage in the group $\EL{d}{\R}$ of the center of the nilpotent subquotient   $\nF{\varphi} / \ker\varphi$. In other words
$$ \dF{\varphi} / \ker \varphi = Z(\nF{\varphi}/ \ker \varphi).$$
 It is clear that  $\dF{\varphi} \nrm \EL{d}{\R}$. 
By the normal subgroup structure theorem (cited here as Theorem \ref{thm:normal structure theorem}) there is a uniquely determined  ideal $\degI{\varphi} \nrm \R$ satisfying
$$ \EL{d}{\degI{\varphi}} \le \dF{\varphi}\le \SLtil{d}{\degI{\varphi}}. $$
We will say that $\degI{\varphi}$ is the \emph{degeneracy ideal} associated to the character $\varphi$. The ideals in question form an ascending chain as follows
$$ \kerI{\varphi} \subset \degI{\varphi} \subset \finI{\varphi} \subset \levI{\varphi} \subset \R. $$

Our   goal of showing that $\mfi{\kerI{\varphi}} = \levI{\varphi}$ is achieved in two steps.  Namely, we first show that  $|\sfrac{\degI{\varphi}}{\kerI{\varphi}}| <  \infty$ and then show that $|\sfrac{\finI{\varphi}}{\degI{\varphi}}| <  \infty$.  Putting together these two steps implies   that $\mfi{\kerI{\varphi}} = \levI{\varphi}$ as $\mfi{\finI{\varphi}} = \levI{\varphi}$ is already known. We turn our attention to  the first step.


\begin{theorem}
\label{thm:the depth of K is I}
$|\sfrac{\degI{\varphi}}{\kerI{\varphi}}| <  \infty$.
\end{theorem}

The proof of Theorem \ref{thm:the depth of K is I} involves a careful analysis of Pontryagin dual groups of additive matrix groups over certain ideals.
To be precise, denote  $\aF{\varphi} = \degI{\varphi} / \kerI{\varphi}$.
We regard $\aF{\varphi}$ as an ideal in the ring $\rF{\varphi}$.
  Let $\aFf{\varphi}$ denote the   subgroup of the Pontryagin dual group $\aFd{\varphi}$ given by
$$ \aFf{\varphi} = 
\bigcup_{\kerI{\varphi} \le \J \nrm \R, \, \mfi{\J} = \levI{\varphi}} \ann{(\J \cap \degI{\varphi})}.
$$
 The subgroup  $\aFf{\varphi}$ is a direct limit of countably many finite groups.  In particular $\aFf{\varphi}$ is countable. 
 
The map $\iota$ identifies the abelian subquotient $\dF{\varphi}/\zF{\varphi}$   with a subgroup of the matrix group $ \Mnzt{d}{\aF{\varphi}}$ containing $\Off{d}{\aF{\varphi}}$, see Corollaries \ref{cor:induced exact sequence} and \ref{cor:iota is isomorphism on N/Z}.
 Therefore $\widehat{\dF{\varphi}/\zF{\varphi}}$ is a   quotient of the dual group $\widehat{\Mn{d}{\aF{\varphi}}} \cong \Mn{d}{\aFd{\varphi}}$. Let $ \Theta_\varphi$ denote
  the image of the subgroup $\Mn{d}{\aFf{\varphi}}$ with respect to this  quotient map.
As the group $\Mn{d}{\aFf{\varphi}}$ is $\EL{d}{\rF{\varphi}}$-invariant  its quotient $\Theta_\varphi$   is $\EL{d}{\rF{\varphi}}$-invariant as well.

 \begin{prop}
 \label{prop:a criterion to be in aFf}
 A character $\chi \in \widehat{\dF{\varphi}/\zF{\varphi}}$ belongs to the subgroup $ \Theta_\varphi$ if and only if $(g\chi)_{|\Off{d}{\aF{\varphi}}} \in \Off{d}{\aFf{\varphi}}$ for every element $g \in \EL{d}{\rF{\varphi}} $.
 \end{prop}

\begin{proof}

Consider the two annihilator subgroups 
$$ D = \ann{\Mnzt{d}{\aF{\varphi}}} = \left\{ \psi \mathrm{Id}_d \: : \: \psi \in \aFd{\varphi} \right\}  $$
and
 $ C = \ann{(\dF{\varphi}/\zF{\varphi})}$ so that $D \le C  \le \Mn{d}{\aFd{\varphi}}$.
We have  by definition that
$$  \Theta_\varphi = \Mn{d}{\aFf{\varphi}} + C \le \Mn{d}{\aFd{\varphi}}/C.$$

Let $\overline{\chi} = \chi + C \in \widehat{\dF{\varphi}/\zF{\varphi}}$ be given for some  character  $\chi \in  \Mn{d}{\aFd{\varphi}}$.
If $\overline{\chi} \in \Theta_\varphi $  then $g\overline{\chi} \in  \Theta_\varphi$ as well for all elements $g \in \EL{d}{\rF{\varphi}}$. The conclusion in the \enquote{only if} direction follows. 

Arguing in the \enquote{if}  direction, write $\chi = \chi_1 + \chi_2$ where $\chi_1 \in \Off{d}{\aFd{\varphi}}$ and $\chi_2 = \mathrm{diag}(\psi_1,\ldots,\psi_d) \in \Mn{d}{\aFd{\varphi}}$.  The   assumption with respect to the trivial  element $g= \mathrm{Id}_d \in \EL{d}{\rF{\varphi}}$ immediately  implies  that $\chi_1 \in \Off{d}{\aFf{\varphi}} \le \Theta_\varphi$. It remains to show that $\chi_2 + C \in \Theta_\varphi$ as well. To see this, note that the character $\elm{i}{j}(1) \chi_2$ is given by
$$ (\elm{i}{j}(1) \chi_2)_{k,l} = \begin{cases} \psi_i - \psi_j & \text{$i = k$ and $j = l$} \\ 0 & \text{otherwise} \end{cases} $$
for all pairs of distinct indices $i,j \in \{1,\ldots,d\}$. 
The assumption gives $\psi_i - \psi_j \in \aFf{\varphi}$ for all such pairs $i$ and $j$. In other words, all possible  differences of   entries in the diagonal element $\chi_2$   belong to $\aFf{\varphi}$. This implies that $\chi_2 \in \Mn{d}{\aFf{\varphi}} + D$. The desired conclusion follows.
\end{proof}

 
%

We know that the function $|\varphi|^2 : \EL{d}{\R} \to \mathbb{R}$ satisfies  $|\varphi|^2 \in \traces{\EL{d}{\R}}$ by Proposition \ref{prop:square power is a trace}.  Moreover $\zF{\varphi} \le \ker |\varphi|^2$ according to Corollary \ref{cor:center in kernel of square}.
Consider the restriction of the trace $|\varphi|^2$ to the subquotient group $\dF{\varphi}/\zF{\varphi}$ identified with a subgroup of the   matrix group $\Mnzt{d}{\aF{\varphi}}$ in the manner  of Corollary \ref{cor:iota is isomorphism on N/Z}.



%

\begin{prop}
\label{prop:probability measure for phi squared}
 $ |\varphi|^2_{|(\dF{\varphi}/\zF{\varphi})}  = \tFd{\varphi}$  some $\EL{d}{\rF{\varphi}}$-invariant Borel probability measure $\tF{\varphi}$ supported on the subgroup $\Theta_\varphi \le \widehat{\dF{\varphi}/\zF{\varphi}}$.
\end{prop}

\begin{proof} 
Throughout this proof we will be implicitly relying on the isomorphism   $\Fn{d}{\degI{\varphi}}/\Fn{d}{\kerI{\varphi}} \cong \Off{d}{\aF{\varphi}}$ established in Corollary \ref{cor:induced exact sequence}. In particular, there exists a Borel probability measure  $\lambda_\varphi$ on the dual group $\Off{d}{\aFd{\varphi}}$ satisfying $\lambda_\varphi \left( \Off{d}{\aFf{\varphi}} \right) = 1$ and   $ \varphi_{|\Fn{d}{\degI{\varphi}}} = \widehat{\lambda_\varphi} $.
This fact  is merely a reformulation of  Theorem \ref{thm:existence of a depth ideal} in terms of  our standing notations. As  $ \Off{d}{\aFf{\varphi}}$ is a subgroup  the convolution $\lambda_\varphi * \bar{\lambda}_\varphi$ satisfies $(\lambda_\varphi * \bar{\lambda}_\varphi)\left( \Off{d}{\aFf{\varphi}}\right) = 1$ as well.

Consider the unique $\EL{d}{\rF{\varphi}}$-invariant Borel probability measure  $\tF{\varphi}$  on the dual   group $\widehat{\dF{\varphi}/\zF{\varphi}}$ satisfying $ |\varphi|^2_{|(\dF{\varphi}/\zF{\varphi})}  = \tFd{\varphi}$. 
As $|\varphi|^2_{|\Fn{d}{\degI{\varphi}}} =   \widehat{\lambda_\varphi * \bar{\lambda}_\varphi}$ it follows from the previous paragraph that $g_*\tF{\varphi}$-almost every\footnote{Of course $g_* \theta_\varphi = \theta_\varphi$ for all elements $g \in \EL{d}{\rF{\varphi}}$. We are using the expression  $g_* \theta_\varphi$ explicitly to make it clear that the assumptions of Proposition  \ref{prop:a criterion to be in aFf} are satisfied.} character $\chi \in   \widehat{\dF{\varphi}/\zF{\varphi}}$ satisfies $\chi_{|\Off{d}{\aF{\varphi}}} \in \Theta_\varphi$ for all elements $g \in \EL{d}{\rF{\varphi}}$. The desired conclusion follows from Proposition  \ref{prop:a criterion to be in aFf}.
\end{proof}

\begin{prop}
\label{prop:finite orbit on D}
$\varphi_{|\dF{\varphi}} = \widehat{\eta}_\varphi$ for some atomic probability measure $\eta_\varphi$ supported on a finite $\EL{d}{\R}$-orbit in $\chars{\dF{\varphi}}$.
\end{prop}
\begin{proof}
It is a consequence of   Proposition \ref{prop:probability measure for phi squared} that 
$ |\varphi|^2_{|(\dF{\varphi}/\zF{\varphi})}  = \tFd{\varphi}$ for some \emph{atomic} probability measure $\tF{\varphi}$ on the dual   group $\widehat{\dF{\varphi}/\zF{\varphi}}$. 
In fact we may regard $\tF{\varphi}$  as a Borel probability measure on the larger dual abelian group $\widehat{\dF{\varphi}/\ker\varphi}$.
 As such it is possible to  write $\theta_\varphi = \eta_\varphi * \bar{\eta}_\varphi$ where $\eta_\varphi$  is  the unique probability measure on the dual group $\widehat{\dF{\varphi}/\ker\varphi}$   satisfying $\varphi_{|\dF{\varphi}} = \widehat{\eta}_\varphi$. The probability measure $\eta_\varphi$ must be \emph{atomic} as well by   Proposition \ref{prop:convolution has atoms if and only if both have atoms}.  Finally, since the 
 probability measure $\eta_\varphi$ is ergodic and $\EL{d}{\rF{\varphi}}$-invariant it must be  supported on a single finite $\EL{d}{\rF{\varphi}}$-orbit. This orbit   can be identified with a subset of  $\chars{\dF{\varphi}}$.
\end{proof}

We are ready to complete our first step towards the proof of Theorem \ref{thm:ker has finite index in lev}, namely proving that the kernel ideal is commensurable to the degeneracy ideal.

\begin{proof}[Proof of Theorem \ref{thm:the depth of K is I}]
Let $\varphi \in \chars{\EL{d}{\R}}$ be a character with kernel ideal $\kerI{\varphi}$, degeneracy ideal $\degI{\varphi}$ and  level ideal $\levI{\varphi}$.

 The relative character $\varphi_{|\dF{\varphi}} \in \charrel{\EL{d}{\R}}{\dF{\varphi}}$ is a convex combination over a finite  $\EL{d}{\R}$-orbit in the space $\chars{\dF{\varphi}}$ by Proposition \ref{prop:finite orbit on D}.
Write $$ \varphi_{|\dF{\varphi}} = \frac{1}{N} \sum_{i=1}^N \chi_i$$ for some  integer $N \in \mathbb{N}$ and   characters $\chi_1,\ldots,\chi_N \in \chars{\dF{\varphi}}$. 
The transitivity of the $\EL{d}{\R}$-action on the characters $\chi_i$ allows   to find  an  ideal $\J \nrm \R$ satisfying $\kerI{\varphi} \subset \J, \mfi{\J} = \levI{\varphi}$ and
$$ \chi_{i|\Fn{d}{\degI{\varphi}}} \in \Off{d}{\ann{(\J \cap \degI{\varphi})}} $$
 for all indices $i \in \{1,\ldots,N\}$. The above statement makes an implicit use of the   isomorphism   $\Fn{d}{\degI{\varphi}}/\Fn{d}{\kerI{\varphi}} \cong \Off{d}{\aF{\varphi}}$, see    Corollary \ref{cor:induced exact sequence}.
 
  Consider the intersection     ideal
$ \LL =  \J  \cap \degI{\varphi} \nrm \R$. Note that $\mfi{\LL} = \mfi{\degI{\varphi}}$ from Proposition \ref{prop:properties of depth}.
The ideal $\LL$ satisfies
$$\Fn{d}{\LL}   \le \Fn{d}{\degI{\varphi}} \cap \bigcap_{i=1}^N \ker \chi_{i} \le \ker\varphi.$$
Recall that the subgroup  $\EL{d}{\LL}$ was defined as the normal closure of  $\Fn{d}{\LL}$. Therefore  $\EL{d}{\LL} \le \ker \varphi$ and  $\LL \subset \kerI{\varphi}$.  We may conclude that $\mfi{\kerI{\varphi}} = \mfi{\degI{\varphi}}$  relying once more on Proposition \ref{prop:properties of depth}.
\end{proof}

\subsection*{The degeneracy and the supporting ideals are commensurable}
Let $\varphi$ be a character of the group $ \EL{d}{\R}$   with kernel ideal $\kerI{\varphi}$, degeneracy ideal $\degI{\varphi}$, supporting ideal $\finI{\varphi}$  and  level ideal $\levI{\varphi}$.

\begin{prop}
\label{prop:Howe revisited}
$\varphi \in \Ind{\EL{d}{\R}}{\nF{\varphi}}{\dF{\varphi}} $.
\end{prop}
\begin{proof}
The restriction $\varphi_{|\nF{\varphi}}$ is a faithful relative character of the   subquotient group $\nF{\varphi} / \ker\varphi$. Recall that $\left[\nF{\varphi},\nF{\varphi}\right] \le \zF{\varphi} \le Z(G)$ as in Corollary \ref{cor:nF is nilpotent}. Moreover $\dF{\varphi}/\ker\varphi = Z(\nF{\varphi}/\ker \varphi)$ by definition. The   conclusion follows from Lemma \ref{lem:relative Howe}.
\end{proof}

We   proceed with Theorem \ref{thm:deg has finite index}, which is the second  and final step towards proving Theorem \ref{thm:ker has finite index in lev}.

\begin{theorem}
\label{thm:deg has finite index}
$|\sfrac{\finI{\varphi}}{\degI{\varphi}}| <  \infty$.
 \end{theorem}

%

\begin{proof} 
Recall that $\varphi \in \chars{\EL{d}{\R}}$ is a character with kernel ideal $\kerI{\varphi}$, degeneracy ideal $\degI{\varphi}$, supporting ideal $\finI{\varphi}$ and  level ideal $\levI{\varphi}$.
Assume towards contradiction that  $|\sfrac{\finI{\varphi}}{\degI{\varphi}}| =  \infty$. 
We claim that 
$$\varphi \in \Ind{\EL{d}{\R}}{\elm{1}{2}(\levI{\varphi})}{\elm{1}{2}(\finI{\varphi})}.$$ 
Indeed, consider any element $g \in \elm{1}{2}(\levI{\varphi})  \setminus  \elm{1}{2} (\finI{\varphi})$.   The assumption     $|\sfrac{\finI{\varphi}}{\degI{\varphi}}| =  \infty$ allows us to apply    Proposition \ref{prop:not in finite index implies infinite many infinite index}   with respect to the  ideals
$$ \LL = \degI{\varphi}, \quad \J = \finI{\varphi} \quad \text{and}\quad \I = \levI{\varphi} $$
and obtain   a sequence elements $x_n \in \EL{d}{\R}$ such that the commutators $[g,x_n] $   belong to $ \nF{\varphi} = \SL{d}{\finI{\varphi}} \cap \EL{d}{\R}$ and are   pairwise distinct modulo the normal subgroup $\dF{\varphi}$   for all $n \in \NN$. Recall that $\varphi \in \Ind{\EL{d}{\R}}{\nF{\varphi}}{\dF{\varphi}}$ as was established in Proposition  \ref{prop:Howe revisited}. We   conclude that  $\varphi(g) = 0$ relying on Lemma 	\ref{lem:application of lemma on sequences of elements for vanishing} applied with respect to the subgroups
$ N = \dF{\varphi} $ and $  H = \nF{\varphi}$.
The claim follows.

The assumption towards contradiction leads to the conclusion     $\varphi(\elm{1}{2}(r)) = 0 $ for all ring elements $r \in \R \setminus \degI{\varphi}$. Indeed:
\begin{itemize}
\item if  $r \in \R \setminus \levI{\varphi}$ then $ \varphi(\elm{1}{2}(r)) = 0 $ from the properties of the level ideal $\levI{\varphi}$ (see Proposition \ref{prop:vanishing on elementary matrices outside level}), 
\item if $r \in \levI{\varphi} \setminus \finI{\varphi}$ then  $ \varphi(\elm{1}{2}(r)) = 0 $ from the claim made in the previous paragraph, and
\item if $r \in \finI{\varphi} \setminus \degI{\varphi}$ then $ \varphi(\elm{1}{2}(r)) = 0 $ from Proposition  \ref{prop:Howe revisited}.
\end{itemize}
As $|\sfrac{\finI{\varphi}}{\degI{\varphi}}| =  \infty$ we arrive at a  contradiction to yet another property   of the level ideal $\levI{\varphi}$, namely    Proposition \ref{prop:level is uniquely determined}.  
\end{proof}

Since  $|\sfrac{\levI{\varphi}}{\finI{\varphi}}| < \infty$ essentially by definition (see Lemma  \ref{lem:bekkas lemma on finite index ideal}),
the two Theorems \ref{thm:the depth of K is I} and
\ref{thm:deg has finite index} put together imply that $|\sfrac{\levI{\varphi}}{\kerI{\varphi}}| <  \infty$. In other words $\mfi{\kerI{\varphi}} = \levI{\varphi}$. The proof of 
Theorem \ref{thm:ker has finite index in lev} is now complete.

\section{The character $\varphi$ is induced from $\SLtil{d}{\levI{\varphi}} \cap \EL{d}{\R}$}
\label{sec:induced from level}

Let $\R$ be a Noetherian ring and $d \ge 3$ be a fixed integer. 
Let $\varphi \in \chars{\EL{d}{\R}}$ be a character  with level ideal $\levI{\varphi}$ and kernel ideal $\kerI{\varphi}$.   These two ideals satisfy $\mfi{\kerI{\varphi}} = \levI{\varphi}$ by Theorem 
\ref{thm:ker has finite index in lev}. 
The main goal of   \S\ref{sec:induced from level} is the following.

\begin{thm}
\label{thm:vanishing outside big S_varphi using normal form decomposition}
If $d > \sr{\R}$ then the character $\varphi$ is induced from the normal subgroup $\SLtil{d}{\levI{\varphi}} \cap \EL{d}{\R}$.
\end{thm}

 Recall that a character of a group $G$ is said to be induced from the subgroup $H$ if it vanishes on the complement $G \setminus H$.
 
The statement of   Theorem \ref{thm:vanishing outside big S_varphi using normal form decomposition} is vacuous if $\levI{\varphi} = \R$. For this reason we will assume for the remainder of \S\ref{sec:induced from level} that the level ideal $\levI{\varphi}$ is a proper ideal.

The proof of Theorem \ref{thm:vanishing outside big S_varphi using normal form decomposition} is built out of a sequence  of several Propositions gradually enlarging the subset of the group $\EL{d}{\R}$ where the character $\varphi$ is known to vanish. We use the notation $\mathrm{Ind}$ introduced on \S\ref{sec: character theory}, page \pageref{def Ind}.

\begin{remark}
Strictly speaking, we are abusing the notation $\mathrm{Ind}$ in making the following statements, for the groups $\SLtil{d}{\levI{\varphi}}$ etc. are not necessarily subgroups of $\EL{d}{\R}$. In that case our notation is to be understood in the sense of taking the respective intersections with the group $\EL{d}{\R}$.
\end{remark}

\begin{prop}
\label{prop:first vanishing}
$\varphi \in  \Ind{\EL{d}{\R}}{\SLtil{d}{\kerI{\varphi}} \ver{1}(\R)}{\SLtil{d}{\kerI{\varphi}}  \ver{1}(\levI{\varphi})}$.
\end{prop}
\begin{proof}
Recall that $\EL{d}{\kerI{\varphi}} \le \ker \varphi$ and   $\SLtil{d}{\kerI{\varphi}} \cap \EL{d}{\R} \le Z(\EL{d}{\R} / \ker\varphi)$.   Schur's lemma says that the restriction of the character $\varphi$ to the  subgroup $\SLtil{d}{\kerI{\varphi}}$  is multiplicative  (see Lemma 
\ref{lem:Schur's lemma for characters} for details). This means that $\varphi(g) = 0$ if and only $\varphi(gh) = 0$ for all elements $g \in \EL{d}{\R}$ and $h \in \SLtil{d}{\kerI{\varphi}} \cap \EL{d}{\R}$. On the other hand, the properties of the level ideal $\levI{\varphi}$ and Proposition   \ref{prop:vanishing on elementary matrices outside level} in particular imply that
$$ \varphi \in \Ind{\EL{d}{\R}}{\ker \varphi \, \ver{1}(\R) }{\ker \varphi \, \ver{1}(\levI{\varphi}) }.$$
These two facts conclude the proof.
\end{proof}

Let $\rF{\varphi}$ denote the quotient ring $\R/\levI{\varphi}$ modulo the level ideal $\levI{\varphi}$.

\begin{prop}
\label{prop:vanishing on a semidirect product goes up a finite index on elementary}
$\varphi \in 
 \Ind{\EL{d}{\R}}{\SLtil{d}{\levI{\varphi}}\ver{1}(\R)}{\SLtil{d}{\levI{\varphi}}}
$.
\end{prop}
 \begin{proof}
Let  $g \in \EL{d}{\R} \cap {\SLtil{d}{\levI{\varphi}}\ver{1}(\R)}  $ be any element. Assume that $g \notin {\SLtil{d}{\levI{\varphi}}}$. The element $g$ can be written as $g = hv$ for a pair of  elements $ h\in \SLtil{d}{\levI{\varphi}}$ and $v \in \ver{1}(\R) \setminus \ver{1}(\levI{\varphi})$.
In particular there is some index $ j\in\{2,\ldots,d\} $ such that $v_{j,1} \notin \levI{\varphi}$. Fix an arbitrary index $ k \in \{2,\ldots,d\} \setminus \{j\}$.   The reduction modulo the level ideal $\levI{\varphi}$ of the element $v$ does not centralize the elementary subgroup $\elm{k}{j}(\rF{\varphi}) 
\le \EL{d}{\rF{\varphi}}$ by Proposition \ref{prop:centraliser of an elementary matrix with unit entry}.
According to Item (\ref{part 3}) of Proposition \ref{prop:not commuting implies infinitely many pairwise disjoint, in N}
 there are elements   $x_n \in \elm{k}{j}(\R)$
 such that the commutators
 $$\left[x_m^{-1} x_n, v\right] \in \ver{1}(\R)$$
  all belong to pairwise distinct cosets of $\ver{1}(\levI{\varphi})$ for all $n, m\in\NN$ with $n < m$.

The subquotient group $\SLtil{d}{\levI{\varphi}} / \SLtil{d}{\kerI{\varphi}}$ is finite according to Proposition   \ref{prop: congurence subgroups of ideal is of finite index in the congurence subgroup of its depth}.
 In particular, the  action by conjugation of the group $\EL{d}{\R}$ on this subquotient has finite orbits. By the pigeon hole principle and up to passing to a further subsequence, we may assume that
$ x_n^{-1} h x_n = x_m^{-1} h x_m $
modulo the normal subgroup $\SLtil{d}{\kerI{\varphi}}$ and for every $n < m$. 
Using standard  commutator identities we now deduce that
	$$ \left[x_m^{-1} x_n,g\right] = \left[x_m^{-1} x_n,hv\right] = \left[x_m^{-1} x_n,v\right] \left[x_m^{-1} x_n,h\right]^v  $$
Note that $\left[x_m^{-1} x_n,v\right] \in \ver{1}(\R)$ and $\left[x_m^{-1} x_n,h\right]^v \in \SLtil{d}{\kerI{\varphi}}$.

Consider the trace $|\varphi|^2 \in \traces{\EL{d}{\R}}$, see Proposition \ref{prop:square power is a trace}. Clearly $\varphi(h) = 0$ if and only if $|\varphi|^2(h) = 0$ for any element $h \in \EL{d}{\R}$.  So 
$|\varphi|^2 \in 
\Ind{\EL{d}{\R}}{H}{N}$ with respect to the normal subgroups
$$ H = \SLtil{d}{\kerI{\varphi}} \ver{1}(\R) \quad \text{and} \quad N = \SLtil{d}{\kerI{\varphi}} \ver{1}(\levI{\varphi})$$
and according to Proposition \ref{prop:first vanishing}.  Moreover $\SLtil{d}{\kerI{\varphi}} \le \ker |\varphi|^2$ by Corollary \ref{cor:center in kernel of square}.

As the elements
$ \left[x_m^{-1} x_n,v\right]  $ of $\ver{1}(\R)$  are all distinct modulo $\ver{1}(\levI{\varphi})$, we  may conclude that $\varphi(g) = 0$ from Lemma \ref{lem:application of lemma on sequences of elements for vanishing} applied with respect to the  subgroups $H$ and $N$.
\end{proof}

\begin{prop}
\label{prop:vanishing on outside centralizer part 2}
$\varphi \in   \Ind{\EL{d}{\R}}{\SLtil{d}{\levI{\varphi}} \mathrm{N}_{\EL{d}{\R}}(\ver{1}(\R))   }{\SLtil{d}{\levI{\varphi}} }$.
\end{prop}
\begin{proof}
Let $g \in \EL{d}{\R}  \cap \SLtil{d}{\levI{\varphi}} \mathrm{N}_{\EL{d}{\R}}(\ver{1}(\R))  $ be any element. We may   assume without loss of generality  that $g \notin \SLtil{d}{\levI{\varphi}} \ver{1}(\R) $ for otherwise the result follows from the previous Proposition \ref{prop:vanishing on a semidirect product goes up a finite index on elementary}.

 The reduction    modulo the level ideal $\levI{\varphi}$      of the element $g$   does not  centralize the  vertical group $\ver{1}(\rF{\varphi})$  in the group $\EL{d}{\rF{\varphi}}$, see Corollary \ref{cor:centraliser of a vertical/horizontal}. 
It follows from Item (\ref{part 2}) of  Proposition \ref{prop:not commuting implies infinitely many pairwise disjoint, in N} applied with respect to the quotient ring $\rF{\varphi}$ that there is a sequence of elements
$x_n \in \ver{1}(\R)  $ such that $\left[g,x_n\right] \in \ver{1}(\R)$ and those commutators belong to  pairwise distinct cosets of   $\ver{1}(\levI{\varphi})$. 
Consider the subgroups
  $$H =\SLtil{d}{\levI{\varphi}} \ver{1}(\R), K = \mathrm{N}_{\EL{d}{\R}}(\ver{1}(\R)), L = \{e\} \quad \text{and} \quad N = \SLtil{d}{\levI{\varphi}}.$$
In particular, it is clear that  $x_n \in \mathrm{N}_G(H)$, $\left[K,x_n\right] \subset H$ and $ \left[L,x_n\right] = \{e\} \subset H  $.    We conclude that $\varphi(g) = 0$ relying on Lemma \ref{lem:vanishing with two subgroups} applied with respect to these subgroups $H,K,L$ and $N$.
\end{proof}

The previous three propositions are stated for the vertical group $\ver{1}(\R)$. Of course, analogous results hold true for all other vertical groups $\ver{i}(\R)$ as well as     the horizontal groups $\hor{j}(\R)$, but this fact is not needed below.

\begin{cor}
\label{cor:vanishing if  centralizes an elementary group}
Let $g \in \EL{d}{\R} \setminus  \SLtil{d}{\levI{\varphi}}$ be an element whose reduction modulo the level ideal $\levI{\varphi}$    in the group  $\EL{d}{\rF{\varphi}}$ centralizes   the elementary group $\elm{j}{i}(\rF{\varphi}) $ for a given pair of distinct indices $i,j \in \{1,\ldots,d\}$. Then  $\varphi(g) = 0$.
\end{cor}
\begin{proof}
The centralizer of the elementary group $\EL{d}{\rF{\varphi}}$ satisfies 
$$ C_{\EL{d}{\rF{\varphi}}}(\elm{j}{i}(\rF{\varphi})) \le
Z(\EL{d}{\rF{\varphi}})
\left( \ngp{i}(\rF{\varphi}) \ver{i}(\rF{\varphi}) \cap \ngp{j}(\rF{\varphi}) \hor{j}(\rF{\varphi}) \right).
$$
See Proposition   \ref{prop:centraliser of an elementary matrix with unit entry} for details.
Combined with the normalizer computations done in Proposition \ref{prop:normalizer of vertical or horizontal subgroup} this means that
$$g \in {\SLtil{d}{\levI{\varphi}} \mathrm{N}_{\EL{d}{\R}}(\ver{i}(\R))   } \setminus {\SLtil{d}{\levI{\varphi}} }.  $$
The fact that $\varphi(g) = 0$ follows immediately from Proposition \ref{prop:vanishing on outside centralizer part 2}.
\end{proof}

We are ready to establish that the character   $\varphi$ is induced from the subgroup   $\SLtil{d}{\levI{\varphi}} \cap \EL{d}{\R}$ corresponding to its level ideal $\levI{\varphi}$.


\begin{proof}[Proof of Theorem \ref{thm:vanishing outside big S_varphi using normal form decomposition}]
Consider an element $g \in  \EL{d}{\R} \setminus \SLtil{d}{\levI{\varphi}}$. We may without loss of generality  replace $g$ by any of its conjugates as the character   $\varphi$ is conjugation invariant. Relying on the assumption that $d > \sr{\R}$ and making use of the normal form  decomposition introduced in Proposition 	\ref{prop:normal form for conjugates}, we may  replace $g$ by a suitable conjugate and write
$g = hnm$
for some corresponding three elements
$$  h\in\hor{1}(\R), \quad n \in \ngp{2}(\R)\hor{2}(\R) \quad \text{and} \quad m \in \ngp{1}(\R)\ver{1}(\R). $$

First  consider the special case where the element $h$ is trivial. This means that 
$$g = nm \in \ngp{2}(\R)\hor{2}(\R)\ngp{1}(\R)\ver{1}(\R) \setminus \SLtil{d}{\levI{\varphi}}.$$
 If the reduction modulo the level ideal $\levI{\varphi}$ of the element $g$ in $\SL{d}{\rF{\varphi}}$ centralizes the elementary group $\elm{2}{1}(\rF{\varphi})$ then $\varphi(g) = 0$ according to Corollary \ref{cor:vanishing if  centralizes an elementary group}. Otherwise there are elements $x_n \in \elm{2}{1}(\R) \le \ver{1}(\R)$  such that the commutators $[g,x_n]$ are pairwise distinct modulo the normal subgroup $\SLtil{d}{\levI{\varphi}}$ according to Item (\ref{part 1}) of Proposition \ref{prop:not commuting implies infinitely many pairwise disjoint, in N}. Consider the subgroups
 $$ K = \{e\}, \quad H = \ngp{2}(\R) \hor{2}(\R), \quad L= \ngp{1}(\R)\ver{1}(\R), \quad \text{and} \quad N = \SLtil{d}{\levI{\varphi}}.$$
 Note that $x_n \in H, [x_n, K] \subset H$ and $[x_n, L] \subset \ver{1}(\R) \le H$. The  fact that $\varphi(g) = 0$ follows from Lemma  \ref{lem:vanishing with two subgroups} applied with respect to these subgroups $K,H,L$ and $N$.

Next consider the general  case, namely $g = hnm$ with the element $h \in \hor{1}(\R)$ being  arbitrary.
  If the reduction modulo the level ideal $\levI{\varphi}$ of the element $g$ in $\EL{d}{\rF{\varphi}}$   centralizes the elementary group $\elm{2}{3}(\rF{\varphi})$ then $\varphi(g) = 0$ according to Corollary \ref{cor:vanishing if  centralizes an elementary group}. Otherwise according to Item (\ref{part 1}) of Proposition \ref{prop:not commuting implies infinitely many pairwise disjoint, in N} there are elements $  y_n   \in \elm{2}{3}(\R)$ such that the commutators $[g,y_n]$ are pairwise distinct  modulo the normal subgroup $\SLtil{d}{\levI{\varphi}}$. 
  Consider the subgroups 
   $$ K = \hor{1}(\R), \quad H = \ngp{2}(\R)\hor{2}(\R)\ngp{1}(\R)\ver{1}(\R), \quad L= \{e\}, \quad \text{and} \quad N = \SLtil{d}{\levI{\varphi}}.$$
     Note that $y_n \in H$ and $[y_n, L] \subset H$. Moreover  Proposition \ref{prop:commutator of elementary and horizontal/vertical} gives
     $$[y_n, K] \subset \elm{1}{3}(\R) \le \ngp{2}(\R) \le H.$$
   Recall that $\varphi(g) = 0$ for all elements $g \in  \ngp{2}(\R)\hor{2}(\R)\ngp{1}(\R)\ver{1}(\R) \setminus \SLtil{d}{\levI{\varphi}} $ as established in the first case  of this proof.
 This implies that
  $\varphi(g) = 0$ relying on Lemma  \ref{lem:vanishing with two subgroups} applied with respect to these subgroups $K,H,L$ and $N$.
\end{proof}

\section{Classification of characters}
\label{sec:induced from abelian}

Let $\R$ be a commutative Noetherian ring with unit. Fix an integer    $d > \max\{\sr{\R},2\}$.  The goal of the current section is to complete the classification of the characters of the group $\EL{d}{\R}$ thereby proving our main results Theorems \ref{thm:main theorem} and \ref{thm:main converse}.

Consider any fixed character $\varphi \in \chars{\EL{d}{\R}}$. Let $\levI{\varphi} \nrm \R$ be the level ideal associated to the character $\varphi$ as established in Theorem \ref{thm:existence of a depth ideal}.  Let $\kerI{\varphi} \nrm \R$ be the   kernel  ideal defined so that
$$\EL{d}{\kerI{\varphi}} \le \ker \varphi \le \SLtil{d}{\kerI{\varphi}}.$$
We proved in Theorem \ref{thm:ker has finite index in lev} that $\mfi{\kerI{\varphi}} = \levI{\varphi}$. 
Consider the subquotient group
 $$ \mathcal{A}_d(\varphi) = \mathcal{A}_d(\levI{\varphi},\kerI{\varphi}) = \left( \SLtil{d}{ \levI{\varphi}} \cap \EL{d}{\R} \right) / \ker \varphi.$$
The group $\mathcal{A}_d(\varphi)$  is  virtually central in the quotient group $\EL{d}{\R}/\ker \varphi$ (and is in particular virtually abelian) according to Theorem \ref{thm:virtually abelian}. 
There is a natural $\EL{d}{\R}$-action on the set $\chars{\mathcal{A}_d(\varphi)}$ given by
$$ (g\psi)(x) = \psi(x^g)\quad \forall \psi \in \chars{\mathcal{A}_d(\varphi)}, x \in \mathcal{A}_d(\varphi)$$
for all elements $g\in \EL{d}{\R}$. Given a   subset $O \subset  \chars{\mathcal{A}_d(\varphi)}  $ define its annihilator
$$\mathrm{Ann}(O) = \bigcap_{\psi \in O} \ker \psi \le \mathcal{A}_d(\varphi).$$
If the subset $O$ is $\EL{d}{\R}$-invariant then $\mathrm{Ann}(O)$ is a normal subgroup of $\mathcal{A}_d(\varphi)$. In particular  $\mathrm{Ann}(O)$ is associated to some ideal in the ring $\R$ in the sense of Theorem \ref{thm:normal structure theorem}. We   say that the subset $O$ is  \emph{essential} if $\mathrm{Ann}(O) \le \SLtil{d}{\kerI{\varphi}}$.

\begin{proof}[Proof of Theorem \ref{thm:main theorem}]
The fact that the character $\varphi$ is induced from the normal subgroup $\SLtil{d}{\levI{\varphi}} \cap \EL{d}{\R}$   is the content of  Theorem   \ref{thm:vanishing outside big S_varphi using normal form decomposition}. 

Recall that the subquotient group $\mathcal{A}_d(\varphi)$ is virtually central in the quotient group $\EL{d}{\varphi}/\ker \varphi$.   It follows from Proposition \ref{prop:relative character of virtually central}  that there is a finite $\EL{d}{\R}$-orbit   $O_\varphi \subset \chars{\mathcal{A}_d(\varphi)}$ such that
$$ \varphi_{| \mathcal{A}_d(\varphi)} = \frac{1}{|\Orb{\varphi}|} \sum_{\psi \in \Orb{\varphi}} \psi.$$
It remains to show that the orbit $\Orb{\varphi}$ is essential. Note that
$$   \mathrm{Ann}(\Orb{\varphi}) =  \SLtil{d}{\levI{\varphi}} \cap \ker \varphi \le \SLtil{d}{\kerI{\varphi}} .$$ This concludes the proof.
\end{proof}

It is now straightforward to conclude that any character $\varphi$ of the group $\EL{d}{\R}$ is induced from a finite dimensional representation.

\begin{proof}[Proof of Corollary \ref{cor:finitely induced}]
This follows immediately from Theorem \ref{thm:main theorem} combined with Corollary \ref{cor:every irreducible rep of virtually abelian is finite dimensional}.
\end{proof}


We are ready to prove the converse direction of the character classification.
 
\begin{proof}[Proof of Theorem \ref{thm:main converse}]
Let $\I,\K \nrm \R$ with $\mfi{\K} = \I$ be a pair of ideals   as in the statement of the Theorem. Denote
$$ \mathcal{A}_d(\I,\K) = \SLtil{d}{\I} / \ker \varphi.$$
We know that the subquotient $\mathcal{A}_d(\I,\K)$ is a virtually central subgroup of the quotient $\EL{d}{\R}/ \ker\varphi$ according to Theorem \ref{thm:virtually abelian}. Consider the finite essential $\EL{d}{\R}$-orbit  $O \subset \chars{\mathcal{A}_d(\I,\K)}$ given in the statement of the Theorem. 

Let $\psi \in  \charrel{G}{\SLtil{d}{\I} \cap \EL{d}{\R}}$ be the relative character corresponding to the $\EL{d}{\R}$-orbit $O$.   Let $\varphi \in \traces{G}$ be the trace obtained by inducing $\psi$ from the normal subgroup  $\SLtil{d}{\I} \cap \EL{d}{\R}$ to the entire group $\EL{d}{\R}$ as done in Proposition \ref{prop:extension by 0 is a trace}, namely
$$ \varphi(g) = \begin{cases} \psi(g) & g \in \SLtil{d}{\levI{\varphi}}\\ 0 & \text{otherwise}\end{cases} \quad \forall g \in \EL{d}{\R}.$$

 We claim  that the induced trace $\varphi$ is indeed a character, namely that it can cannot be written as a non-trivial convex combination.
Consider the uniquely determined Borel probability measure $\mu_\varphi$ on the set $\chars{\EL{d}{\R}}$ such that  $\varphi = \int \zeta \, \mathrm{d} \mu_\varphi(\zeta)$ as provided by Choquet's theorem. The level ideal $\levI{\zeta}$ of $\mu_\varphi$-almost every character $\zeta \in \chars{\EL{d}{\R}}$ coincides with the   ideal $\mfi{\K} = \I$, see   Corollary \ref{cor:decomposition of a trace vanishing outside}. This means that $\zeta(g) = 0$ for all elements $g \notin \SLtil{d}{\I}$ and $\mu_\varphi$-almost all characters $\zeta$ by the properties of   level ideals, see Theorem \ref{thm:vanishing outside big S_varphi using normal form decomposition}. On the other hand, the restriction of  every character $\zeta \in \chars{\EL{d}{\R}}$ to the normal subgroup $\SLtil{d}{\I} \cap \EL{d}{\R}$ is a relative character. As such these restrictions must $\mu_\varphi$-almost surely coincide with the restriction   $\varphi_{|\SLtil{d}{\I} \cap \EL{d}{\R}}$ to the same subgroup.
We conclude that   $\mu_\varphi$ is an atomic probability measure supported on a single point. In other words $\varphi \in \chars{\EL{d}{\R}}$  as required.

The fact that the level ideal $\levI{\varphi}$ of the character $\varphi \in \chars{\EL{d}{\R}}$ constructed above coincides with the given ideal $\I$ was established in the course of the previous claim. Since the given orbit $O$ is essential  the kernel ideal $\kerI{\varphi}$ of the character $\varphi$ must coincide   with the given ideal $\K$. Lastly, it is clear that  the associated $\EL{d}{\R}$-orbit  $O_\varphi \subset \mathcal{A}_d(\levI{\varphi},\kerI{\varphi}) $ coincides with the given one $ O$.

It remains to show that the character $\varphi \in \chars{\EL{d}{\R}}$ satisfying the three conditions $\levI{\varphi} =\I, \kerI{\varphi} = \K$ and $\Orb{\varphi} = O$ is uniquely determined. Indeed any such character $\varphi$ satisfies $\varphi(g) = 0 $ for all elements $g \in \EL{d}{\R} \setminus \SLtil{d}{\levI{\varphi}}$, and its restriction to the normal subgroup $\SLtil{d}{\levI{\varphi}} \cap \EL{d}{\R}$ is determined by the kernel ideal $\kerI{\varphi}$ and the finite essential orbit $O_\varphi$.
%
%
%
%
\end{proof}

\bibliographystyle{alpha}
\bibliography{characters}

\end{document}